\documentclass[12pt,leqno]{article}

\usepackage{amsfonts,amsmath,amsthm,amssymb}
\usepackage{epsfig,graphicx,color,a4wide}
\usepackage{enumerate}
\usepackage{hyperref}

\def\R{\mathbb{R}}
\def\Z{\mathbb{Z}}
\def\N{\mathbb{N}}
\def\eps{\varepsilon}
\def\oG{\mathfrak{G}}
\def\oL{\mathfrak{L}}

\DeclareMathOperator{\Lip}{Lip}

\newtheorem{theo}{Theorem}[section]
\newtheorem{lem}[theo]{Lemma}
\newtheorem{pro}[theo]{Proposition}
\newtheorem{cor}[theo]{Corollary}

\theoremstyle{definition}
\newtheorem{defi}[theo]{Definition}
\theoremstyle{remark}
\newtheorem{rem}[theo]{Remark}

\newtheorem{counterex}{Counter-example}

\numberwithin{equation}{section}

\begin{document}

\title{\bf Flux-limited solutions for quasi-convex Hamilton-Jacobi equations on networks}

\author{C. Imbert\footnote{CNRS UMR 8553, D\'epartement de Math\'ematiques et
    Applications, \'Ecole Normale Sup\'erieure (Paris), 45 rue d'Ulm,
    75230 Paris cedex 5, France}~ and R. Monneau\footnote{Universit\'e
    Paris-Est, CERMICS (ENPC), 6-8 Avenue Blaise Pascal, Cit\'e
    Descartes, Champs-sur-Marne, F-77455 Marne-la-Vall\'ee cedex 2,
    France} }

\maketitle


\begin{abstract}
  We study Hamilton-Jacobi equations on networks in the case where
  Hamiltonians are quasi-convex with respect to the gradient variable
  and can be discontinuous with respect to the space variable at
  vertices.  First, we prove that imposing a general \emph{vertex
    condition} is equivalent to imposing a specific one which only
  depends on Hamiltonians and an additional free parameter, the
  \emph{flux limiter}. Second, a general method for proving comparison
  principles is introduced. This method consists in constructing a
  \emph{vertex test function} to be used in the doubling variable
  approach. With such a theory and such a method in hand, we present
  various applications, among which a very general existence and
  uniqueness result for quasi-convex Hamilton-Jacobi equations on
  networks.
\end{abstract}



\setcounter{tocdepth}{1}
\tableofcontents

\section{Introduction}

This paper is concerned with Hamilton-Jacobi (HJ) equations on
networks associated with Hamiltonians that are quasi-convex and
coercive in the gradient variable and possibly discontinuous at the
vertices of the network in the space variable. 

Space discontinuous Hamiltonians have been identified as both
important/relevant and difficult to handle; in particular, a few
theories/approaches (see below) were developed to study the associated
HJ equations. In this paper, we show that if they are assumed to be
quasi-convex and coercive in the gradient variable, then not only
uniquess can be proved for very general conditions at discontinuities
(referred to as \emph{junction conditions}), but such conditions can
even be classified: imposing a general junction condition reduce to
impose a junction condition of optimal control type, referred to as
\emph{a flux-limited junction condition.} As far as uniqueness is
concerned, a comparison principle is proved. We show that the doubling
variable approach can be adapted to the discontinuous setting if we go
beyond the classical test function $|x-y|^2/2$ by using a \emph{vertex
  test function} instead. This vertex test function can be used to do
much more, like dealing with second order terms \cite{in} or getting
error estimates for monotone schemes \cite{ik}.

We  point out that the present article is written in the
one-dimensional setting for pedagogical reasons but our theory extends
readily to higher dimensions \cite{higher}.

\subsection{The junction framework}

We focus in this introduction and in most of the article on the
simplest network, referred to as a \emph{junction}, and on
Hamiltonians which are constant with respect to the space variable on
each edge. Indeed, this simple framework leads us to the main
difficulties to be overcome and allows us to present the main
contributions.  We will see in Section~\ref{s.n} that the case of a
general network with $(t,x)$-dependent Hamiltonians is only an
extension of this special case.

A \emph{junction} is a network made of one vertex and a finite number of
infinite edges. It is endowed with a flat metric on each edge. It can
be viewed as the set of $N$ distinct copies ($N\ge 1$) of the
half-line which are glued at the origin.  For $i=1,...,N$, each branch
$J_i$ is assumed to be isometric to $[0,+\infty)$ and
\begin{equation}\label{eq::J}
J=\bigcup_{i=1,...,N} J_i \quad \quad \mbox{with}\quad J_i\cap J_j 
=\left\{0\right\} \quad \mbox{for}\quad i\not=j
\end{equation}
where the origin $0$ is called the \emph{junction point}. For points
$x,y\in J$,  $d(x,y)$ denotes the geodesic distance on $J$ defined
as
\[d(x,y)=\begin{cases}
|x-y| & \quad \mbox{if $x,y$ belong to the same branch},\\
|x|+|y| & \quad \mbox{if $x,y$ belong to different branches}.
\end{cases}\]
For a smooth real-valued function $u$ defined on $J$, $\partial_i
u(x)$ denotes the (spatial) derivative of $u$ at $x\in J_i$ and the
``gradient'' of $u$ is defined as follows,
\begin{equation}\label{eq::18}
u_x(x):=\begin{cases}
\partial_i u(x) & \quad \mbox{if} \quad x\in 
J_i^*:= J_i\setminus \{0\},\\
(\partial_1 u(0),...,\partial_N u(0)) & \quad \mbox{if} \quad x=0.
\end{cases}
\end{equation}
With such a notation in hand, we consider the following
Hamilton-Jacobi equation on the junction $J$
\begin{equation}\label{eq:hj-f}
\left\{\begin{array}{lll}
u_t + H_i(u_x)= 0  &\mbox{for}\quad t\in (0,+\infty) &\quad \mbox{and}\quad x\in J_i^*,\\
u_t + F(u_x)=0   &\mbox{for}\quad  t\in (0,+\infty) &\quad \mbox{and}\quad x=0
\end{array}\right.
\end{equation}
subject to the initial condition
\begin{equation}\label{eq::2}
u(0,x)=u_0(x) \quad \mbox{for}\quad x\in J.
\end{equation}
The second equation in \eqref{eq:hj-f} is referred to as \emph{the
  junction condition}. In general, minimal assumptions are required in
order to get a good notion of weak (\textit{i.e.} viscosity) solutions. We shed
some light on the fact that Equation~\eqref{eq:hj-f} can be thought
as a \emph{system} of Hamilton-Jacobi equations associated with $H_i$
coupled through a ``dynamical'' boundary condition involving $F$. This
point of view can be useful, see Subsection~\ref{ss.kr}. As far as
junction functions are concerned, we will construct below some special
ones (denoted by $F_A$) from the Hamiltonians $H_i$ ($i=1,...,N$) and
a real parameter $A$.

We consider the important case of Hamiltonians $H_i$ satisfying the
following structure condition: 
\begin{equation}\label{assum:H}
\text{For } i =1,\dots,N, \quad H_i \text{ continuous, quasi-convex and coercive}.
\end{equation}
We recall that $H_i$ is quasi-convex if its sub-level sets $\{ p :  H_i(p) \le \lambda \}$ are convex.
In particular, since $H_i$ is also assumed to be coercive, there exist numbers $p_i^0\in\R$ such that
\[ \left\{\begin{array}{l} 
H_i  \text{ nonincreasing in } (-\infty,p_i^0] \\
 H_i  \text{ nondecreasing in }  [p_i^0,+\infty).
\end{array}\right. \]

\subsection{First main new idea: classification of junction conditions}

In the present paper, two notions of viscosity solutions are
introduced: \emph{relaxed (viscosity) solutions} (see
Definition~\ref{defi:relaxed}), which can be used to deal with all
junction conditions, and \emph{flux-limited (viscosity) solutions}
(see Definition~\ref{defi::1}) which are associated with flux-limited
junction conditions. Relaxed solutions are used to prove existence and
ensure stability. Flux-limited solutions satisfy the junction
condition in a stronger sense and are used in order to prove
uniqueness. Our first main result states that relaxed solutions for
general junction conditions are in fact flux-limited solutions for
some junction conditions of optimal-control type.

We now introduce the notion of flux-limited junction condition.
Given a \emph{flux limiter} $A \in \R \cup \{-\infty\}$, the
$A$-limited flux through the junction point is defined for
$p=(p_1,\dots,p_N)$ as
\begin{equation}\label{eq::4}
F_A(p)=\max \left(A,\quad  \max_{i=1,...,N} H_i^-(p_i)\right)
\end{equation}
where $H_i^-$ is the nonincreasing part of $H_i$ defined by
\[H_i^-(q)=\begin{cases}
H_i(q) &\quad \mbox{if}\quad q\le p_i^0,\\
H_i(p_i^0) &\quad \mbox{if}\quad q > p_i^0.
\end{cases}\]
We now consider the following important special case of \eqref{eq:hj-f},
\begin{equation}\label{eq::1bis}
\left\{\begin{array}{lll}
u_t + H_i(u_x)= 0  &\mbox{for}\quad t\in (0,+\infty) &\quad \mbox{and}\quad x\in J_i^*,\\
u_t + F_A(u_x)=0   &\mbox{for}\quad  t\in (0,+\infty) &\quad \mbox{and}\quad x=0.
\end{array}\right.
\end{equation}

We point out that the flux functions $F_A$  associated with $A\in
[-\infty, A_0]$ coincide if one chooses 
\begin{equation}\label{eq::A_0}
A_0=\max_{i=1,...,N} \min_{\R} H_i.
\end{equation}

As announced above, general junction conditions are proved to be
equivalent to those flux-limited junction conditions. Let us be more
precise: a \emph{junction function} is an $F\colon\R^N \to \R$ satisfying
\begin{equation}\label{assum:F}
F:\R^N \to \R \text{ is continuous and non-increasing with respect to
  all variables}. 
\end{equation}
\begin{theo}[General junction conditions reduce to flux-limited ones]\label{th:fa}
  Assume that the Hamiltonians satisfy \eqref{assum:H} and that the
  junction function satisfies \eqref{assum:F}.  Then there exists $A_F \in \R$
  such that any continuous relaxed (viscosity) solution of \eqref{eq:hj-f} is in
  fact a flux-limited (viscosity) solution of \eqref{eq::1bis} with $A=A_F$.
\end{theo}
\begin{rem}
Assumption~\eqref{assum:F} is minimal, at least ``natural''; indeed,
monotonicity is related to the notion of viscosity solutions that will
be introduced. In particular, it is needed in order to construct
solutions through the Perron method \cite{I}. 
\end{rem}
\begin{rem}
  Relaxed and flux-limited solutions are respectively
  introduced in Definitions~\ref{defi:relaxed} and \ref{defi::1}.
\end{rem}
\begin{rem}
Relaxed solutions of \eqref{eq:hj-f} are assumed to be continuous in Theorem~\ref{th:fa}. 
This assumption can be weakened, see Proposition~\ref{pro:fa-gen-sup} below. 
\end{rem}

\paragraph{The special case of convex Hamiltonians.}
In the special case of convex Hamiltonians $H_i$ with different
minimum values, Problem~\eqref{eq::1bis} can be viewed as the
Hamilton-Jacobi-Bellman equation satisfied by the value function of an
optimal control problem; see for instance \cite{imz} when
$A=-\infty$. In this case, existence and uniqueness of viscosity
solutions for \eqref{eq::1bis}-\eqref{eq::2} (with $A=-\infty$) have
been established either with a very rigid method \cite{imz} based on
an explicit Oleinik-Lax formula which does not extend easily to
networks, or in cases reducing to $H_i=H_j$ for all $i,j$ if
Hamiltonians do not depend on the space variable
\cite{schieborn,acct}. In such an optimal control framework,
trajectories can stay for a while at the junction point. In this case,
the running cost at the junction point equals $-\max_i (\min H_i)$. In
this special case, the parameter $A$ consists in replacing the
previous running cost at the junction point by $\min (-A,\min_i
L_i(0))$. In Section~\ref{s:oct}, the link between our results and
optimal control theory will be investigated.

\subsection{Second main new idea: the vertex test function}

The second main contribution of this paper is to provide the reader
with a general yet handy and flexible method to prove a comparison
principle, allowing in particular to deal with Hamiltonians that are
quasi-convex and coercive with respect to the gradient variable and
are possibly discontinuous with respect to the space variable at the
vertices.

It is known that the core of the theory for HJ equations lies in the
proof of a strong uniqueness result, \textit{i.e.} of a comparison
principle.  It is also known that it is difficult to get uniqueness
results for discontinuous Hamiltonians. Indeed, the standard proof of the
comparison principle in the Euclidian setting is based on the
so-called \emph{doubling variable technique}; and such a method, even
in the monodimensional case, generally fails for piecewise
constant (in $x$) Hamiltonians at discontinuities (see the last
paragraph of Subsection~\ref{ss.kr}). Since the network setting
contains the previous one, the classical doubling variable technique
is known to fail at vertices \cite{schieborn,acct,imz}.

Before discussing the method we develop to prove it,
we state the comparison principle. 
\begin{theo}[Comparison principle on a junction]\label{th::2}
  Assume that the Hamiltonians satisfy \eqref{assum:H}, the junction
  function satisfies \eqref{assum:F} and that the initial datum $u_0$
  is uniformly continuous.  Then for all (relaxed) sub-solution $u$
  and (relaxed) super-solution $v$ of \eqref{eq::1bis}-\eqref{eq::2}
  satisfying for some $T>0$ and $C_T>0$,
\begin{equation}\label{eq::27}
u(t,x)\le C_T (1+ d(0,x)),\quad v(t,x)\ge -C_T(1+d(0,x)),\quad
\text{for all}\quad (t,x) \in [0,T)\times J,
\end{equation}
we have
\[u\le v \quad \text{in}\quad [0,T)\times J.\]
\end{theo}
Combining Theorems~\ref{th:fa} and \ref{th::2}, we get the following one. 
\begin{theo}[Existence and uniqueness on a junction]\label{th::1}
  Assume that the Hamiltonians satisfy \eqref{assum:H}, that $F$
  satisfies \eqref{assum:F} and that the initial datum $u_0$ is
  uniformly continuous.  Then there exists a unique continuous
  (relaxed) viscosity solution $u$ of \eqref{eq:hj-f}, \eqref{eq::2}
  such that for every $T>0$, there exists a constant $C_T>0$ such that
\[|u(t,x)-u_0(x)|\le C_T \quad \text{for all}\quad (t,x)\in [0,T)\times J.\]
\end{theo}
As we previously mentioned it, we prove Theorem~\ref{th::2} by
remarking that the doubling variable approach can still be used if a
suitable \emph{vertex test function} $G$ at each vertex is
introduced. Roughly speaking, such a test function will allow the
edges of the network to exchange the necessary information. More
precisely, the usual penalization term, $\frac{(x-y)^2}{\eps}$ with
$\eps>0$, is replaced with $\eps G\left(\eps^{-1}x,\eps^{-1}y\right)$.
For a general HJ equation
\[u_t+H(x,u_x)=0,\]
the vertex test function has to (almost) satisfy,
\[H(y,-G_y(x,y))-H(x,G_x(x,y)) \le 0\] (at least close to the vertex
$x=0$).  This key inequality compensates for the lack of compatibility
between Hamiltonians\footnote{Compatibility conditions are assumed in
  \cite{schieborn,acct} for instance.}.  The construction of a
(vertex) test function satisfying such a condition allows us to
circumvent the discontinuity of $H(x,p)$ at the junction point.

As explained above, this method consists in combining the doubling
variable technique with the construction of a vertex test function
$G$.  We took our inspiration for the construction of this function
from papers like \cite{gnpt,akr} dealing with scalar conservation laws
with discontinuous flux functions. In such papers, authors stick to
the case $N=2$.

A natural family of explicit solutions of \eqref{eq::1bis} is given by
\[u(t,x)=p_i x - \lambda t \quad \mbox{if}\quad x\in J_i\]
for $(p,\lambda)$ in the \emph{germ} ${\mathcal G}_A$  
defined as follows,
\begin{equation}\label{eq::50}
{\mathcal G}_A = 
\begin{cases}
\left\{(p,\lambda)\in\R^N\times \R,
\quad H_i(p_i)=F_A(p)=\lambda \quad \mbox{for all}\quad
i=1,...,N\right\} & \text{ if } N \ge 2, \\
 \left\{(p_1,\lambda)\in\R\times \R,\quad H_1(p_1)=\lambda\ge
 A\right\}
& \text{ if } N=1.
\end{cases}
\end{equation}
In the special case of convex Hamiltonians satisfying $H_i''>0$ the
vertex test function $G$ is a regularized version\footnote{Such a
  function should indeed be regularized since it is not $C^1$ on the
  diagonal $\{ x = y \}$ of $J^2$.}  of the function $A+G^0$, where $G^0$ is defined as
follows: for $(x,y)\in J_i\times J_j$,
\begin{equation}\label{eq:defi-g0}
G^0(x,y)=\sup_{(p,\lambda)\in {\mathcal G}_A}\left(p_i x -p_j y -
\lambda\right).
\end{equation}
In particular, we have $A+G^0(x,x)=0$.

\subsection{The network setting}

We will extend our results to the case of networks and quasi-convex
Hamiltonians depending on time and space and to flux limiters $A$
depending on time and vertex, see Section~\ref{s.n}.  Noticeably, a
localization procedure allows us to use the vertex test function
constructed for a single junction.

In order to state the results in the network setting, we need to make
precise the assumptions satisfied by the Hamiltonians associated with
each edge and the flux limiters associated with each vertex.
This results in a rather long list of assumptions. Still, when reading
the proof of the comparison principle in this setting, the reader may
check that the main structure properties used in the proof are
gathered in the technical Lemma~\ref{lem::l146}.

As an application of the comparison principle, we consider a
  model case for homogenization on a network. The network
  $\mathcal{N}_\eps$ whose vertices are $\eps \Z^d$ is naturally
  embedded in $\R^d$. We consider for all edges a Hamiltonian only
  depending on the gradient variable but which is ``repeated $\eps
  \Z^d$-periodically with respect to edges''. We prove that when $\eps
  \to 0$, the solution of the ``oscillating'' Hamilton-Jacobi equation
  posed in $\mathcal{N}_\eps$ converges toward the unique solution of
  an ``effective'' Hamilton-Jacobi equation posed in $\R^d$.

\paragraph{A first general comment about the main results.} 
Our proofs do not rely on optimal control interpretation (there is no
representation formula of solutions for instance) but on PDE methods.
We believe that the construction of a vertex test function is flexible
and opens many perspectives. It also sheds light on the fact that the
framework of quasi-convex Hamiltonians, which is slightly more general
than the one of convex ones (at least in the evolution case), deserves
special attention.

\subsection{Comparison with known results}
\label{ss.kr}

\paragraph{Hamilton-Jacobi equations on networks.} 
There is a growing interest in the study of Hamilton-Jacobi equations
on networks. The first results were obtained in \cite{schieborn} for
eikonal equations. Several years after this first contribution, the
three papers \cite{acct,imz,sc13} were published more or less
simultaneously. In these three papers, the Hamiltonians are always
\emph{convex} with respect to the gradient variables and optimal
control plays in important role (at least in \cite{acct,imz}). Still,
frameworks are significantly different.

Recently, a general approach of eikonal equations in metric
spaces has been proposed in \cite{ghn,af,gs} (see also \cite{nakayasu}).  

In \cite{acct}, the authors study an optimal control problem in $\R^2$
and impose a \emph{state constraint}: the trajectories of the
controlled system have to stay in the embedded network. From this
point of view, \cite{acct} is related to \cite{frp98,frp00} where
trajectories in $\R^N$ are constrained to stay in a closed set $K$
which can have an empty interior.  But as pointed out in \cite{acct},
the framework from \cite{frp98,frp00} implies some restricting
conditions on the geometry of the embedded networks. Our approach can
now handle the general case for networks. 

Our approach is also used to reformulate ``state
constraint'' solutions by Ishii and Koike \cite{ik} (see
Proposition~\ref{pro::gr7}).

The reader is referred to \cite{cm} where the different notions of
viscosity solutions used in \cite{acct,imz,sc13} are compared; in the
few cases where frameworks coincide, they are proved to be equivalent. 

In \cite{imz}, the comparison principle was a consequence of a
super-optimality principle (in the spirit of \cite{ls85} or
\cite{soravia99a,soravia99b}) and the comparison of sub-solutions with
the value function of the optimal control problem. Still, the idea of
using the ``fundamental solution'' $\mathcal{D}$ to prove a comparison
principle originates in the proof of the comparison of sub-solutions
and the value function. Moreover, as explained in
Subsection~\ref{ss.godo}, the comparison principle obtained in this
paper could also be proved, for $A =-\infty$ and under more restrictive
assumptions on the Hamiltonians, by using this fundamental solution.

The reader is referred to \cite{acct,imz,sc13} for further references
about Hamilton-Jacobi equations on networks.

\paragraph{Networks, regional optimal control and stratified spaces.} 
We already pointed out that the Hamilton-Jacobi equation on a network
can be regarded as a system of Hamilton-Jacobi equations coupled
through vertices. In this perspective, our work can be compared with
studies of Hamilton-Jacobi equations posed on, say, two domains
separated by a frontier where some \emph{transmission conditions}
should be imposed.  Contributions to such problems are
\cite{bbc,bbc2,rz,rsz,aot}.  This can be even more general by
considering equations in stratified spaces \cite{bh07,bc15}.

We first point out that the framework of these works is genuinely
multi-dimensional while in this paper we stick to a monodimensional
setting; still, our method generalizes to a higher dimensional setting
\cite{higher}. Another difference between their approach and the one
presented in the present work and in papers like \cite{acct,sc13,imz}
is that these authors write a Hamilton-Jacobi equation on the frontier
(which is lower-dimensional). Another difference is that techniques
from dynamical systems play also an important role. We mention that
the techniques from \cite{aot} can be applied to treat the cases
considered in our work.

Still, results can be compared. Precisely, considering a framework
were both results can be applied, that is to say the monodimensional
one, we will prove in Section~\ref{sec:bbc} that the value function
$U^-$ from \cite{bbc2} coincides with the solution of
\eqref{eq::1bis} for some constant $A$ that is determined. 
And we prove more (in the monodimensional setting; see also extensions
below): we prove that the value function $U^+$ from
\cite{bbc2} coincides with the solution of \eqref{eq::1bis} for some
(distinct) constant $A$ which is also computed. 

\paragraph{Hamilton-Jacobi equations with discontinuous source terms.}
There are numerous papers about Hamilton-Jacobi equations with
discontinuous Hamiltonians. The first contribution is due to Dupuis
\cite{dupuis}; see also \cite{de,gso,cr,dzs}.  The recent paper
\cite{gh13} considers a Hamilton-Jacobi equation where specific
solutions are expected. In the one-dimensional space, it can be proved
that these solutions are in fact flux-limited solutions in the sense
of the present paper with $A=c$ where $c$ is a constant appearing in
the HJ equation at stake in \cite{gh13}. The introduction of
\cite{gh13} contains a rather long list of results for HJ equations
with discontinuous Hamiltonians; the reader is referred to it for
further details.

\paragraph{Contributions of the paper.} In light of the review we made
above, we can emphasize the main contributions of the paper: compared
to \cite{schieborn,sc13}, we deal not only with eikonal equations but
with general Hamilton-Jacobi equations. In contrast to \cite{acct}, we
are able to deal with networks with infinite number of edges, that are
not embedded. In constrast to \cite{acct,imz,schieborn,sc13}, we can
deal with quasi-convex (but not necessarily convex) discontinuous
Hamilton-Jacobi equations with general junctions conditions. For such
equations, flux-limited solutions are introduced and a flexible PDE
framework is developed instead of an optimal control
approach. Eventhough, the link with optimal control (in the spirit of
\cite{acct,bbc,bbc2}) and with regional control (in the spirit of
\cite{bbc,bbc2}) are thoroughly investigated. In particular, a PDE
characterization of the two value functions introduced in \cite{bbc2}
is provided, one of the two characterizations being new.

Several applications are also developed: the extension to the network
setting and some homogenization results. 

\paragraph{Perspectives.}

More homogenization results were recently obtained  in \cite{gim}. An
example of applications of this result is the case where a periodic
Hamiltonian $H(x,p)$ is perturbed by a compactly supported function of
the space variable $f(x)$, say. Such a situation is considered in
lectures by Lions at Coll\`ege de France \cite{lions-cdf}. Rescaling
the solution, the expected effective Hamilton-Jacobi equation is
supplemented with a junction condition which keeps memory of the
compact perturbation.

We would also like to mention that our results extend to a higher
dimensional setting (in the spirit of \cite{bbc,bbc2}) for
quasi-convex Hamiltonians \cite{higher}.

\subsection{Organization of the article and notation}

\paragraph{Organization of the article.} 
The paper is organized as
follows. In Section~\ref{s.v}, we introduce the notion of viscosity
solution for Hamilton-Jacobi equations on junctions, we prove that
they are stable (Proposition~\ref{pro::2}) and we give an existence
result (Theorem~\ref{th::3}). In Section~\ref{s.c}, we prove the
comparison principle in the junction case (Theorem~\ref{th::2}). In
Section~\ref{s.G}, we construct the vertex test function
(Theorem~\ref{th::G}). In Section~\ref{s:oct}, a general optimal
control problem on a junction is considered and the associated value
function is proved to be a solution of \eqref{eq::1bis} for some
computable constant $A$. 
In Section~\ref{sec:bbc}, the two value
functions introduced in  \cite{bbc2} are shown to be solutions of
\eqref{eq::1bis}  for two explicit (and distinct) constants $A$.
In Section~\ref{s.n}, we explain how to
generalize the previous results (viscosity solutions, HJ equations,
existence, comparison principle) to the case of networks. In
Section~\ref{s.h}, we present a straightforward application of our
results by proving a homogenization result passing from an
``oscillating'' Hamilton-Jacobi equation posed in a network embedded
in an Euclidian space to a Hamilton-Jacobi equation in the whole
space.  Finally, we prove several technical results in
Appendix~\ref{s.a} and we state results for stationary Hamilton-Jacobi
equations in Appendix~\ref{s.stat}.

\paragraph{Notation for a junction.} 
A junction is denoted by $J$. It is made of a finite number of edges
and a junction point. The $N$ edges of a junction, $J_1,\dots,J_N$
($N \in \N \setminus\{0\}$) are isometric to $[0,+\infty)$. The open
edge is denoted by $J_i^* = J_i \setminus \{0\}$.  Given a final time
$T>0$, $J_T$ denotes $(0,T) \times J$.

The Hamiltonians on the branches $J_i$ of the junction are denoted by
$H_i$; they only depend on the gradient variable.  The Hamiltonian at
the junction point is denoted by $F_A$ and is defined from all $H_i$
and a constant $A$ which ``limits'' the flux of information at the
junction.

Given a function $u:J \to \R$, its gradient at $x$ is denoted by
$u_x$; it is a real number if $x \neq 0$ but it is a vector of $\R^N$
at $x=0$. We let $|u_x|$ denote $|\partial_i u|$ outside the junction point
and $\max_{i=1,...,N} |\partial_i u|$ at the junction point.  If now
$u(t,x)$ also depends on the time $t\in (0,+\infty)$, $u_t$ denotes
the time derivative.

\paragraph{Notation for networks.} 
A network is denoted by $\mathcal{N}$.  It is made of vertices $n \in
\mathcal{V}$ and edges $e \in \mathcal{E}$. Each edge is either
isometric to $[0,+\infty)$ or to a compact interval whose length is
  bounded from below; hence a network
  is naturally endowed with a metric.  The associated open
  (resp. closed) balls are denoted by $B (x,r)$ (resp. $\bar B (x,r)$)
  for $x \in \mathcal{N}$ and $r>0$.

In the network case, an Hamiltonian is associated with each edge $e$ and
is denoted by $H_e$. It depends on time and space; moreover, the
limited flux functions $A$ can depend on time $t$ and the vertex $n$: 
$A_n (t)$.

\paragraph{Further notation.}
Given a metric space $E$, $C(E)$ denotes the space of continuous
real-valued functions defined in $E$. A modulus of continuity is a
function $\omega: [0,+\infty) \to [0,+\infty)$ which is non-increasing
    and $\omega (0+)=0$.

\section{Relaxed and flux-limited  solutions}
\label{s.v}

This section starts with the introduction of two notions of
viscosity solutions in the junction case and of their studies. Relaxed
(viscosity) solutions are first introduced; they are defined for
general junction conditions. They naturally satisfy good stability
properties (see for instance Proposition~\ref{pro::2}).  Flux-limited
solutions are associated with flux-limited junction conditions.  They
satisfy the junction condition in a stronger sense (see
Proposition~\ref{pro::1}). The main contribution of this section is
the proof of Theorem~\ref{th:fa}. It relies on the observation that
the set of test functions for flux-limited solutions can be reduced
drastically: it is enough to consider test functions with fixed space
slopes (Theorem~\ref{th::gr1}). 

\subsection{Definitions}

In order to introduce the two notions of viscosity solution which will
be used in the remaining of the paper, we first introduce the class of
test functions. For $T>0$, set $J_T= (0,T)\times J$. We define the
class of test functions on $(0,T)\times J$ by
\[C^1(J_T)=\left\{\varphi\in C(J_T),\; \text{the restriction of
    $\varphi$ to $(0,T)\times J_i$ is $C^1$ for
    $i=1,...,N$}\right\}.\]

We (classically) say that a test function $\phi$ touches a function
$u$ from below (respectively from above) at $(t,x)$ if $u-\phi$
reaches a minimum (respectively maximum) at $(t,x)$ in a neighborhood
of it.

We recall
the definition of upper and lower semi-continuous envelopes $u^*$ and
$u_*$ of a (locally bounded) function $u$ defined on $[0,T)\times J$,
\[u^*(t,x)=\limsup_{(s,y)\to (t,x)} u(s,y)
\qquad \text{and}\qquad u_*(t,x)=\liminf_{(s,y)\to (t,x)} u(s,y).\] 

\begin{defi}[Relaxed  solutions]\label{defi:relaxed}
Assume that the Hamiltonians satisfy \eqref{assum:H} and that $F$
satisfies \eqref{assum:F} and let $u:[0,T)\times J\to \R$.
\begin{enumerate}[i)]
\item We say that $u$ is a \emph{relaxed sub-solution}
  (resp. \emph{relaxed super-solution}) of \eqref{eq:hj-f} in
  $(0,T)\times J$ if for all test function $\varphi\in C^1(J_T)$
  touching $u^*$ from above (resp. from below) at
  $(t_0,x_0) \in J_T$, we have
\[\varphi_t + H_i(\varphi_x) \le 0  \quad (\text{resp.}\quad \ge 0) 
\quad \text{at } (t_0,x_0)\]
 if $x_0 \neq 0$, and 
\[\left.\begin{array}{lll}
\text{either } & 
\varphi_t + F(\varphi_x) \le 0  &\quad (\text{resp.}\quad \ge 0)  \\
\text{or } &
\varphi_t + H_i(\partial_i \varphi) \le 0 & \quad (\text{resp.}\quad
\ge 0) \quad \text{ for some } i
\end{array} \right| 
\quad \text{at } (t_0,x_0)\]
 if $x_0 = 0$.
\item We say that $u$ is a \emph{relaxed sub-solution}
  (resp. \emph{relaxed super-solution}) of \eqref{eq:hj-f},
  \eqref{eq::2} on $[0,T)\times J$ if additionally
\[u^*(0,x) \le u_0(x) \quad (\text{resp.}\quad u_*(0,x) \ge u_0(x))
\quad \text{for all}\quad x\in J.\]
\item  We say that $u$ is a \emph{relaxed  solution} if $u$ is
  both a relaxed sub-solution and a relaxed super-solution.
\end{enumerate}
\end{defi}
We give a second definition of viscosity solutions in the case of
flux-limited junction functions $F_A$: the junction condition is
satisfied ``in a classical sense'' for test functions touching sub-
and super-solutions at the junction point.
\begin{defi}[Flux-limited  solutions]\label{defi::1}
Assume that the Hamiltonians satisfy \eqref{assum:H} and let
$u:[0,T)\times J\to \R$.
\begin{enumerate}[i)]
\item We say that $u$ is a \emph{flux-limited sub-solution}
  (resp. \emph{flux-limited super-solution}) of \eqref{eq::1bis} in
  $(0,T)\times J$ if for all test function $\varphi\in C^1(J_T)$
  touching $u^*$ from above (resp. from below) at $(t_0,x_0) \in J_T$,
  we have
\begin{eqnarray}
\nonumber 
\varphi_t + H_i(\varphi_x) \le 0 & \quad (\text{resp.}\quad \ge 0) 
\quad \text{at $(t_0,x_0)$ \qquad if $x_0\in J_i^*$}\\
\label{eq::10}
\varphi_t + F(\varphi_x) \le 0 & \quad (\text{resp.}\quad \ge 0) \quad
\text{at $(t_0,x_0)$ \qquad if $x_0 = 0$}.
\end{eqnarray}
\item We say that $u$ is a \emph{flux-limited sub-solution}
  (resp. \emph{flux-limited super-solution}) of \eqref{eq::1bis},
  \eqref{eq::2} on $[0,T)\times J$ if additionally
\[u^*(0,x) \le u_0(x) \quad (\text{resp.}\quad u_*(0,x) \ge u_0(x))
\quad \text{for all}\quad x\in J.\]
\item  We say that $u$ is a \emph{flux-limited solution} if $u$ is
  both a flux-limited sub-solution and a flux-limited super-solution.
\end{enumerate}
\end{defi}

\subsection{The ``weak continuity'' condition for sub-solutions}

If $F$ not only satisfies \eqref{assum:F} but is also \emph{semi-coercive}, 
that is to say if 
\begin{equation}\label{assum:semi-coercive}
F(p) \to +\infty \quad \text{ as } \quad \max_i \ (\max(0, -p_i)) \to +\infty 
\end{equation}
then any $F$-relaxed sub-solution satisfies a ``weak continuity'' condition at the
junction point. Precisely, the following lemma holds true.
\begin{lem}[``Weak continuity'' condition at the junction point]\label{lem:weak-cont}
Assume that the Hamiltonians satisfy \eqref{assum:H} and that $F$ satisfies
\eqref{assum:F} and \eqref{assum:semi-coercive}. Then any relaxed sub-solution 
$u$ of \eqref{eq:hj-f} satisfies for all $t \in (0,T)$ and all $i \in \{1,\dots,N\}$,
\[ u(t,0) = \limsup_{(s,y) \to (t,0), y \in J_i^*} u (s,y).\]
\end{lem}
\begin{proof}
Since $u$ is upper semi-continuous, we know that for all $t\in (0,T)$ and $i$, 
\[ u(t,0) \ge \limsup_{(s,y) \to (t,0), y \in J_i^*} u (s,y).\]
Assume that there exists $t^*$ and $i_0$ such that 
\[u(t^*,0) > \limsup_{(s,y) \to (t^*,0), y \in J_{i_0}^*} u (s,y).\]

Since $u$ is upper semi-continuous, we know that we can find
$t_0$ arbitrarily close to $t^*$ such that $u(t_0,0)$ is
arbitrarily close to $u (t^*,0)$ and such that there exists a
$C^1$ function $\Psi(t)$ (strictly) touching  $u(t,0)$ from above at
$t_0$. In particular, we can ensure 
\begin{equation}\label{eq::rtrac-1}
 u (t_0,0) > \limsup_{(s,y) \to (t_0,0), y \in J_{i_0}^*} u (s,y)
\end{equation}
and 
\[ 
\begin{cases}
u (t,0) & < \Psi (t) \quad \text{ for } t \in [t_0-r_0,t_0+r_0] \setminus \{t_0\} \\
u(t_0,0) &= \Psi (t_0) .
\end{cases}
\] 
In particular, since $(\Psi-u)(t_0\pm r_0,0) >0$, 
there exist $\delta_1>0$ and $r_1>0$ small enough such that 
\begin{equation}\label{eq:control-bd}
 u (t_0\pm r_0,x)+\delta_1 \le \Psi(t_0\pm r_0) \quad \text{ for } x \in B(0,r_1) \subset J.
\end{equation}

We now consider the test function $\phi (t,x) = \Psi (t) + p_i x$ for
$x \in J_i$.  We claim that for $i\not= i_0$ and for $p_i=p_i(r_1)$ large enough, $u - \phi$ reaches
its maximum $M_i$ on
$Q_0 = [t_0-r_0,t_0+r_0] \times [0,r_1] \subset (0,T) \times J_i$ at $(t_0,0)$.
We first remark that $M_i \ge u(t_0,0) - \Psi (t_0) =0$. Moreover, for
$(t_0\pm r_0,x)$ and $x \in [0,r_1]$, \eqref{eq:control-bd} implies that
\[ u(t_0\pm r_0,x) - \Psi (t_0\pm r_0) - p_i x \le -\delta_1 <M_i. \]
For $(t,x) \in Q_0$ and $x=r_1$, we have for $p_i$ large enough
\[ u(t,x) - \Psi (t) - p_i x \le \|u^+\|_{L^\infty (Q_0)} + \| \Psi \|_{L^\infty ([t_0-r_0,t_0+r_0])} 
- p_i r_1 < M_i. \]
Hence the supremum is reached either for $x=0$ or $x$ in the interior of $Q_0$. 
In the latter case, this yields the viscosity inequality
\[ \Psi'(t) + H_i (p_i) \le 0\]
which cannot hold true for large $p_i$. We conclude that 
\[
\begin{cases}
 u (t,x) &< \Psi (t) + p_i x \quad \text{ in } Q_0 \setminus \{ (t_0,0) \} \\
 u(t_0,0) &= \Psi (t_0) .
\end{cases}
\]
We now  get 
\[
\begin{cases}
 u (t,x)  < \Psi (t) + p_i x  &  \text{ in } [t_0-r_0,t_0+r_0] \times [0,r_1]
 \setminus \{ (t_0,0) \}  \text{  with $p_i>0$  if } i \neq i_0\\
 u (t,x)  < \Psi (t) + p_{i_0} x  & \text{ in } [t_0-r_0,t_0+r_0] \times [0,r_1]
 \setminus \{ (t_0,0) \}  \text{  with $p_{i_0}<0$ if } i = i_0\\
 u(t_0,0) = \Psi (t_0) .&
\end{cases}
\]
where we have used \eqref{eq::rtrac-1}
for any negative $p_{i_0}$ and any small enough $r_1=r_1(p_{i_0})$. 
This implies that 
\[ \Psi'(t_0) + F (p_1,\dots,p_{i_0}, \dots,p_N) \le 0 \]
which cannot hold true for $p_{i_0}$ very negative because of \eqref{assum:semi-coercive}. 
The proof is now complete. 
\end{proof}

\subsection{General junction conditions and stability}

The first stability result is concerned with the supremum of relaxed
sub-solutions.  Such a result is used in the Perron process to
construct relaxed solutions. Its proof is standard so we skip it.
\begin{pro}[Stability by supremum/infimum]\label{pro::2}
Assume that the Hamiltonians $H_i$ satisfy \eqref{assum:H} and that
$F$ satisfies  \eqref{assum:F}. Let ${\mathcal
  A}$ be a nonempty set and let $(u_a)_{a\in {\mathcal A}}$ be a
familly of relaxed sub-solutions (resp. relaxed super-solutions) of
\eqref{eq:hj-f} on $(0,T)\times J$.  Let us assume that
\[u=\sup_{a\in {\mathcal A}} u_a \quad (\text{resp.}\quad u=\inf_{a\in {\mathcal A}} u_a)\]
is locally bounded on $(0,T)\times J$. Then $u$ is a relaxed
sub-solution (resp. relaxed super-solution) of \eqref{eq:hj-f} on
$(0,T)\times J$.
\end{pro}
In the following proposition, we assert that, for the special junction
functions $F_A$, the junction condition is in fact always satisfied
\emph{in the classical (viscosity) sense}, that is to say in the sense of
Definition~\ref{defi::1} (and not Definition~\ref{defi:relaxed}).
\begin{pro}[flux-limited junction conditions are satisfied in the
    classical sense]
\label{pro::1}
As\-sume that the Hamiltonians satisfy \eqref{assum:H} and consider $A
\in \R$. If $F=F_A$, then relaxed  super-solutions
(resp. relaxed  sub-solutions) coincide with flux-limited
super-solutions (resp. flux-limited sub-solutions).
\end{pro}
\begin{proof}[Proof of Proposition~\ref{pro::1}]
The proof was done in \cite{imz} for the case $A=-\infty$, using the
monotonicities of the $H_i$. We follow the same proof and omit
details. \bigskip

\textsc{The super-solution case.} 
Let $u$ be a relaxed super-solution satisfying the junction condition in the
viscosity sense and let us assume by contradiction that there exists a
test function $\varphi$ touching $u$ from below at $P_0=(t_0,0)$ for
some $t_0\in (0,T)$, such that
\begin{equation}\label{eq::11}
\varphi_t + F_A(\varphi_x) <0 \quad \text{at}\quad P_0.
\end{equation}
Then we can construct a test function $\tilde{\varphi}$ satisfying
$\tilde{\varphi}\le \varphi$ in a neighborhood of $P_0$, with equality
at $P_0$ such that
\[\tilde{\varphi}_t(P_0) = \varphi_t(P_0) \quad \text{and}\quad 
\partial_i \tilde{\varphi}(P_0) = \min(p^0_i, \partial_i \varphi(P_0))
\quad \text{for}\quad i=1,...,N.\] Using the fact that
$F_A(\varphi_x)=F_A(\tilde{\varphi}_x)\ge H_i^-(\partial_i
\tilde{\varphi}) = H_i(\partial_i \tilde{\varphi})$ at $P_0$, we
deduce a contradiction with \eqref{eq::11} using the viscosity
inequality satisfied by $\varphi$ for some $i \in \{1,\dots,N\}$.
\medskip

\textsc{The sub-solution case.}
Let now $u$ be a sub-solution satisfying the junction condition in the
viscosity sense and let us assume by contradiction that there exists a
test function $\varphi$ touching $u$ from above at $P_0=(t_0,0)$ for
some $t_0\in (0,T)$, such that
\begin{equation}\label{eq::12}
\varphi_t + F_A(\varphi_x) >0 \quad \text{at}\quad P_0.
\end{equation}
Let us define
\[I=\left\{i\in \left\{1,...,N\right\},\quad H_i^-(\partial_i \varphi) <
F_A(\varphi_x) \quad \text{at}\quad P_0\right\}\]
and for $i\in I$, let $q_i\ge p_i^0$ be such that
\[H_i(q_i)=F_A(\varphi_x(P_0))\]
where we have used the fact that $H_i(+\infty)=+\infty$.  Then we can
construct a test function $\tilde{\varphi}$ satisfying
$\tilde{\varphi}\ge \varphi$ in a neighborhood of $P_0$, with equality
at $P_0$, such that
\[\tilde{\varphi}_t(P_0) = \varphi_t(P_0)\quad \text{and}\quad 
\partial_i \tilde{\varphi}(P_0) = \left\{\begin{array}{ll}
\max(q_i, \partial_i \varphi(P_0)) &\quad \text{if}\quad i\in I,\\
\partial_i \varphi(P_0) &\quad \text{if}\quad i\not\in I.
\end{array}\right.\]
Using the fact that $F_A(\varphi_x)=F_A(\tilde{\varphi}_x)\le
H_i(\partial_i \tilde{\varphi})$ at $P_0$, we deduce a contradiction
with \eqref{eq::12} using the  viscosity
inequality for $\varphi$ for some $i \in \{1,\dots,N\}$.
\end{proof}
The last stability result is concerned with sub-solutions of 
the Hamilton-Jacobi equation away from the junction point and
which satisfy the ``weak continuity'' condition. The following 
proposition asserts that such a ``weak continuity'' is stable
under upper semi-limit. 
\begin{pro}[Stability of the ``weak continuity'' condition]\label{prop:stab-weak}
Consider a family of Hamiltonians $H^\eps$ satisfying \eqref{assum:H}. We also assume
that the coercivity of the Hamiltonians is uniform in $\eps$. Let $u^\eps$ be a family
of subsolutions of 
\[ u_t + H_i^\varepsilon (u_x)= 0 \quad \text{ in } (0,T) \times J_i^* \]
for all $i=1,\dots, N$ such that, for all $i$,
\begin{equation}\label{eq:wc-eps}
 u^\eps (t,0) = \limsup_{(s,y) \to (t,0), y \in J_i^*} u^\eps (s,y).
\end{equation}
If the upper semi-limit $\bar u = \limsup^* u^\eps$ is everywhere finite,
then it satisfies for all $i$
\[ \bar u (t,0) = \limsup_{(s,y) \to (t,0), y \in J_i^*} \bar u (s,y).\]
\end{pro}
\begin{proof}
We argue by contradiction by assuming that there exists $i_0$ and $t^* \in (0,T)$ such that 
\[ \bar u (t^*,0) > \limsup_{(s,y) \to (t_0,0), y \in J_{i_0}^*} \bar
u (s,y).\]

Our goal is first to use a perturbation argument to get a test
function $\Psi (t)$ touching strictly $\bar u$ from above at a time
$t_0$ where the previous inequality still hold true. Using the upper
semi-continuity of $\bar u$, we can keep $\bar u$ away from $\Psi (t)$
in a neighborhood of the point corresponding to the boundary of the
time interval where $\bar u$ and $\Psi$ are strictly separed. From the
definition of $\bar u$, we also get a sequence of points
$(t_\eps,x_\eps)$ realizing the value $\bar u (t_0,0)$. Considering
now $\Psi (t) + px$ for $p$ positive and very large, we use the
sequence $(t_\eps,x_\eps)$ in order to get a contact point of $u^\eps$
with this test-function away from $x=0$. This will lead to the desired
contradiction since $p$ is arbitrarily large. 
\medskip

We now make precise how to use the previous strategy. 
Since $\bar u$ is upper semi-continuous, we know that we can find
$t_0$ arbitrarily close to $t^*$ such that $\bar u(t_0,0)$ is
arbitrarily close to $\bar u (t^*,0)$ and such that there exists a
$C^1$ function $\psi(t)$ (strictly) touching  $\bar u(t,0)$ from above at
$t_0$. In particular, we can ensure 
\begin{equation}\label{eq:contra}
 \bar u (t_0,0) > \limsup_{(s,y) \to (t_0,0), y \in J_{i_0}^*} \bar u (s,y)
\end{equation}
and 
\[ 
\begin{cases}
\bar u (t,0) & < \Psi (t) \quad \text{ for } t \in [t_0-r_0,t_0+r_0] \setminus \{t_0\} \\
\bar u(t_0,0) &= \Psi (t_0) .
\end{cases}
\] 
In particular, since $(\Psi-\bar u)(t_0\pm r_0,0) >0$, 
there exist $\delta_1>0$ and $r_1>0$ such that 
\[ \bar u (t_0\pm r_0,x)+2\delta_1 \le \Psi(t_0\pm r_0) \quad \text{ for } x \in B(x_0,r_1) \subset J.\]
Since $\bar u$ is the upper relaxed-limit of $u^\eps$, this implies in particular that for $\eps$ small
enough, 
\begin{equation}\label{eq:bd1}
 u^\eps (t_0\pm r_0,x) + \delta_1 \le \Psi (t_0\pm r_0) \quad \text{ for } x \in B(x_0,r_1) \subset J.
\end{equation}
We claim that 
\[ \Psi (t_0,0) = \bar u(t_0,0) > \limsup_{\eps \to 0,s \to t_0} u^\eps (s,0).\]
Indeed, if the previous inequality is replaced with an equality, this
would contradict \eqref{eq:wc-eps}. In particular, reducing $r_0$ and
$\delta_0$ if necessary, we can further assume that for $\eps \in ]0,\eps_0[$,
\begin{equation}\label{eq:strict2}
 \forall t \in [t_0-r_0,t_0+r_0] \setminus \{t_0\}, 
 \quad  u^\eps (t,0) + \delta_0 \le \Psi(t_0). 
\end{equation}

Let $(t_\eps,x_\eps) \to (t_0,0)$ be such that 
\[ \bar u (t_0,0) = \lim_{\eps \to 0} u^\eps (t_\eps,x_\eps).\]
By \eqref{eq:strict2}, we know that $x_\eps \neq 0$ for $\eps$ small
enough.  We also know that there exists $j_0$ such that
$x_\eps \in J_{j_0}^*$ for $\eps$ small enough (along a subsequence)
with $j_0 \neq i_0$. Indeed, if $x_\eps \in J_{i_0}^*$ (at least along
a subsequence), then
\[ \bar u (t_0,0) = \lim u^\eps (t_\eps,x_\eps) \le \limsup \bar u (t_\eps,x_\eps)
\le  \limsup_{(s,y) \to (t_0,0), y \in J_{i_0}^*} \bar u (s,y)\]
which is in contradiction with \eqref{eq:contra}.
\bigskip

We now consider $\Psi(t) + px$ with $p>0$ and we consider 
the point $(s^\eps,y^\eps)$ where the maximum of $u^\eps - \Psi (t) - px$ is
reached in $Q_0= [t_0-r_0,t_0+r_0] \times [0,r_1] \subset (0,T) \times J_{j_0}$.  
Remark that for $x=0$ and $(t,x) \in Q_0$, \eqref{eq:strict2} implies that 
\[ u^\eps (t,0) - \Psi (t) \le - \delta_0<0.\]
Analogously, for $t=t_0\pm r_0$ and $(t,x) \in Q_0$, \eqref{eq:bd1} implies 
that 
\[ u^\eps (t_0\pm r_0,x) - \Psi (t_0 \pm r_0) - px \le -\delta_1 <0.\]
Finally, for $x=r_1$ and $(t,x) \in Q_0$, we have for $\eps$ small and some $\delta_2 >0$, 
\[ u^\eps (t,0) - \Psi (t) - p r_1 \le \bar u(t,0) + \delta_2 + \|
\Psi \|_\infty - p r_1.\]
Since $\bar u$ is locally bounded from above (because it is upper
semi-continuous), we conclude that we can choose $p$ large (depending
on $\delta_2 + \|\Psi\|_\infty$ and a local bound of $\bar u$ from
above) such that for $x=r_1$ and $(t,x) \in Q_0$, we have for $\eps$
small and some $\delta_2 >0$,
\[ u^\eps (t,0) - \Psi (t) - p r_1 \le  - \delta_1.\]
Finally, the maximum $M^\eps$ of  $u^\eps - \Psi (t) - px$ in $Q_0$ satisfies 
\[ M^\eps  \le u^\eps (t_\eps,x_\eps) - \Psi (t_\eps) - p x_\eps \to \bar u (t_0,0) - \Psi (t_0) =0.\]
We conclude that $(s^\eps, y^\eps)$ belongs to the interior of $Q_0$ which entails 
\[ \Psi'(s_\eps) + H_{j_0}^\eps (p) \le 0\]
which cannot hold true for $p$ very large because of the uniform coercivity of $H_{j_0}^\eps$. 
The proof is now complete. 
\end{proof}

\subsection{Reducing the set of test functions}
\label{svs}

We show in this subsection, that to check the flux-limited junction
condition, it is sufficient to consider very specific test functions.
This important property is useful both from a theoretical point
of view and from the point of view of applications. 

We consider functions satisfying a Hamilton-Jacobi equation in $J
\setminus \{ 0 \}$, that is to say, solutions of
\begin{equation}\label{eq::gr2}
u_t + H_i(u_x) =0 \quad \mbox{for}\quad (t,x)\in (0,T)\times 
J_i^* 
\end{equation}
for $i=1,\dots,N$. The non-increasing part $H_i^-$ of the Hamiltonian
$H_i$ is used in the definition of flux-limited junction
conditions. In the next theorem, the non-decreasing part $H_i^+$ is needed. 
It is defined by 
\[H_i^+(q)=\begin{cases}
H_i(q) &\quad \mbox{if}\quad q\ge p_i^0,\\
H_i(p_i^0) &\quad \mbox{if}\quad q < p_i^0
\end{cases}\]
where we recall that $p_i^0$ is a point realizing the minimum of $H_i$. 
\begin{theo}[Reduced set of test functions]\label{th::gr1}
Assume that the Hamiltonians satisfy \eqref{assum:H} and
consider $A \in [A_0,+\infty[$ with $A_0$ given in
    \eqref{eq::A_0}. Given arbitrary solutions $p_i^A\in \R$, $i=1,\dots,N$, of
\begin{equation}\label{eq::gr15}
H_i(p^A_i)=H_i^+(p^A_i)=A,
\end{equation}
let us fix any time independent test function $\phi_0(x)$ satisfying
\[ \partial_i \phi_0 (0)= p_i^A.\] 
Given a function $u:(0,T)\times J\to \R$, the following properties
hold true.
\begin{enumerate}[\upshape i)]
\item \label{thgr1-i} If for all $i=1,\dots,N$, $u$ is an upper
  semi-continuous sub-solution of \eqref{eq::gr2} and satisfies
\begin{equation}\label{eq:weak-cont}
u(t,0) = \limsup_{(s,y) \to (t,0), y \in J_i^*} u(s,y),
\end{equation} 
 then $u$ is a $A_0$-flux limited sub-solution.
\item \label{thgr1-ii} Given $A>A_0$ and $t_0\in (0,T)$, if for all
  $i=1,\dots,N$, $u$ is an upper semi-continuous sub-solution of
  \eqref{eq::gr2} and satisfies \eqref{eq:weak-cont} and for any test
  function $\varphi$ touching $u$ from above at $(t_0,0)$ with
\begin{equation}\label{eq::gr4}
\varphi (t,x) = \psi (t) + \phi_0 (x)
\end{equation}
for some $\psi \in C^1(0;+\infty)$, we have
\[\varphi_t + F_A(\varphi_x) \le 0 \quad \mbox{at}\quad (t_0,0),\]
  then $u$ is a  $A$-flux-limited sub-solution at $(t_0,0)$.
\item \label{thgr1-iii} Given $t_0 \in (0,T)$, if $u$ is lower
  semi-continuous super-solution of \eqref{eq::gr2} and if for any
  test function $\varphi$ touching $u$ from below at $(t_0,0)$
  satisfying \eqref{eq::gr4}, we have
  \begin{equation}\label{eq::rtrac-2}
\varphi_t + F_A(\varphi_x) \ge 0 \quad \mbox{at}\quad (t_0,0),
  \end{equation}
 then $u$ is a $A$-flux-limited super-solution at $(t_0,0)$.
\end{enumerate}
\end{theo}
\begin{rem}
  Theorem~\ref{th::gr1} exhibits (necessary and) sufficient conditions
  for sub- and super-solutions of \eqref{eq::gr2} to be flux-limited
  solutions.  After proving Theorem \ref{th::gr1}, we realized that
  this result shares some similarities with the way of checking the
  entropy condition at the junction for conservation law equations
  associated to bell-shaped fluxes.  Indeed it is known that it is
  sufficient to check the entropy condition only with one particular
  stationary solution of the Riemann solver (see \cite{BKT,akr,AC}).
\end{rem}
\begin{counterex}
  The set of test functions can be reduced to a single one for
  flux-limited sub-solution only if the ``weak continuity''
  condition~\eqref{eq:weak-cont} is imposed. Indeed, if this condition
  is not satisfied, then the conclusion is false. Consider for instance
Hamiltonians reaching their minimum at $p_i^0=0$ and such that $A_0=0$
and consider $A \ge A_0 =0$ such that $A T<1$ and 
consider
\[ u (t,x) = \begin{cases} 1 -A t & \text{ for } (t,x) \in (-T,T) \times \{ 0 \} \\
  0 & \text{ elsewhere.} \end{cases} \]
We remark that $u$ does not satisfy \eqref{eq:weak-cont} but it
trivially satisfies \eqref{eq::gr2}.  Now consider $p_i^\eps \le 0$
such that $H_i (p_i^\eps) = \eps^{-1}$; the test function defined as 
\[ \phi (t,x) = 1 -A t + p_i^\eps x \quad \text{ for } \quad x \in J_i \]
touches $u$ from above at any $(t,0)$ and if $u$ were a $A$-flux-limited
solution, we would get 
\[ -A + A \vee \eps^{-1} \le 0 \]
which is false for $\eps$ small enough. If now $u$ is touched from above 
by a test function $\psi (t) + \phi_0(x)$ at $(t,0)$, then $\psi'(t) = -A$ so that
\[ \psi'(t) + A \le 0.\]
\end{counterex}
In order to prove this result, the two following technical lemmas are needed.
\begin{lem}[Super-solution property for the critical slope on each branch]\label{lem:ip0-gr}
Let $u:(0,T)\times J_i \to \R$ be a lower super-solution of \eqref{eq::gr2} for some $i =1,\dots,N$.
Let $\phi$ be a test function touching $u$ from below at  some point $(t_0,0)$ with  $t_0\in (0,T)$. 
Consider the following \emph{critical slope}
\[ \bar p_{i} = \sup \{  \bar p \in \R: \exists r>0, \, \phi (t,x) + \bar  p x  \le u (t,x) \text{ for } (t,x)
\in  (t_0-r,t_0+r) \times [0,r) \mbox{ with } x\in J_i \}. \]
If $\bar p_i < +\infty$, then we have 
\begin{equation}\label{i-ineq+-gr} 
\phi_t+ H_{i} (\partial_i \phi+\bar p_i)  \ge 0 \quad \mbox{at}\quad  (t_0,0)\quad \mbox{with}\quad \bar p_{i}\ge 0.
\end{equation}
\end{lem}
\begin{proof}
From the definition of $\bar p_{i}$, we know that, for all $\eps>0$ small enough, there exists $\delta = \delta (\eps) \in (0,\eps)$ such that
\[ u(s,y) \ge \phi (s,y) + (\bar p_{i} -\eps) y  \quad
\text{ for all } (s,y) \in (t-\delta,t+\delta) \times [0,\delta)\mbox{ with }y\in J_{i}\]
and there exists $(t_\eps,x_\eps) \in B_{\delta/2} (t,0)$ such that 
\[ u(t_\eps,x_\eps) < \phi (t_\eps,x_\eps) + (\bar p_{i} +\eps) x_\eps.\]
Now consider a smooth function $\Psi:\R^2 \to [-1,0]$ such that 
\[ \Psi \equiv \begin{cases} 0 & \text{ in } B_{\frac12}(0), \\ 
-1 & \text{ outside  } B_1(0)\end{cases} \]
and define 
\[ \Phi (s,y) = \phi (s,y) + 2 \eps \Psi_\delta (s,y) + \begin{cases} 
(\bar p_{i} + \eps) y  & \text{ if } y \in
  J_{i} \\
0 & \text{ if not}
\end{cases} \]
with $\Psi_\delta (s,y) = \delta \Psi (s/\delta,y/\delta)$. We have 
\[ \Phi (s,y) \le \phi (s,y) \le u(s,y) \quad \text{ for } (s,y) \in
B_\delta (t,0) \text{ and } y \notin J_{i} \]
and 
\[\left\{\begin{array}{ll}
\Phi (s,y) = \phi (s,y) - 2 \eps \delta + (\bar p_{i}+\eps)y \le u (s,y) 
&\quad \text{ for } (s,y) \in \left(\partial  B_\delta
(t,0)\right)\cap \left(\R \times J_{i}\right),\medskip \\
\Phi (s,0) \le  \phi (s,0) \le u (s,0) &\quad \text{ for } s\in (t-\delta,t+\delta)
\end{array}\right.\]
and 
\[ \Phi (t_\eps,x_\eps) = \phi (t_\eps,x_\eps) + (\bar p_{i}+\eps)x_\eps > u (t_\eps,x_\eps).\]
We conclude that there exists a point $(\bar t_\varepsilon,\bar
x_\varepsilon) \in B_\delta (t,0) \cap \left(\R \times
J_{i}^*\right)$ such that $u - \Phi$ reaches a minimum in
$\overline{B_\delta(t,0)}\cap \left(\R \times
J_i\right)$. Consequently,
\[ \Phi_t (\bar t_\eps,\bar x_\eps) +  H_{i} (\partial_{i} \Phi
(\bar t_\eps,\bar x_\eps))
 \ge 0\]
which implies 
\[ \phi_t (\bar t_\eps,\bar x_\eps) + 2\eps (\Psi_\delta)_t (\bar t_\eps,\bar x_\eps) +  H_{i} (\partial_{i} \phi
(\bar t_\eps,\bar x_\eps) + 2 \varepsilon \partial_y (\Psi_\delta) (\bar t_\eps,\bar x_\eps) + \bar p_{i}+\eps) \ge 0.\]
Letting $\eps$ go to $0$ yields \eqref{i-ineq+-gr}.
This ends the proof of the lemma.
\end{proof}
\begin{lem}[Sub-solution property  for the critical slope on each branch]\label{lem:ip0-2-gr}
Let  $u:(0,T)\times J_i \to \R$ be a  sub-solution of \eqref{eq::gr2} for some $i=1,\dots,N$.
Let $\phi$ be a test function touching $u$ from above at  some point $(t_0,0)$ with  $t_0\in (0,T)$. 
Consider the following \emph{critical slope},
\[ \bar p_{i} = \inf \{  \bar p \in \R: \exists r>0,\;  \phi (t,x) + \bar  p x  \ge u (t,x) \text{ for } (t,x)
\in  (t_0-r,t_0+r) \times [0,r) \mbox{ with } x\in J_i \}. \]
If $u$ satisfies \eqref{eq:weak-cont} then $-\infty < \bar p_i \le 0$ and  
\begin{equation}\label{i-ineq+2-gr} 
\phi_t+ H_{i} (\partial_i \phi+\bar p_i)  \le 0 \quad \mbox{at}\quad  (t_0,0).
\end{equation}
\end{lem}
\begin{proof}
We only prove that $\bar p_i > -\infty$ since this is the only main difference with the proof of the previous lemma. 

Assume that $\overline{p}_i = -\infty$.  This implies that there
exists $p_n \to - \infty$ and $r_n >0$ such that $\phi + p_n x \ge u$
in $B_n = (t_0-r_n,t_0+r_n) \times [0,r_n) \subset \R \times J_i$.
Remark first that, replacing $\phi$ with $\phi+ (t-t_0)^2+x^2$ if
necessary, we can assume that 
\begin{equation}\label{eq:strict} 
u(t,x) < \phi(t,x) + p_n x \text{ if } (t,x) \neq (t_0,0).
\end{equation}
 In particular, there exits $\delta_n >0$ such
that $\phi +p_n x \ge u + \delta_n$ on on
$\partial B_n \setminus \{ (t_0,0) \}$,
where we recall that by definition of $\partial B_n$ (inside $J_T$) does not contain $(t_0-r_0,t_0+r_0)\times \left\{0\right\}$.
Since $u$ satisfies \eqref{eq:weak-cont}, there exists $(t_\eps,x_\eps) \to (t_0,0)$ such 
that $x_\eps \in J_i^*$ and $u(t_0,0) = \lim_{\eps \to 0} u (t_\eps,x_\eps)$. 

We now introduce the following perturbed test function 
\[ \Psi (t,x) = \phi (t,x) + p_n x + \frac{\eta}x \]
where $\eta=\eta (\eps)$ is a small parameter to be chosen later. 
Let $(s_\eps,y_\eps)$ realizing the infimum of $\Psi-u$ in $B_n$. In particular,
\begin{equation}\label{estim:pen}
(\phi +p_n (\cdot) -u)(s_\eps,y_\eps) \le  \Psi (s_\eps,y_\eps) - u(s_\eps,y_\eps) 
\le  \Psi (t_\eps,x_\eps) - u(t_\eps,x_\eps) \to 0 \quad \text{ as } \quad \eps \to 0
\end{equation}
as soon as $\eta(\eps) = o (x_\eps)$. In particular, in view of \eqref{eq:strict}, this implies 
that $(s_\eps,y_\eps) \to (t_0,0)$ as $\eps \to 0$. Since $u$ is a subsolution of \eqref{eq::gr2}, 
we know that 
\[ \phi_t (s_\eps,y_\eps) + H_i \left(\phi_x (s_\eps,y_\eps) + p_n -
\frac{\eta}{y_\eps^2} \right) \le 0.
\]
Hence we can pass to the limit as $\eps \to 0$ in the viscosity inequality and get 
\[ 
\phi_t (t_0,0) + H_i (\phi_x (t_0,0) + p_n^0) \le 0
\]
where $p_n^0 = \liminf_{\eps \to 0} p_n - \frac{\eta}{y_\eps^2} \in [-\infty,0]$. The 
previous inequality implies in particular that $p_n^0 > - \infty$ and $p_n^0$ is bounded 
from below by a constant $C$ which only depends on $H_i$ and $\phi_t,\phi_x$ at $(t_0,0)$. 
But this also implies that $p_n \ge C$ which is the desired contradiction. The proof of the
finiteness of $\overline{p}_i$ is now complete. 
\end{proof}

We are now ready to make the proof of Theorem \ref{th::gr1}.
\begin{proof}[Proof of Theorem \ref{th::gr1}]
We first prove the results concerning sub-solutions and then turn to
super-solutions. 

\paragraph{Sub-solutions.}
Let $u$ be a sub-solution of \eqref{eq::gr2}.  Let $\phi$ be a test
function touching $u$ from above at $(t_0,0)$.  Let
$\phi_t(t_0,0)=-\lambda$. We want to show
\begin{equation}\label{eq::gr10}
F_A(\phi_x)\le \lambda \quad \mbox{at}\quad (t_0,0).
\end{equation}
Notice that by Lemma~\ref{lem:ip0-2-gr}, for all $i=1,\dots,N$, there
exists $\bar p_i\le 0$ such that
\begin{equation}\label{eq::gr11}
H_i(\partial_i \phi  +\bar p_i)\le \lambda \quad \mbox{at}\quad (t_0,0).
\end{equation}
In particular, we deduce that
\begin{equation}\label{eq::gr12}
A_0\le \lambda.
\end{equation}
Inequality (\ref{eq::gr11}) also implies that at  $(t_0,0)$
\begin{align*}
F_A(\phi_x) & = \max(A,\displaystyle \max_{i=1,\dots,N} H_i^-(\partial_i \phi))\\
& \le \max(A, \displaystyle\max_{i=1,\dots,N} H_i^-(\partial_i \phi +\bar p_i))\\
& \le  \max(A, \displaystyle \max_{i=1,\dots,N} H_i(\partial_i \phi +\bar p_i))\\
& \le \max (A,\lambda).
\end{align*}
In particular for $A=A_0$, this implies the desired inequality
\eqref{eq::gr10}.  Assume now that \eqref{eq::gr10} does not hold
true. Then we have
\[A_0\le \lambda< A.\]
Then \eqref{eq::gr11} implies that 
\[\partial_i \phi(t_0,0)  +\bar p_i < p^A_i=\partial_i\phi_0(0).\]
Let us consider the modified test function
\[\varphi(t,x)= \phi (t,0) +\phi_0(x) \quad \mbox{for}\quad x\in J\]
which is still a test function touching $u$ from above at $(t_0,0)$
(in a small neighborhood).  This test function $\varphi$ satisfies in
particular \eqref{eq::gr4}.  Because $A>A_0$, we then conclude that
\[\varphi_t + F_A(\varphi_x) \le 0 \quad \mbox{at}\quad (t_0,0)\]
\textit{i.e.}
\[-\lambda + A\le 0\]
which gives a contradiction. Therefore (\ref{eq::gr10}) holds true.

\paragraph{Super-solutions.}
Let $u$ be a super-solution of \eqref{eq::gr2}.  Let $\phi$ be a test
function touching $u$ from below at $(t_0,0)$.  Let
$\phi_t(t_0,0)=-\lambda$. We want to show
\begin{equation}\label{eq::gr20}
F_A(\phi_x)\ge \lambda \quad \mbox{at}\quad (t_0,0).
\end{equation}
Notice that by Lemma~\ref{lem:ip0-gr}, there exists $\bar p_i\ge 0$ for $i=1,\dots,N$ such that
\begin{equation}\label{eq::gr21}
H_i(\partial_i \phi  +\bar p_i)\ge \lambda \quad \mbox{at}\quad (t_0,0).
\end{equation}
Note that \eqref{eq::gr20} holds true if $\lambda\le A$ or if there
exists one index $i$ such that $H^-_i(\partial_i \phi +\bar
p_i)=H_i(\partial_i \phi +\bar p_i)$.  Assume by contradiction that
\eqref{eq::gr20} does not hold true. Then we have in particular
\begin{equation}\label{eq::gr22}
A_0\le A < \lambda \le  H^+_i(\partial_i \phi  +\bar p_i) \quad \mbox{at}\quad (t_0,0),\quad \mbox{for}\quad i=1,\dots,N.
\end{equation}
From the fact that $H^-_i(\partial_i \phi +\bar
p_i)<H_i(\partial_i \phi +\bar p_i)$ for all index $i$, we deduce in particular that
$$\partial_i \phi(t_0,0)  +\bar p_i >p^A_i=\partial_i\phi_0(0).$$
We then introduce the modified test function
\[\varphi(t,x)=\phi(t_0,0) +\phi_0(x) \quad \mbox{for}\quad x\in J\]
which is a test function touching $u$ from below at $(t_0,0)$ (this is
a test function below $u$ in a small neighborhood of $(t_0,0)$).  This
test function $\varphi$ satisfies in particular (\ref{eq::gr4}).  We
then conclude that
\[\varphi_t + F_A(\varphi_x) \ge 0 \quad \mbox{at}\quad (t_0,0)\]
\textit{i.e.}
\[-\lambda + A\ge 0\]
which gives a contradiction. Therefore \eqref{eq::gr20} holds true.
This ends the proof of the theorem.
\end{proof}

\subsection{An additional characterization of flux-limited sub-solutions}

As an application of Theorem \ref{th::gr1}, we give an equivalent
characterization of sub-solutions in terms of the properties of its
trace at the junction point $x=0$.
\begin{theo}[Equivalent characterization of flux-limited sub-solutions] \label{th::gr6}
Assume that the Hamiltonians $H_i$ satisfy \eqref{assum:H}.  Let
$u:(0,T)\times J\to \R$ be an upper semi-continuous sub-solution of
\eqref{eq::gr2}.  If $u$ is a $A$-flux-limited sub-solution then for
any function $\psi \in C^1(0,T)$ such that $\psi$ touches $u(\cdot,0)$
from above at $t_0\in (0,T)$, we have
\begin{equation}\label{eq::gr26}
\psi_t + A \le 0 \quad \mbox{at}\quad t_0.
\end{equation}
Conversely, if \eqref{eq::gr26} holds true for any $\psi$ as above and 
if $u$ satisfies for all $i$,
\[ u (t,0 ) = \limsup_{(s,y) \to (t,0), y \in J_i^*} u(s,y), \]
then $u$ is a $A$-flux-limited sub-solution. 
\end{theo}
\begin{proof}[Proof of Theorem \ref{th::gr6}]
We successively prove that the condition is necessary and sufficient. 

\paragraph{Necessary condition.}
Let $\psi\in C^1(0,T)$ touching $u(\cdot,0)$ from above at $(t_0,0)$ with $t_0\in (0,T)$. 
As usual, we can assume without loss of generality that the contact point is strict.
Let $\varepsilon>0$ small enough in order to satisfy
\begin{equation}\label{eq::gr25}
\frac{1}\varepsilon >p^A_i 
\end{equation}
where $p^A_i$ is chosen as in \eqref{eq::gr15}. Let
\[\phi(t,x)=\psi(t) +\frac{x}{\varepsilon} \quad \mbox{for}\quad x\in J_i \quad \mbox{for}\quad i=1,\dots,N.\]
For $r>0,\delta>0$, let
\[\Omega:=\left(t_0-r,t_0+r\right)\times B_\delta(0)\]
where $B_\delta(0)$ is the ball in $J$ centered at $0$ and of radius $\delta$.
From the upper semi-continuity of $u$, we can choose $r,\delta$  small
enough, and then $\varepsilon$  small enough, so that
\[\sup_\Omega (u-\phi) > \sup_{\partial \Omega} (u-\phi).\]
Therefore there exists a point $P_\varepsilon=(t_\varepsilon,x_\varepsilon)\in \Omega$ such that we have
\[\sup_\Omega (u-\phi) = (u-\phi)(P_\varepsilon).\]
If $x_\varepsilon\in J_i^*$, then we have
\[\phi_t + H_i(\partial_i\phi) \le 0 \quad \mbox{at}\quad P_\varepsilon\]
\textit{i.e.}
\[\psi'(t_\varepsilon) + H_i(\varepsilon^{-1})\le 0.\]
This is impossible for $\varepsilon$ small enough, because of the coercivity of $H_i$.
Therefore we have $x_\varepsilon=0$, and get
\[\phi_t + F_A(\phi_x) \le 0 \quad \mbox{at}\quad P_\varepsilon.\]
Because of \eqref{eq::gr25}, we deduce that $F_A(\phi_x) =A$ and then
\[\psi'(t_\varepsilon) +A \le 0 \quad \mbox{with}\quad t_\varepsilon\in (t_0-r,t_0+r).\]
In the limit $r\to 0$, we get the desired inequality \eqref{eq::gr26}.

\paragraph{Sufficient condition.}
Let $\phi(t,x)$ be a test function touching $u$ from above at
$(t_0,0)$ for some $t_0\in (0,T)$. From Theorem \ref{th::gr1}, we know
that we can assume that $\phi$ satisfies \eqref{eq::gr4}.  Then
$\phi(t,0)$ touches $u(t,0)$ from above at $t_0$. Therefore we have by
assumption
\[\phi_t(t_0,0) +A \le 0.\]
Because of \eqref{eq::gr4}, we get the desired inequality
\[\phi_t +F_A(\phi_x) \le 0 \quad \mbox{at}\quad (t_0,0).\]
This ends the proof of the theorem.
\end{proof}

\subsection{General junction conditions reduce to flux-limited ones}

\begin{pro}[General junction conditions reduce to flux-limited ones]
\label{pro:fa-gen-sup}
Let the Hamiltonians satisfy \eqref{assum:H} and $F$ satisfy
\eqref{assum:F}. There exists $A_F \in \R$ such that
\begin{itemize}
\item any relaxed
super-solution of \eqref{eq:hj-f} is an $A_F$-flux-limited super-solution
 and any relaxed sub-solution of
\eqref{eq:hj-f} such that for all $i=1,\dots,N$, 
\[ u(t,0) = \limsup_{(s,y) \to (t,0), y \in J_i^*} u(s,y) \]
is a $A_F$-flux-limited sub-solution;
\item any $A_F$-flux-limited sub-solution (resp. super-solution) is a
  relaxed sub-solution (resp. super-solution) of \eqref{eq:hj-f}.
\end{itemize}
\end{pro}
\begin{counterex}
  If the ``weak continuity'' condition does not hold, then the
  conclusion of the proposition is false. Indeed, consider $N=1$ and
  $H_1(p)= |p|$ and $F \equiv 0$. In this case $A_0 = 0$ and
  $A_F = 0$.  Then the function
  \[ u (t,x) = \begin{cases} 1 & \text{ if } x =0, \\ 0 & \text{ if }
    x >0 \end{cases}\]
  is a relaxed solution of \eqref{eq:hj-f} but it does not satisfy the
  ``weak continuity'' condition. Moreover, it is not a
  $0$-flux-limited sub-solution: indeed, $\phi (t,x)= 1 + p_i x$ for
  $x \in J_i$ touches $u$ from above and $\phi_t + F_A (\phi_x) = F_A (p)$ 
which is not necessarily non-positive since $p$ can be chosen arbitrarily.
\end{counterex}
The flux limiter $A_F$ is given by the following lemma. 
\begin{lem}[Definitions of $A_F$ and $\bar p$]\label{lem::gr51}
Let $\bar p^0=(\bar p_1^0,\dots,\bar p_N^0)$ with $\bar p_i^0\ge p_i^0$ be
the minimal real number such that $H_i(\bar p_i^0)=A_0$ with $A_0$
given in (\ref{eq::A_0}).
\begin{description}
\item If $F(\bar p^0)\ge A_0$, then there exists a unique $A_F \in \R$
  such that there exists $\bar p=(\bar p_1,\dots,\bar p_N)$ with $\bar
  p_i\ge \bar p_i^0\ge p_i^0$ such that
\[H_i(\bar p_i)=H_i^+(\bar p_i)=A_F=F(\bar p).\]
\item If $F(\bar p^0)< A_0$, we set $A_F=A_0$ and $\bar p=\bar p^0$.
\end{description}
In particular, we have
\begin{eqnarray}\label{eq::pr9}
 \{ \forall i :  p_i\ge \bar p_i \} & \Rightarrow F(p) \le A_F, \\
\label{eq::pr10}
 \{ \forall i : p_i \le \bar p_i\} & \Rightarrow F(p) \ge A_F.
\end{eqnarray} 
\end{lem}
\begin{proof}[Proof of Proposition~\ref{pro:fa-gen-sup}]
  Let $A$ denote $A_F$. We first prove that relaxed super-solutions
  are flux-limited solutions. We only do the proof for super-solutions
  since it is very similar for sub-solutions.

Without loss of generality, we assume that $u$ is lower
semi-continuous.  Consider a test function $\phi$ touching $u$ from
below at $(t,x) \in (0,+\infty) \times J$,
\[ \phi \le u \text{ in } B_R(t,x) \quad \text{ and } \quad \phi (t,x)= u(t,x)\] 
for some $R>0$. If $x \neq 0$, there is
nothing to prove. We therefore assume that $x=0$. In particular, we
have
\begin{equation}\label{ineq}
 \phi_t (t,0) + \max (F(\phi_x (t,0)), \max_i H_i (\partial_i \phi (t,0)))
\ge 0.
\end{equation}
By Theorem \ref{th::gr1}, we can assume that the test function
satisfies
\begin{equation}\label{eq::gr50}
\partial_i \phi (t,0) = \bar p_i 
\end{equation}
where $\bar p_i$ is given in Lemma \ref{lem::gr51}. 
We now want to prove that
\[ \phi_t (t,0) + A \ge 0.\]
This follows immediately from \eqref{ineq}, \eqref{eq::gr50} and the
definition of $\bar p_i$ in Lemma~\ref{lem::gr51}.  \medskip

We now prove that flux-limited sub-solutions are relaxed sub-solutions. 
Once again, we only do the proof for sub-solutions since it is very similar
for super-solutions. Consider a test function $\phi$ touching $u$ from above
at $(t,0)$. Then 
\[ A_F \vee \max_i H_i^-(p_i) \le \lambda \]
with $p_i = \partial_i \phi (t,0)$ and $\lambda = - \phi_t (t,0)$. 
We distinguish three cases. 

Assume first that for all $i$, $p_i \ge \pi_i^+ (A_F)$. 
Then $F(p) \le F(\pi^+(A_F)) \le A_F \le \lambda.$

If there exists $i_0$ such that $p_{i_0} < \pi_{i_0}^+(A_F)$ and $H_{i_0} (p_{i_0}) \le A_F$,
we have $H_{i_0} (p_{i_0}) \le \lambda$. 

If there exists $i_0$ such that $p_{i_0} < \pi_{i_0}^+(A_F)$ and $H_{i_0} (p_{i_0}) > A_F$,
then we have $H_{i_0} (p_{i_0}) = H_{i_0}^- (p_{i_0}) \le \lambda$.
The proof is now complete.
\end{proof}

\subsection{Existence of solutions}

\begin{theo}[Existence]\label{th::3}
  Let $T>0$ and $J$ be the junction defined in \eqref{eq::J}.  Assume
  that Hamiltonians satisfy \eqref{assum:H}, that the junction
  function $F$ satisfies \eqref{assum:F} and that the initial datum
  $u_0$ is uniformly continuous. Then there exists a relaxed viscosity
  solution $u$ of \eqref{eq:hj-f}-\eqref{eq::2} in $[0,T)\times J$ and
  a constant $C_T>0$ such that
\[|u(t,x)-u_0(x)|\le C_T \quad \text{for all}\quad (t,x)\in [0,T)\times J.\]
\end{theo}
\begin{proof}[Proof of Theorem~\ref{th::3}]
The proof follows classically along the lines of Perron's method (see
\cite{I,CGG}), and then we omit details. 

\paragraph{Step 1: Barriers.}
Because of the uniform continuity of $u_0$, for any $\eps \in (0,1]$,
it can be regularized by convolution to get a modified initial data
$u_0^\eps$ satisfying
\begin{equation}\label{eq::34}
|u_0^\eps-u_0|\le \eps \quad \text{and}\quad |(u_0^\eps)_x|\le L_\eps
\end{equation}
with $\displaystyle L_\eps \ge \max_{i=1,...,N} |p_i^0|$. Indeed, if
we consider $u_i: \R \to \R$ such that $u_i(x) = (u_0)|_{J^i} (x)$ for $x\ge 0$
and $u_i (x) = u_i(0)$ for $x<0$, we can get $u_i^\eps$ such that $|u_i^\varepsilon -u_0|\le \eps/2$ on ${J_i}$
and $|(u_i^\eps)_x |\le L_\eps$. In particular, $|u_i(0)-u_0(0)| \le \eps/2$. We can now
define $u_0^\eps (x) = u_i(x) -u_i(0)+ u_0(0)$ and get $u_0^\eps$ such that \eqref{eq::34} holds true. 
Let 
\[C_\eps = \max \left(\max_{i=1,...,N} \max_{|p_i|\le L_\eps}
|H_i(p_i)|, \max_{|p_i|\le L_\eps} F(p_1,\dots,p_N)\right).\]
Then the functions 
\begin{equation}\label{eq::35}
u_\eps^\pm(t,x)=  u_0^\eps(x)\pm C_\eps t \pm \eps
\end{equation}
are global super and sub-solutions with respect to the initial data
$u_0$. We then define
\[u^+(t,x)=\inf_{\eps\in (0,1]} u^+_\eps(t,x)\quad
\text{and}\quad u^-(t,x)=\sup_{\eps\in (0,1]} u^-_\eps(t,x).\]
Then we have $u^-\le u^+$ with $u^-(0,x)= u_0(x) = u^+(0,x)$.
Moreover, by stability of sub/super-solutions (see Proposition~\ref{pro::2}),
we get that $u^+$ is a super-solution and $u^-$ is a sub-solution of
\eqref{eq:hj-f} on $(0,T)\times J$.

\paragraph{Step 2: Maximal sub-solution and preliminaries.}
Consider the set
\[S=\left\{w:[0,T)\times J\to \R,\quad 
\text{$w$ is a sub-solution of \eqref{eq:hj-f} 
on $(0,T)\times J$}, \quad u^-\le w\le u^+\right\}.\] It contains
  $u^-$. Then the function
  \[u(t,x)=\sup_{w\in S} w(t,x)\] is a sub-solution of \eqref{eq:hj-f}
  on $(0,T)\times J$ and satisfies the initial condition.  It remains
  to show that $u$ is a super-solution of \eqref{eq:hj-f} on
  $(0,T)\times J$.  This is classical for a Hamilton-Jacobi equation
  on an interval, so we only have to prove it at the junction point.
  We assume by contradiction that $u$ is not a super-solution at
  $P_0=(t_0,0)$ for some $t_0\in (0,T)$. This implies that there
  exists a test function $\varphi$ satisfying $u_*\ge \varphi$ in a
  neighborhood of $P_0$ with equality at $P_0$, and such that
\begin{equation}\label{eq::13}
\left\{\begin{array}{l}
\varphi_t + F(\varphi_x) <0,\\
\varphi_t + H_i(\partial_i \varphi) <0 ,\quad \text{for}\quad i=1,...,N
\end{array}\right|  \quad \text{at}\quad P_0.
\end{equation}
We also have $\varphi\le u_*\le u^+_*$. As usual, the fact that $u^+$
is a super-solution and condition~\eqref{eq::13} imply that we cannot
have $\varphi=(u^+)_*$ at $P_0$. Therefore we have for some $r>0$ small
enough
\begin{equation}\label{eq::15}
\varphi< (u^+)_* \quad \text{on}\quad \overline{B_r(P_0)}
\end{equation}
where we define the ball $B_r(P_0)=\left\{(t,x)\in (0,T)\times J,\quad
|t-t_0|^2 + d^2(0,x)<r^2\right\}$.  Substracting $|(t,x)-P_0|^2$ to
$\varphi$ and reducing $r>0$ if necessary, we can assume that
\begin{equation}\label{eq::14}
\varphi < u_* \quad \text{on}\quad \overline{B_r(P_0)}\setminus \left\{P_0\right\}.
\end{equation}
Further reducing $r>0$, we can also assume that \eqref{eq::13} still
holds in $\overline{B_r(P_0)}$.

\paragraph{Step 3: Sub-solution property and contradiction.}
We claim that $\varphi$ is a sub-solution of \eqref{eq:hj-f} in
$B_r(P_0)$.  Indeed, if $\psi$ is a test function  touching
$\varphi$ from above at $P_1=(t_1,0)\in B_r(P_0)$, then
\[\psi_t(P_1) = \varphi_t(P_1) \quad \text{and}\quad 
\partial_i \psi(P_1)\ge \partial_i \varphi(P_1) \quad \text{for}\quad
i=1,...,N.\] 
Using the fact that $F$ is non-increasing with respect
to all variables, we deduce that
\[\psi_t + F(\psi_x) <0 \quad \text{at}\quad P_1\]
as desired. Defining for $\delta>0$,
\[u_\delta = \begin{cases}
\max(\delta+ \varphi, u) &\quad \text{in }  B_r(P_0),\\
u &\quad \text{outside}
\end{cases}\]
and using \eqref{eq::14}, we can check that $u_\delta
= u> \delta+ \varphi$ on $\partial B_r(P_0)$ for $\delta>0$ small
enough.  This implies that $u_\delta$ is a sub-solution lying above
$u^-$. Finally \eqref{eq::15} implies that $u_\delta\le u^+$ for
$\delta>0$ small enough.  Therefore $u_\delta\in S$, but is is
classical to check that $u_\delta$ is not below $u$ for $\delta>0$,
which gives a contradiction with the maximality of $u$.
\end{proof}

\subsection{Further properties of flux-limited solutions}

In this section, we focus on properties of solutions of the following equation
\begin{equation}\label{eq::gr56}
u_t + H(u_x)=0 
\end{equation}
for a single Hamiltonian satisfying \eqref{assum:H}. We start with
the following result, which is strongly related to the reformulation
of state constraints from \cite{ik}, and its use in \cite{at}.
\begin{pro}[Reformulation of state constraints]\label{pro::gr7}
Assume that $H$ satisfies \eqref{assum:H}.
Let  $u:(0,T)\times [a,b] \to \R$. If $u$ satisfies
\begin{equation}\label{eq::gr8}
\left\{\begin{array}{l}
u_t + H(u_x)= 0 \quad \mbox{for} \quad (t,x)\in (0,T)\times (a,b),\\
u_t + H^-(u_x)= 0 \quad \mbox{for} \quad (t,x)\in (0,T)\times \left\{a\right\},\\
u_t + H^+(u_x)= 0 \quad \mbox{for} \quad (t,x)\in (0,T)\times \left\{b\right\}
\end{array}\right.
\end{equation}
in the viscosity sense if and only if
\begin{equation}\label{eq::gr7}
\left\{\begin{array}{l}
u_t + H(u_x)\ge 0 \quad \mbox{for} \quad (t,x)\in (0,T)\times \overline{\Omega},\\
u_t + H(u_x)\le 0 \quad \mbox{for} \quad (t,x)\in (0,T)\times \Omega
\end{array}\right.
\end{equation}
in the viscosity sense and 
\begin{equation}\label{eq:abc}
 u(t,c) = \limsup_{(s,y) \to (t,c), y \in ]a,b[} u(s,y) \quad \text{ for } \quad c=a,b.
\end{equation}
\end{pro}
\begin{proof}[Proof of Proposition \ref{pro::gr7}]
Remark first that only  boundary conditions should be studied.

We first prove that \eqref{eq::gr7} implies \eqref{eq::gr8}. 
From Theorem~\ref{th::gr1}-\ref{thgr1-i}), we deduce that the viscosity
sub-solution inequality is satisfied on the boundary for
\eqref{eq::gr8} with the choice $A=A_0=\min H$.

Let us now consider a test function $\varphi$ touching $u_*$ from
below at the boundary $(t_0,x_0)$.  We want to show that $u_*$ is a
viscosity super-solution for \eqref{eq::gr8} at $(t_0,x_0)$.  By
Theorem~\ref{th::gr1}, it is sufficient to check the inequality
assuming that
\[ \varphi (t,x) = \psi (t) + \phi (x) \]
with
$$\left\{\begin{array}{ll}
H(\phi_x)=H^+(\phi_x)=A_0 & \quad \mbox{at}\quad x_0 \quad \mbox{if}\quad x_0=a,\\
H(\phi_x)=H^-(\phi_x)=A_0 &  \quad \mbox{at}\quad x_0\quad \mbox{if}\quad x_0=b.
\end{array}\right.$$
(The second equality involves $H^-$ instead of $H^+$ because, locally
around $b$, the domain looks like $]b-\eps,b]$ and not $[b,b+\eps[$.)
Remark that we have in all cases $H(\phi_x)=H^+(\phi_x)=H^-(\phi_x)$
at $x_0$.  We then deduce from the fact that $u_*$ is a viscosity
super-solution of \eqref{eq::gr7}, that $u_*$ is also a viscosity
super-solution of \eqref{eq::gr8} at $(t_0,x_0)$.

We now prove that \eqref{eq::gr8} implies \eqref{eq::gr7}.  The second
line of \eqref{eq::gr7} is easy to get. As far as the first line is
concerned, it follows from the fact that $H\ge H^\pm$. This ends the
proof of the proposition.
\end{proof}

\begin{pro}[Classical viscosity solutions are also solutions ``at one point''] \label{pro::gr42}
Assume that $H$ satisfies \eqref{assum:H} and consider a classical
Hamilton-Jacobi equation posed in the whole line,
\begin{equation}\label{eq::gr41}
u_t+H(u_x) = 0 \quad \mbox{for all}\quad (t,x)\in (0,T)\times \R
\end{equation} 
\begin{enumerate}[\upshape i)]
\item {\upshape (Sub-Solutions)}
Let $u:(0,T)\times \R\to \R$ be a sub-solution of (\ref{eq::gr41}).
Then $u$ satisfies
\begin{equation}\label{eq::rtrac-3}
u_t (t,0)+ \max(H^+(u_x(t,0^-)),H^-(u_x(t,0^+)))\le 0.
\end{equation}
\item {\upshape (Super-Solutions)}
Let $u:(0,T)\times \R\to \R$ be a super-solution of (\ref{eq::gr41}).
Then $u$ satisfies
\begin{equation}\label{eq::rtrac-4}
u_t (t,0) + \max(H^+(u_x(t,0^-)),H^-(u_x(t,0^+)))\ge 0.
\end{equation}
\end{enumerate}
\end{pro}
\begin{rem}
We remark that the reverse implication holds true since, when testing with $C^1$ function,
$u_x( t,0^-)=u_x (t,0^+)$ and $H = \max (H^+,H^-)$. 
\end{rem}
\begin{proof} {\bf Sub-solutions.} 
In order to apply  Theorem~\ref{th::gr1}-\ref{thgr1-i}), 
we first remark that the following lemma, whose proof is postponed, 
 implies that $u$ satisfies the ``weak continuity''
condition (\ref{eq:weak-cont}) with the choice $H_2=H_3=H$ and $H_1 (p)= H(-p)$.
\begin{lem}[``weak continuity'' condition with $C^1$ test
  functions] \label{lem:wc-c1} Given two Hamiltonians $H_1$, $H_2$
  satisfying \eqref{assum:H} and $H_3$ continuous and coercive, let
  $u : (0,T) \times \R \to \R$ be upper semi-continuous such that
  all $C^1$ function $\phi$ touching $u$ from above at $(t,x)$
  satisfies
\[\begin{cases} 
\phi_t (t,x) + H_1 (\phi_x (t,x)) \le 0 & \text{ if } x <0, \\
\phi_t (t,x) + H_2 (\phi_x (t,x)) \le 0 & \text{ if } x >0, \\
\phi_t (t,x) + H_3 (\phi_x (t,x)) \le 0 & \text{ if } x =0.
\end{cases}\]
Then for all $t \in (0,T)$, 
\[ u (t,0) = \limsup_{(s,y) \to (t,0), y>0} u(s,y) = \limsup_{(s,y) \to (t,0), y<0} u(s,y).\]
\end{lem}
Thanks to Theorem~\ref{th::gr1}-\ref{thgr1-i}, 
we deduce that $u$ is a $A_0$-flux-limited sub-solution with $A_0=\min H$,
which implies (\ref{eq::rtrac-3}).\\

\noindent {\bf Super-solutions.} We do not have to use Lemma \ref{lem:wc-c1}, but instead we have to check 
(\ref{eq::rtrac-2}) with $A=A_0$ and a good choice of a test function $\phi_0$ on $J=J_1\cup J_2$.
Indeed, we simply choose
$$\phi(x)=\left\{\begin{array}{lll}
\phi_0(y) & \quad \mbox{for}\quad y=x\in J_1 & \quad \mbox{if}\quad x\ge 0,\\
\phi_0(y) & \quad \mbox{for}\quad y=-x\in J_2 & \quad \mbox{if}\quad x\le 0,\\
\end{array}\right.$$
such that $\phi$ is $C^1$ on $\R$ and $H(\phi_0'(0))=\min H=A_0$.
This implies (\ref{eq::rtrac-4}) and ends the proof of the proposition.
\end{proof}
We now prove Lemma~\ref{lem:wc-c1}.
\begin{proof}[Proof of Lemma~\ref{lem:wc-c1}]
Assume first that there exists $t^*$ such that 
\[ u(t^*,0) > \limsup_{(s,y) \to (t^*,0), y>0} u(s,y) 
\quad \text{ and } \quad u(t^*,0) > \limsup_{(s,y) \to (t^*,0), y<0} u(s,y).\]
Since $u(t,0)$ is upper semi-continuous, there exists $t_0$ arbitrarily
close to $t^*$ with $u(t_0,0)$ arbitrarily close to $u(t^*,0)$ such that 
there exists a $C^1$ function $\Psi (t)$ (strictly) touching $u(t,0)$ from 
above at $(t_0,0)$. In particular, we can get $\delta_0$ and $r_0$ such that 
\[ u (t_0,0) \ge u(s,y) + \delta_0 \text{ for } (s,y) \in B_{r_0}(t,0), y \neq 0.\]
In this first case, the test function $\Psi (t) + px$ (with $p$ arbitrary) 
touches $u$ from above at $(t_0,0)$. This implies 
\[ \Psi'(t_0) + H_3 (p) \le 0 \]
which contradicts the coercivity of $H_3$. 

Assume now that 
\[ u(t^*,0) = \limsup_{(s,y) \to (t^*,0), y\ge 0} u(s,y) \quad \text{ and
} \quad u(t^*,0) > \limsup_{(s,y) \to (t^*,0), y<0} u(s,y).\]
In this case, we can argue as in the proof of
Lemma~\ref{lem:weak-cont}, the intervals $(-\infty,0]$ and
$[0,+\infty)$ playing the role of $J_{i}$ for $i\not= i_0$ and $J_{i_0}$
respectively; in particular, we construct a test function
$\Psi (t) + px $ with $p$ very negative and get a contradiction 
with the coercivity of $H_2$.

The remaining case is similar to the previous one. The proof is now complete. 
\end{proof}
\begin{pro}[Restriction of sub-solutions are sub-solutions] \label{pro::gr40}
Assume that $H$ satisfies \eqref{assum:H}. Let $u:(0,T)\times \R\to
\R$ be upper semi-continuous satisfying
\begin{equation}\label{eq::gr43}
u_t+H(u_x) \le 0 \quad \mbox{for all}\quad (t,x)\in (0,T)\times \R
\end{equation} 
Then the restriction $w$ of $u$ to $(0,T)\times \left[0,+\infty\right)$ satisfies
\[\left\{\begin{array}{ll}
w_t + H(w_x) \le 0 \quad \mbox{for all}\quad (t,x)\in (0,T)\times \left(0,+\infty\right),\\
w_t + H^-(w_x)\le 0 \quad \mbox{for all}\quad (t,x)\in (0,T)\times \left\{0\right\}.
\end{array}\right.\]
\end{pro}
\begin{proof}[Proof of Proposition \ref{pro::gr40}]
We simply have to study $w$ at the boundary. From
Proposition~\ref{pro::gr42}, we know that $u$ satisfies in the
viscosity sense
\[u_t + \max(H^+(u_x(t,0^-)),H^-(u_x(t,0^+)))\le 0.\]
By Theorem~\ref{th::gr6} with two branches, we deduce that $v(t)=u(t,0)$ satisfies
\[v_t + \min H \le 0.\]
Again by Theorem~\ref{th::gr6} (now with one branch) and because $v(t)=w(t,0)$, we deduce  that $w$ satisfies
\[w_t + H^-(w_x)\le 0 \quad \mbox{for all}\quad (t,0) \in (0,T)\times \left\{0\right\}\]
which ends the proof.
\end{proof}
\begin{rem}\label{eq::gr55}
Notice that the restriction of a super-solution of (\ref{eq::gr56}) may
not be a super-solution on the boundary, as shown by the following example:
for $H(p)=|p|-1$, the solution $u(t,x)=x$ solves $u_t
+ H(u_x)=0$ in $\R$ but does not solve $u_t + H^- (u_x) \ge 0$ at $x=0$.
\end{rem}

\section{Comparison principle on a junction}
\label{s.c}

This section is devoted to the proof of the comparison principle in
the case of a junction (see Theorem~\ref{th::2}). In view of
Propositions~\ref{pro:fa-gen-sup} and \ref{pro::1}, it is enough to
consider sub- and super-solutions (in the sense of
Definition~\ref{defi::1}) of \eqref{eq::1bis} for some $A=A_F$.

It is convenient to introduce the following shorthand notation
\begin{equation}\label{eq::17}
H(x,p)=\left\{\begin{array}{lll}
H_i(p) &\quad \text{for} \quad p=p_i &\quad \text{if} \quad x\in J_i^*,\\
F_A(p) &\quad \text{for} \quad p=(p_1,...,p_N) &\quad \text{if} \quad x=0.
\end{array}\right.
\end{equation}
In particular, keeping in mind the definition of $u_x$ (see \eqref{eq::18}),
Problem~\eqref{eq::1bis} on the junction  can be
rewritten as follows
\[u_t + H(x,u_x)=0 \quad \text{for all}\quad (t,x)\in (0,+\infty)\times J.\]

We next make a trivial but useful observation.
\begin{lem} 
\label{pi0zero}
It is enough to prove Theorem~\ref{th::2} further assuming that
\begin{equation}\label{eq::24}
p_i^0=0 \quad \text{for}\quad i=1,...,N\quad \text{and}\quad
0=H_1(0)\ge H_2(0)\ge ...\ge H_N(0).
\end{equation}
\end{lem}
\begin{proof}
We can assume without loss of generality that 
\[ H_1(p_1^0) \ge ... \ge H_N(p_N^0).\]  
Let us define
\[u(t,x)=\tilde{u}(t,x)+ p_i^0 x - t H_1(p_1^0) \quad \text{for}\quad x\in J_i.\]
Then $u$ is a solution of \eqref{eq::1bis} if and only if $\tilde{u}$
is a solution of \eqref{eq::1bis} with each $H_i$ replaced with
$\tilde{H}_i(p)=H_i(p+p_i^0) -H_1(p_1^0)$ and $F_A$ replaced with
$\tilde{F}_{\tilde{A}}$ constructed using the Hamiltonians
$\tilde{H}_i$ and the parameter $\tilde{A}=A-H_1(p_1^0)$. 
\end{proof}

\subsection{The vertex test function}

Then our key result is the following one.
\begin{theo}[The vertex test function -- general case]\label{th::G}
Let $A\in \R\cup \left\{-\infty\right\}$ and $\gamma>0$. Assume the
Hamiltonians satisfy \eqref{assum:H} and \eqref{eq::24}.  Then there
exists a function $G:J^2\to \R$ enjoying the following properties.
\begin{enumerate}[\upshape i)]
\item \emph{(Regularity)}
\[G\in C(J^2)\quad \text{and}\quad \left\{\begin{array}{l}
G(x,\cdot)\in C^1(J) \quad \text{for all}\quad x\in J,\\
G(\cdot,y)\in C^1(J) \quad \text{for all}\quad y\in J.
\end{array}\right.\]
\item \emph{(Bound from below)} $G\ge 0=G(0,0)$.
\item \emph{(Compatibility condition on the diagonal)}
For all $x\in J$,
\begin{equation}\label{eq::85}
0\le G(x,x) -G(0,0)  \le \gamma.
\end{equation}
\item \emph{(Compatibility condition on the gradients)}
For all $(x,y)\in J^2$,
\begin{equation}\label{eq::17bis}
H(y,-G_y(x,y))-H(x,G_x(x,y)) \le \gamma
\end{equation}
where notation introduced in \eqref{eq::18} and \eqref{eq::17} are used.
\item \emph{(Superlinearity)}
There exists $g:[0,+\infty)\to \R$ nondecreasing and s.t. for $(x,y)\in J^2$
\begin{equation}\label{eq::20}
g(d(x,y))\le G(x,y) \quad \text{and}\quad \lim_{a\to +\infty} \frac{g(a)}{a} = +\infty.
\end{equation}
\item \emph{(Gradient bounds)} For all $K>0$, there exists  $C_K>0$
  such that for all $(x,y)\in J^2$,
\begin{equation}\label{eq::19}
d(x,y)\le K \quad \Longrightarrow \quad |G_x(x,y)|+|G_y(x,y)|\le C_K.
\end{equation}
\end{enumerate}
\end{theo}
\begin{rem}\label{rem:g-less-reg}
  The vertex test function $G$ is obtained as a regularized version of
  a function $G^0$ which is $C^1$ except on the diagonal $x=y$.
  It is in fact possible to check directly that $G^0$ does not
    satisfy the viscosity inequalities on the diagonal in the sense of
    Proposition~\ref{pro::gr42} (when it is not $C^1$ on the
    diagonal).
\end{rem}

\subsection{Proof of the comparison principle}

We will also need the following result whose classical proof is given
in Appendix for the reader's convenience. 
\begin{lem}[A priori control]\label{lem::3}
Let $T>0$ and let $u$ be a sub-solution and $v$ be a super-solution as
in Theorem~\ref{th::2}. Then there exists a constant $C=C(T)>0$ such
that for all $(t,x),(s,y)\in [0,T)\times J$, we have
\begin{equation}\label{eq::29}
u(t,x)\le v(s,y) + C(1 + d(x,y)).
\end{equation}
\end{lem}

We are now ready to make the proof of comparison principle.

\begin{proof}[Proof of Theorem~\ref{th::2}] 
  As explained at the beginning of the current section, in view of
  Propositions~\ref{pro:fa-gen-sup} and \ref{pro::1}, it is enough to
  consider sub- and super-solutions (in the sense of
  Definition~\ref{defi::1}) of \eqref{eq::1bis} for some $A=A_F$.

  The remaining of the proof proceeds in several steps.

\paragraph{Step 1: the penalization procedure.}
We want to prove that
\[M=\sup_{(t,x)\in [0,T)\times J} (u(t,x)-v(t,x))\le 0.\] 
Assume by contradiction that $M>0$.  Then for $\alpha,\eta >0$ small
enough, we have $M_{\eps,\alpha}\ge 3M/4>0$ for all
$\eps,\nu >0$ with
\begin{equation}\label{eq::43}
M_{\eps,\alpha}=\sup_{(t,x),(s,y)\in [0,T)\times J}\left\{u(t,x)-v(s,y)
-\eps G\left(\frac{x}{\eps},\frac{y}{\eps}\right)-\frac{(t-s)^2}{2\nu}-\frac{\eta}{T-t}
-\alpha\frac{d^2(0,x)}{2}\right\}
\end{equation}
where the vertex test function $G\ge 0$ is given by
Theorem~\ref{th::G} for a parameter $\gamma$ satisfying
\[ 0<\gamma <\min \left(\frac{\eta}{2T^2}, \frac{M}{8 \eps}\right).\]
Since $M_{\eps,\alpha} \ge 3M/4$, the supremum can be taken over points  
 $(x,y)$ such that the corresponding value is larger than $M/2$. 
Thanks to Lemma~\ref{lem::3} and \eqref{eq::20}, these points satisfy
\begin{equation}\label{eq::40}
0<\frac{M}2 \le C(1+d(x,y)) -\eps
g\left(\frac{d(x,y)}{\eps}\right)
-\frac{(t-s)^2}{2\nu}-\frac{\eta}{T-t} -\alpha\frac{d^2(0,x)}{2}
\end{equation}
which implies in particular that
\begin{equation}\label{eq::41}
\eps g\left(\frac{d(x,y)}{\eps}\right)\le C(1+d(x,y)).
\end{equation}
Because of the superlinearity of $g$ appearing in \eqref{eq::20}, we
know that for any $K>0$, there exists a constant $C_K>0$ such that for
all $a\ge 0$
\[Ka -C_K\le g(a).\]
For $K\ge 2C$, we deduce from \eqref{eq::41} that
\begin{equation}\label{eq::42}
d(x,y) \le \inf_{K\ge 2C} \left\{\frac{C}{K-C} +
\frac{C_K}{C}\eps\right\} =: \omega(\eps)
\end{equation}
where $\omega$ is a concave, nondecreasing function satisfying
$\omega(0)=0$.  We deduce from \eqref{eq::40} and \eqref{eq::42} that
the supremum in \eqref{eq::43} is reached at some point
$(t,x,s,y)=(t_\nu,x_\nu,s_\nu,y_\nu)$.

\paragraph{Step 2: use of the initial condition.} 
We first treat the case where $t_\nu=0$ or $s_\nu=0$.  If there exists a
sequence $\nu\to 0$ such that $t_\nu=0$ or $s_\nu=0$, then calling
$(x_0,y_0)$ any limit of subsequences of $(x_\nu,y_\nu)$, we get from
\eqref{eq::43} and the fact that $M_{\eps,\alpha} \ge M/2$ that
\[0 < \frac{M}2\le u_0(x_0)-u_0(y_0)\le \omega_0(d(x_0,y_0))\le
\omega_0\circ \omega (\eps)\] where $\omega_0$ is the modulus of
continuity of the initial data $u_0$ and $\omega$ is defined in
\eqref{eq::42}. This is impossible for $\eps$ small enough.

\paragraph{Step 3: use of the equation.} 
We now treat the case where $t_\nu>0$ and $s_\nu>0$.  Then we can
write the viscosity inequalities with
$(t,x,s,y)=(t_\nu,x_\nu,s_\nu,y_\nu)$ using the shorthand
notation~\eqref{eq::17} for the Hamiltonian,
\begin{align*}
\frac{\eta}{(T-t)^2} + \frac{t-s}{\nu}+ H(x,G_x(\eps^{-1}x,\eps^{-1}y)+ \alpha d(0,x))\le 0,\\
 \frac{t-s}{\nu}  + H(y,-G_y(\eps^{-1}x,\eps^{-1}y))\ge 0.
\end{align*}
Substracting these two inequalities, we get 
\[\frac{\eta}{T^2}\le H(y,-G_y(\eps^{-1}x,\eps^{-1}y))
- H(x,G_x(\eps^{-1}x,\eps^{-1}y)+ \alpha d(0,x)).\] Using
\eqref{eq::17bis} with $\gamma\in \left(0,\frac{\eta}{2T^2}\right)$,
we deduce for $p=G_x(\eps^{-1}x,\eps^{-1}y)$
\begin{equation}\label{eq::45}
\frac{\eta}{2T^2}\le H(x,p)- H(x,p+ \alpha d(0,x)).
\end{equation}
Because of \eqref{eq::19} and \eqref{eq::42}, we see that $p$ is
bounded for $\eps$ fixed by $|p|\le
C_{\frac{\omega(\eps)}{\eps}}$.  Finally, for
$\eps>0$ fixed and $\alpha\to 0$, we have $\alpha d(0,x)\to 0$,
and using the uniform continuity of $H(x,p)$ for $x\in J$ and $p$
bounded, we get a contradiction in \eqref{eq::45}. The proof is now
complete. 
\end{proof}

\subsection{The vertex test function versus the fundamental solution} \label{ss.godo}

Recalling the definition of the germ ${\mathcal G}_A$ (see
\eqref{eq::50}), let us associate with  any $(p,\lambda)\in
      {\mathcal G}_A$ the following functions for $i,j = 1,...,N$,
\[u^{p,\lambda}(t,x,s,y)=p_i x -p_j y -\lambda (t-s) 
\quad \text{for}\quad (x,y)\in J_i\times J_j,\quad t,s\in\R.\]
The reader can check that they solve the following system,
\begin{equation}\label{eq::60}
\left\{\begin{array}{r}
u_t+ H(x,u_x)=0,\\
-u_s + H(y,-u_y)=0.
\end{array}\right.
\end{equation}
Then, for $N\ge 2$, the function
$\tilde{G}^0(t,x,s,y)=(t-s)G^0\left(\frac{x}{t-s},\frac{y}{t-s}\right)$
can be rewritten as
\begin{equation}\label{eq::62}
\tilde{G}^0(t,x,s,y)=\sup_{(p,\lambda)\in {\mathcal G}_A}
u^{p,\lambda}(t,x,s,y) \quad \text{for}\quad (x,y)\in J\times J,\quad
t-s\ge 0
\end{equation}
which satisfies
\begin{equation}\label{eq::61}
\tilde{G}^0(s,x,s,y)=\begin{cases}
0& \quad \text{if}\quad x=y,\\
+\infty & \quad \text{otherwise.}
\end{cases}
\end{equation}
For $N\ge 2$ and $\displaystyle A> A_0$, it is possible to check (at
least in the smooth convex case -- see \eqref{eq::16} below) that
$\tilde{G}^0$ is a viscosity solution of \eqref{eq::60} for $t-s>0$,
only outside the diagonal $\left\{x=y\not=0\right\}$. Therefore, even
if \eqref{eq::62} appears as a kind of (second) Hopf formula (see for
instance \cite{be,abi}), this formula does not provide a true solution
on the junction.

On the other hand, under more restrictive assumptions on the
Hamiltonians and for $A=A_0$ and $N\ge 2$ (see \cite{imz}), there is a
natural viscosity solution of \eqref{eq::60} with the same initial
conditions \eqref{eq::61}, which is $\mathcal
D(t,x,s,y)=(t-s){\mathcal
  D}_0\left(\frac{x}{t-s},\frac{y}{t-s}\right)$ where ${\mathcal D}_0$
is a cost function defined in \cite{imz} following an optimal control
interpretation.  The function ${\mathcal D}_0$ is not $C^1$ in general
(but it is semi-concave) and it is much more difficult to study it and
to use it in comparison with $G^0$.  Nevertheless, under suitable
restrictive assumptions on the Hamiltonians, it would be also possible
to replace in our proof of the comparison principle the term
$\eps G(\eps^{-1}x,\eps^{-1}y)$ in \eqref{eq::43}
by $\eps {\mathcal D}_0(\eps^{-1}x,\eps^{-1}y)$.

\section{Construction of the vertex test function}
\label{s.G}

This section is devoted to the proof of Theorem~\ref{th::G}.  Our
construction of the vertex test function $G$ is follows the same
pattern as the particular subcase of normalized convex Hamiltonians
$H_i$.

\subsection{The case of smooth convex Hamiltonians}

Assume that the Hamiltonians $H_i$ satisfy the following assumptions
for $i=1,...,N$,
\begin{equation}\label{eq::16}
\left\{\begin{array}{l}
H_i \in C^2(\R) \quad \text{with}\quad H_i''>0 \quad \text{on}\quad \R,\\
H_i'<0 \quad \text{on}\quad (-\infty,0) \quad \text{and}\quad H_i'>0 
\quad \text{on}\quad (0,+\infty),\\
\displaystyle \lim_{|p|\to +\infty} \frac{H_i(p)}{|p|} =+\infty.
\end{array}\right.
\end{equation}
It is useful to associate with each $H_i$ satisfying \eqref{eq::16}
its partial inverse functions $\pi^\pm_i$: 
\begin{equation}\label{eq::21}
\text{ for } \lambda\ge H_i(0), \quad H_i(\pi^\pm_i(\lambda))=\lambda 
\quad \text{such that}\quad \pm \pi^\pm_i(\lambda) \ge 0.
\end{equation}
Assumption~\eqref{eq::16} implies that $\pi^\pm_i \in C^2(\min H_i,
+\infty)\cap C([\min H_i, +\infty))$ thanks to the inverse function theorem.

We recall that $G^0$ is defined, for $i,j=1,...,N$, by
\[G^0(x,y)=\sup_{(p,\lambda)\in \mathcal{G}_A} (p_i x -p_j y
-\lambda) \quad \text{if}\quad (x,y)\in J_i\times J_j\] where
$\mathcal{G}_A$ is defined in \eqref{eq::50}. Replacing $A$
with $\max(A,A_0)$ if necessary, we can always assume that $A\ge A_0$
with $A_0$ given by \eqref{eq::A_0}.

\begin{pro}[The vertex test function -- the smooth convex case]\label{pro:vtf}
Let $A\ge A_0$ with $A_0$ given by \eqref{eq::A_0} and assume that the
Hamiltonians satisfy \eqref{eq::16}.  Then $G^0$ satisfies 
\begin{enumerate}[\upshape i)]
\item \emph{(Regularity)}
\[G^0\in C(J^2)\quad \mbox{and}\quad \left\{\begin{array}{l}
G^0\in C^1(\left\{(x,y)\in J\times J,\quad x\not=y\right\}),\\
G^0(0,\cdot)\in C^1(J) \quad \mbox{and}\quad G^0(\cdot,0)\in C^1(J);
\end{array}\right.\]
\item \emph{(Bound from below)} $G^0 \ge G^0 (0,0)=-A;$
\item \label{item-vtf} \emph{(Compatibility conditions)} \eqref{eq::85} holds with
  $\gamma =0$ for all $x\in J$ and \eqref{eq::17bis} holds with
  $\gamma =0$ for $(x,y)$ such that either $x\not= y$ or $x =y=0$;
\item \emph{(Superlinearity)} \eqref{eq::20} holds for some $g=g^0$;
\item \label{item-bg} \emph{(Gradient bounds)} \eqref{eq::19} holds only for $(x,y) \in J^2$
  such that $x \neq y$  or $(x,y)=(0,0)$;
\item \label{saturation}
\emph{(Saturation close to the diagonal)} For $i\in \left\{1,...,N\right\}$
  and for $(x,y)\in J_i\times J_i$, we have $G^0(x,y)=\ell_i(x-y)$
  with $\ell_i \in C(\R) \cap C^1(\R\setminus \left\{0\right\})$ and
\[\ell_i(a)=\left\{\begin{array}{lll}
a\pi^+_i(A) -A & \quad \text{if}\quad & 0\le a\le z_i^+\\
a \pi^-_i(A) -A & \quad \text{if}\quad & z_i^- \le a\le 0\\
\end{array}\right.\]
where $(z_i^-,z_i^+):=(H_i'(\pi^-_i(A)),H_i'(\pi^+_i(A)))$ and the
functions $\pi^\pm_i$ are defined in \eqref{eq::21}.  Moreover $G^0\in
C^1(J_i\times J_i)$ if and only if $\pi_i^+(A)=0=\pi_i^-(A)$.
\end{enumerate}
\end{pro}
\begin{rem}
The compatibility condition~\eqref{eq::17bis} for $x\not= y$, is in fact an equality
with $\gamma=0$ when $N \ge 2$. 
\end{rem}
The proof of this proposition is postponed until
Subsection~\ref{subsec:pro80}. With such a result in hand, we can now prove
Theorem~\ref{th::G} in the case of smooth convex Hamiltonians. 
\begin{lem}[The case of smooth convex Hamiltonians]\label{lem:case-convex}
Assume that the Hamiltonians satisfy \eqref{eq::16}. Then the conclusion of
Theorem~\ref{th::G} holds true. 
\end{lem}
\begin{proof}
We note that the function $G^0+A$ satisfies all the properties required
for $G$, except on the diagonal $\left\{(x,y)\in J\times J,x=y\not=
0\right\}$ where $G^0$ may not be $C^1$.  To this end, we first
introduce the set $I$ of indices such that $G^0\not\in C^1(J_i\times
J_i)$.  We know from Proposition~\ref{pro:vtf}~\ref{saturation}) that
\[I=\left\{i\in \left\{1,...,N\right\},\quad \pi_i^+(A)
>\pi_i^-(A)\right\}.\]
For $i\in I$, we are going to contruct a regularization
$\tilde{G}^{0,i}$ of $G^0$ 
in a neighbourhood of the diagonal $\left\{(x,y)\in J_i\times J_i, \; 
x=y \not=0\right\}$.

\paragraph{Step 1: Construction of $\tilde{G}^{0,i}$ for $i\in I$.}
Let us define
\[L_i(a)=\begin{cases}
a\pi^+_i(A) & \quad \text{if}\quad a\ge 0,\\
a \pi^-_i(A) & \quad \text{if}\quad a\le 0.
\end{cases}\]
We first consider a convex $C^1$ function $\tilde{L}_i:\R\to \R$
coinciding with $L_i$ outside $(z_i^-,z_i^+)$, that we choose such that
\begin{equation}\label{eq::l9}
0\le \tilde{L}_i -L_i \le 1.
\end{equation}
Then for $\eps\in (0,1]$, we define
\[\ell_i^\eps(a):=\begin{cases}
\eps \tilde{L}_i\left(\frac{a}{\eps}\right) -A 
& \quad \text{if}\quad a \in [\eps z_i^-,\eps z_i^+],\\
\ell_i(a)  & \quad \text{otherwise.}
\end{cases}\]
which is a $C^1(\R)$ (and convex) function.
We now consider a cut-off function $\zeta$ satisfying for some constant $B>0$
\begin{equation}\label{eq::l8}
\left\{\begin{array}{l}
\zeta\in C^\infty(\R),\\
\zeta'\ge 0,\\
\zeta > 0 \quad \text{in} \quad (0,+\infty), \\
\zeta = 0 \quad \text{in}\quad (-\infty,0],\\
\zeta = 1 \quad \text{in}\quad [B,+\infty),\\
\pm z_i^\pm \zeta'  < 1 \quad \text{in}\quad  (0,+\infty)
\end{array}\right.
\end{equation}
and for $\eps \in (0,1]$, we define for $(x,y)\in J_i\times J_i$:
\[\tilde{G}^{0,i}(x,y)=\ell_i^{\eps \zeta(x+y)}(x-y).\]

\paragraph{Step 2: First properties of $\tilde{G}^{0,i}$.}
By construction, we have $\tilde{G}^{0,i}\in C^1((J_i\times J_i)
\setminus \left\{(0,0)\right\})$.  Moreover we have
\[\tilde{G}^{0,i}=G^0 \quad \text{on}\quad (J_i\times J_i) 
\setminus \delta_i^\eps\]
where
\[\delta_i^\eps = \left\{(x,y)\in J_i\times J_i,\quad  
\eps z_i^- \zeta(x+y) <  x-y < \eps z_i^+ \zeta(x+y)\right\}\]
is a neighborhood of the diagonal
\[\left\{(x,y)\in J_i\times J_i,\quad x=y\not= 0\right\}.\]
Because of \eqref{eq::l9}, we also have
\begin{equation}\label{eq::l12}
0 \le G^0-  \tilde{G}^{0,i} \le \eps .
\end{equation}
As a consequence of \eqref{eq::l8}, we have in particular
\[(J_i\times J_i) \setminus \delta_i^\eps \quad 
\supset \quad (J_i \times \left\{0\right\})\cup (\left\{0\right\}\times J_i)\]
and moreover $\tilde{G}^{0,i}$ coincides with $G^0$ on a neighborhood
of $(J_i^* \times \left\{0\right\})\cup (\left\{0\right\}\times
J_i^*)$, which implies that
\begin{equation}\label{eq::l10}
  \tilde{G}^{0,i}= G^0,\quad \tilde{G}^{0,i}_x = G^0_x \quad \text{and}\quad \tilde{G}^{0,i}_y = G^0_y 
\quad \text{on}\quad (J_i \times \left\{0\right\})\cup (\left\{0\right\}\times J_i).
\end{equation}

\paragraph{Step 3: Computation of the gradients of $\tilde{G}^{0,i}$.}
For $(x,y)\in \delta^\eps_i$, we have
\[\left\{\begin{array}{rl}
\tilde{G}^{0,i}_x(x,y)&=\displaystyle (\ell_i^{\eps \zeta(x+y)})'(x-y) + \eps \zeta'(x+y)\  \xi_i\left(\frac{x-y}{\eps \zeta(x+y)}\right)\\
-\tilde{G}^{0,i}_y(x,y)&=\displaystyle (\ell_i^{\eps \zeta(x+y)})'(x-y) - \eps \zeta'(x+y)\  \xi_i\left(\frac{x-y}{\eps \zeta(x+y)}\right)
\end{array}\right.\]
with
\[\xi_i(b)=\tilde{L}_i (b) - b \tilde{L}_i'(b)\]
while if $(x,y)\in (J_i\times J_i)\setminus \delta^\eps_i$ we have
\[\tilde{G}^{0,i}_x(x,y) = -\tilde{G}^{0,i}_y(x,y).\]
Given $\gamma>0$, and using the local uniform  continuity of $H_i$, we
see that we have for $\eps$ small enough
\[H_i(\tilde{G}^{0,i}_x))\le H_i(-\tilde{G}^{0,i}_y) + \gamma \quad \text{in}\quad J_i^*\times J_i^*\]
and using \eqref{eq::l10}, we get
\begin{equation}\label{eq::l11}
H(x,\tilde{G}^{0,i}_x(x,y))- H(y,-\tilde{G}^{0,i}_y(x,y)) \le \gamma \quad \text{for all}\quad (x,y)\in J_i\times J_i.
\end{equation}

\paragraph{Step 4: Definition of $G$.} We set for $(x,y)\in J_i\times J_j$:
\[G(x,y)=\begin{cases}
G^0(x,y)+A &\quad \text{if}\quad i\not=j \quad \text{or}
\quad i=j\not\in I,\\
\tilde{G}^{0,i}(x,y)+A &\quad \text{if}\quad i=j\in I.
\end{cases}\]
From the properties of $G^0$, we recover all the expected properties
of $G$ with $g(a)=g^0(a)+A$.  In particular from
Proposition~\ref{pro:vtf}-\eqref{item-vtf}, \eqref{eq::l11} and
\eqref{eq::l12}, we respectively get the compatibility condition for
the Hamiltonians \eqref{eq::17bis} and the compatibility condition on
the diagonal \eqref{eq::85} for $\eps$ small enough. As far as
\eqref{eq::20} is concerned, we remark that $G(x,y)$ coincide with
$G^0(x,y)+A$ when $d(x,y)$ is large. As far as \eqref{eq::19} is
concerned, $G_x$ and $G_y$ coincide with $G^0_x$ and $G^0_y$ if
$x \in J_i$ and $y\in J_j$ with $i \neq j$; hence we can apply
Proposition~\ref{pro:vtf}-\eqref{item-bg}.  In the case where $x$ and
$y$ belongs to the same branch, $G (x,y)$ is a smooth function of
$x-y$ when $x+y \ge 1$ (since $\zeta(r) =1$ for $r \ge 1$). In
particular, $G_x$ and $G_y$ are bounded as soon as $|x-y|$ is so.
Finally, when $x+y \le 1$, $(x,y)$ is in a compact set and $G_x$ and
$G_y$ are also bounded. 
\end{proof}

\subsection{The general case}

Let us consider a slightly stronger assumption than \eqref{assum:H},
namely
\begin{equation}\label{eq::3}
\left\{\begin{array}{l}
H_i \in C^2(\R) \quad \text{with}\quad H_i''(p_i^0)>0,\\
H_i'<0  \quad  \text{on}\quad  (-\infty,p_i^0)\quad \text{and}\quad  H_i'>0 \quad \text{on}\quad  (p_i^0,+\infty),\\
\displaystyle \lim_{|q|\to +\infty} H_i(q)=+\infty.
\end{array}\right.
\end{equation}

We will also use the following technical result which allows us 
to reduce certain non-convex Hamiltonians to convex Hamiltonians.
\begin{lem}[From non-convex to convex Hamiltonians]\label{lem::1}
Given Hamiltonians $H_i$ satisfying \eqref{eq::3} and
\eqref{eq::24}, there exists a function $\beta:\R\to \R$ such that
the functions $\beta \circ H_i$ satisfy \eqref{eq::16} for
$i=1,...,N$.  Moreover, we can choose $\beta$ such that
\begin{equation}\label{eq::23}
\beta \quad \text{is convex},\quad \beta\in C^2(\R), \quad \beta(0)=0 
\quad \text{and}\quad \beta'\ge \delta >0.
\end{equation}
\end{lem}
\begin{proof}
Recalling \eqref{eq::21}, it is easy to check that $(\beta \circ
H_i)''> 0$ if and only if we have
\begin{equation}\label{eq::22}
(\ln \beta')'(\lambda) > -\frac{H_i''}{(H_i')^2}\circ
  \pi_i^\pm(\lambda)\quad \text{for}\quad \lambda\ge H_i(0).
\end{equation}
Because $H_i''(0)>0$, we see that the right hand side is negative for
$\lambda$ close enough to $H_i(0)$.  

Then it is easy to choose a function $\beta$ satisfying \eqref{eq::22}
and \eqref{eq::23}. Indeed, since we impose $\beta (0)=0$, we only
need to find a non-decreasing $C^1$ function $\beta'$ bounded from
below by some $\delta >0$. Let $\beta'$ be written in the form
$e^B$. We impose $(e^B)(0)=\delta$ and \eqref{eq::23} is satisfied if
$B'$ is bounded from below in $[H_i(0),+\infty)$ by a given function
which is negative at $H_i (0)$.  The subtle point is that $\beta$
should not depend on $i$. It is enough to take the supremum of these
lower bounds, add a small constant which preserves the ``room'' at
$H_i (0)$ and consider a smooth function above this supremum.

Finally, compositing $\beta$ with another convex increasing function
which is superlinear at $+\infty$ if necessary, we can ensure that
$\beta\circ H_i$ is superlinear.
\end{proof}
\begin{lem}[The case of smooth Hamiltonians]\label{lem:case-smooth}
Theorem~\ref{th::G} holds true if the Hamiltonians satisfy
\eqref{eq::3}. 
\end{lem}
\begin{proof}
We assume that the Hamiltonians $H_i$ satisfy \eqref{eq::3}. Thanks to
Lemma~\ref{pi0zero}, we can further assume that they satisfy
\eqref{eq::24}. Let $\beta$ be the function given by Lemma
\ref{lem::1}. If $u$ solves \eqref{eq::1bis} on $(0,T)\times J$, then
$u$ is also a viscosity solution of
\begin{equation}\label{eq::1ter}
\left\{\begin{array}{lll}
\bar \beta (u_t) + \hat{H}_i(u_x)= 0  &\text{for}\quad t\in (0,T) &\quad \text{and}\quad x\in J_i^*,\\
\bar \beta (u_t)  + \hat{F}_{\hat{A}}(u_x)=0   &\text{for}\quad  t\in (0,T) &\quad \text{and}\quad x=0
\end{array}\right.
\end{equation}
with $\hat{F}_{\hat{A}}$ constructed as $F_A$ where $H_i$ and $A$ are replaced with
$\hat{H}_i$ and $\hat{A}$ defined as follows
\[\hat{H}_i = \beta\circ H_i,\quad \hat{A} = \beta(A)\]
and $\bar \beta(\lambda)=-\beta(-\lambda)$.  We can then apply
Theorem~\ref{th::G} in the case of smooth convex Hamiltonians (namely Lemma \ref{lem:case-convex}) to
construct a vertex test function $\hat{G}$ associated to problem
\eqref{eq::1ter} for every $\hat{\gamma}>0$.  This means that we have
with $\hat{H}(x,p) = \beta(H(x,p))$,
\[\hat{H}(y, -G_y) \le \hat{H}(x,G_x) +\hat{\gamma}.\]
This implies
\[{H}(y, -G_y) \le \beta^{-1}(\beta({H}(x,G_x)) + \hat{\gamma}) \le {H}(x,G_x) + \hat{\gamma} |(\beta^{-1})'|_{L^\infty(\R)}.\]
Because of the lower bound on $\beta'$ given by Lemma~\ref{lem::1}, 
we get $|(\beta^{-1})'|_{L^\infty(\R)}\le 1/\delta$ which yields 
the compatibility condition~\eqref{eq::17bis}
with $\gamma= \hat{\gamma}/\delta$ arbitrarily small.
\end{proof}
We are now in position to prove Theorem~\ref{th::G} in the general
case. 
\begin{proof}[Proof of Theorem~\ref{th::G}]
Let us now assume that the Hamiltonians only satisfy
\eqref{assum:H}. In this case, we simply approximate the Hamiltonians
$H_i$ by other Hamiltonians $\tilde{H}_i$ satisfying \eqref{eq::3}
such that
\[|H_i-\tilde{H}_i|\le \gamma.\]
We then apply Theorem~\ref{th::G} to the Hamiltonians $\tilde{H}_i$
and construct an associated vertex test function $\tilde{G}$ also for
the parameter $\gamma$.  We deduce that
\[H(y, -\tilde{G}_y) \le H(x,\tilde{G}_x) + 3\gamma\]
with $\gamma>0$ arbitrarily small, which shows again the compatibility
condition on the Hamiltonians \eqref{eq::17bis} for the Hamiltonians
$H_i$'s. The proof is now complete in the general case.
\end{proof}

\begin{rem}[A variant in the proof of construction of $G^0$]
When the Hamiltonians are not convex, it is also possible to use the
function $\beta$ from Lemma~\ref{lem::1} in a different way by
defining directly the function $G^0$ as follows
\[\tilde{G}^0(x,y)=\sup_{(p,\lambda)\in \mathcal{G}_A}
\left(p_i x -p_j y - \beta(\lambda)\right).\]
\end{rem}

\subsection{A special function}

In order to prove Proposition~\ref{pro:vtf}, we first need to study a
special function $\oG$. Precisely, we define the following convex
function for  $z=(z_1,...,z_N)\in\R^N$,
\[\oG(z)=\sup_{(p,\lambda)\in \mathcal{G}_A} (p\cdot z - \lambda).\]

We remark that if $\pm z_i \ge 0$ then the supremum will select
$\pm p_i \ge 0$ if the two vectors $(p_1,\dots,\pm p_i, \dots p_N)$
belong to the germ $\mathcal{G}_A$.  Moreover, in view of the
definition of the germ, see \eqref{eq::50}, we know that
$(p,\lambda) \in \mathcal{G}_A$ if and only if
$p_i = \pi_i^{\sigma_i} (\lambda)$ for some $\sigma_i \in \{-,+\}$,
$\lambda \ge A$ and $(\sigma_1,\dots,\sigma_N) \neq (+,\dots,+)$ for
$\lambda >A$. These facts explain why we will assume that
$\sigma \neq (+,\dots,+)$ in the two next lemmas.  \medskip

For $\sigma = (\sigma_1,\dots,\sigma_N) \in \{\pm\}^N$, we consider
the following subsets of $\R^N$,
\begin{align*}
 Q_\sigma &= \{z =(z_1,\dots,z_N) \in \R^N: \quad \sigma_i z_i \ge 0,
\quad i=1,\dots,N\} \\
 \Delta_\sigma & = \{ z =(z_1,\dots,z_N) \in Q_\sigma : 
\sum_{i=1}^N \frac{\sigma_i z_i}{\bar z^\sigma_i(A)} \le 1 \} 
\end{align*}
where $\bar z^\sigma_i(A) = \sigma_i H_i'(\pi_i^{\sigma_i}(A))\ge 0$
and the functions $\pi^\pm_i$ are defined in \eqref{eq::21}.
We also make precise that we use the following convenient convention, 
\begin{equation}\label{eq::l7}
\frac{\bar z_i}{\bar z^\sigma_i(A)} =\left\{\begin{array}{ll} 0 &\quad
\text{if}\quad \bar z_i=0,\\ +\infty &\quad \text{if}\quad \bar z_i>0
\quad \text{and}\quad \bar z^\sigma_i(A)=0.
\end{array}\right.
\end{equation}
\begin{lem}[The function $\oG$ in $Q_\sigma$]\label{lem:oGloc}
Under the assumptions of Proposition~\ref{pro:vtf}, we have, for 
 any $\sigma \in \{\pm\}^N$ with $\sigma \neq (+,\dots,+)$
if $N \ge 2$:
\begin{enumerate}[\upshape i)]
\item $\oG$ is $C^1$ on $Q_\sigma$ (up to the boundary).
\item For all $z \in Q_\sigma$, there exists a unique $\lambda = \oL
  (z) \ge A$ such that
\begin{align*}
\oG (z)&= p \cdot z - \lambda \\ 
\nabla \oG (z) &= p = (p_1,\dots,p_N)\\
p_i & = \pi_i^{\sigma_i} (\lambda) 
\end{align*}
 with $(p,\lambda) \in  \mathcal{G}_A$. In particular, $p_i$ is unique. 
\item \label{iii} For all $z \in Q_\sigma$, $\oL(z) = A$ if and only
  if $z \in \Delta_\sigma$. In particular, $\oG$ is linear in
  $\Delta_\sigma$: for $z \in \Delta_\sigma$, $\oG(z) = \sum_i \pi_i^{\sigma_i}(A) z - A$. 
\end{enumerate}
\end{lem}
Before giving global properties of $\oG$, we introduce the set 
\begin{equation}\label{eq:defiomega}
 \bar \Omega = \begin{cases} \R & \text{ if } N=1,\\ 
\R^N \setminus (0,+\infty)^N & \text{ if } N \ge 2.
\end{cases}
\end{equation}
\begin{lem}[Global properties of $\oG$ and $\oL$]\label{lem:oGglob}
Under the assumptions of Proposition~\ref{pro:vtf}, the function $\oG$
is convex and finite in $\R^N$, reaches its minimum $-A$ at $0$ and
the function $\oL$ is continuous in $\bar\Omega$.
\end{lem}
\begin{proof}[Proof of Lemmas~\ref{lem:oGloc} and \ref{lem:oGglob}]
Let $\sigma\in \left\{\pm\right\}^N$ and $z\in Q_\sigma$. We set
\[\pi^\sigma(\lambda)=(\pi^{\sigma_1}_1(\lambda),...,\pi^{\sigma_N}_N(\lambda)).\]
Using the fact that $(\pi^\sigma(A),A) \in \mathcal{G}_A$, we get
$\oG(z) \ge \oG(0)=-A$.

\paragraph{Step 1: Explicit expression of $\oG$.}
For $\sigma \neq (+,\dots,+)$ if $N \ge 2$, we have
\begin{equation}\label{eq::l1}
(p,\lambda)\in \mathcal{G}_A \cap \left(Q_\sigma \times \R\right) \quad 
\Longleftrightarrow \quad 
\lambda\ge A  \quad \text{and}\quad p=\pi^{\sigma}(\lambda).
\end{equation}
This implies in particular that
\begin{equation}\label{eq::l2}
\oG(z)=\sup_{\lambda\ge A}\ (z\cdot \pi^\sigma(\lambda) - \lambda).
\end{equation}

\paragraph{Step 2: Optimization.}
Because of the superlinearity of the Hamiltonians $H_i$ (see
\eqref{eq::16}), we have for $z\not= 0$,
\[\lim_{\lambda\to +\infty} f^\sigma(\lambda) = -\infty \quad
\text{for}\quad f^\sigma(\lambda):=z\cdot \pi^\sigma(\lambda) -
\lambda.\] 
Therefore the supremum in \eqref{eq::l2} is reached for some
$\lambda\in [A,+\infty)$, \textit{i.e.}
\[\oG(z)= z\cdot \pi^\sigma(\lambda) - \lambda.\]
Then we have $\lambda=A$ or 
\(\lambda>A \quad \text{and}\quad (f^\sigma)'(\lambda)=0.\)
Note that for $\lambda> A_0$, we can rewrite $(f^\sigma)'(\lambda)=0$ as
\[\sum_{i=1,...,N}\frac{\bar z_i}{\bar z^\sigma_i} = 1 \quad \text{with}\quad \left\{\begin{array}{l}
\bar z_i = \sigma_i z_i \ge 0,\\ \\ \bar z^\sigma_i=\bar
z^\sigma_i(\lambda) := \sigma_i H_i'(\pi_i^{\sigma_i}(\lambda))> 0.
\end{array}\right.\]
Moreover, we have
\[(\bar z^\sigma_i)'(\lambda)=\frac{H_i''(\pi_i^{\sigma_i}(\lambda))}{\sigma_i H_i'(\pi_i^{\sigma_i}(\lambda))}>0\]
where the strict inequality follows from the strict convexity of
Hamiltonians, see \eqref{eq::16}. Moreover, by
definition of $\bar z^\sigma_i$, we have
\[\lim_{\lambda\to +\infty} \bar z^\sigma_i(\lambda) = +\infty\]
because $H_i$ is convex and superlinear.

\paragraph{Step 3: Foliation and definition  of $\oL$.}
Let us consider the sets
\begin{equation}\label{eq::l5}
  P^\sigma(\lambda)= \left\{\begin{array}{ll }
\left\{\bar z \in [0,+\infty)^N,\quad \displaystyle \sum_{i=1,...,N}\frac{\bar z_i}{\bar z^\sigma_i(\lambda)} = 1\right\} &\quad \text{if}\quad \lambda >A,\\
\\
\left\{\bar z\in [0,+\infty)^N, \displaystyle \sum_{i=1,...,N}\frac{\bar z_i}{\bar z^\sigma_i(A)} \le  1\right\} &\quad \text{if}\quad \lambda =A
\end{array}\right.
\end{equation}
(keeping in mind  convention \eqref{eq::l7}).
\begin{figure}
\begin{center}
\includegraphics[height=7cm]{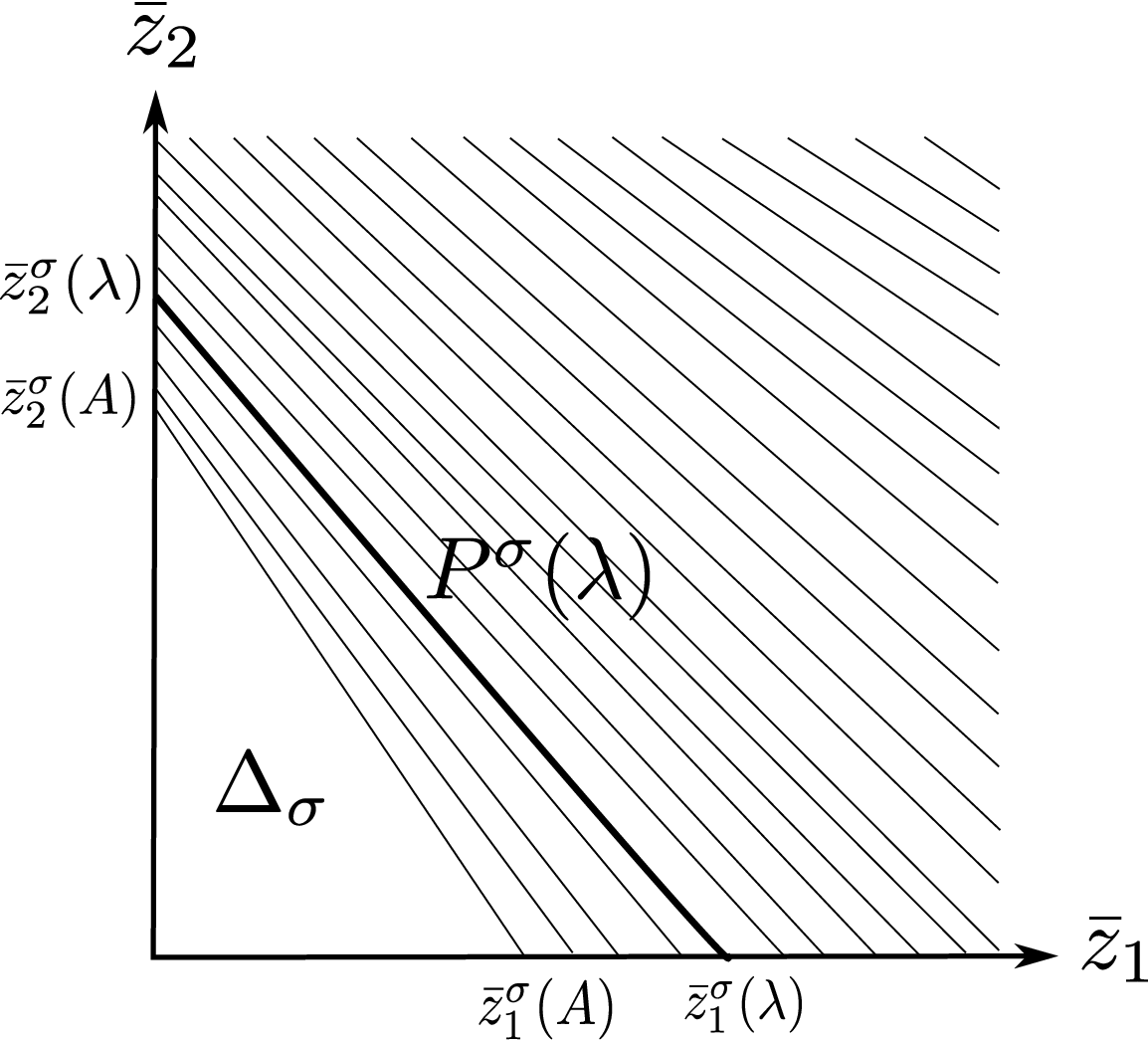}
\end{center}
\caption{The foliation of $[0,+\infty)^2$ ($N=2$) with sets $P^\sigma(\lambda)$ for $\lambda \ge A$.}
\label{figure}
\end{figure}
Because for $\lambda>A$, the intersection points of the piece of hyperplane $P^\sigma(\lambda)$ with each axis $\R e_i$ are $\bar z^\sigma_i(\lambda) e_i$, 
we deduce that we can write the partition (see Figure~\ref{figure})
\[[0,+\infty)^N = \bigcup_{\lambda\ge A} P^\sigma(\lambda)\]
where $P^\sigma(\lambda)$ gives a foliation by hyperplanes for $\lambda>A$.
Then we can define for $z\in Q_\sigma$,
\[\oL^\sigma(z)=\left\{\lambda \quad \text{such that}\quad \bar z \in P^\sigma(\lambda) \quad \text{for}\quad \bar z_i=\sigma_i z_i \quad \text{for}\quad i=1,...,N\right\}.\]
From our definition, we get that the function $\oL^\sigma$ is
continuous on $Q_\sigma$ and satisfies $\oL^\sigma(0)=A$.  For $z\in
Q_\sigma$ such that $z_{i_0} =0$, we see from the definition of
$P^\sigma$ given in \eqref{eq::l5} that the value of $\oL^\sigma(z)$
does not depend on the value of $\sigma_{i_0}$.  Therefore we can glue
up all the $\oL^\sigma$ in a single continuous function $\oL$ defined
for $z\in \bar \Omega$ by
\[\oL(z)= \oL^\sigma(z) \quad \text{if}\quad z\in Q_\sigma.\]
which satisfies $\oL(0)=A$.

\paragraph{Step 4: Regularity of $\oG$ and computation of the gradients.}
For $z\in Q_\sigma \subset \bar \Omega$, we have
\[\oG(z)  = \displaystyle \sup_{\lambda\ge A} \ 
(z\cdot \pi^\sigma(\lambda) -\lambda)\] 
where the supremum is reached only for $\lambda = \oL(z)$.  Moreover
$\oG$ is convex in $\R^N$.  We just showed that the subdifferential of
$\oG$ on the interior of $Q_\sigma$ is the singleton
$\left\{\pi^\sigma(\lambda)\right\}$ with $\lambda = \oL(z)$. This
implies that $\oG$ is differentiable in the interior of $Q_\sigma$ and
\[\nabla \oG(z)=  \pi^\sigma(\lambda) \quad \text{with}\quad \lambda= \oL(z).\]
The fact that the maps $\pi^\sigma$ and $\oL$ are continuous implies
that $\oG_{|Q_\sigma}$ is $C^1$.
\end{proof}

\subsection{Proof of Proposition~\ref{pro:vtf}}
\label{subsec:pro80}

We now turn to the proof of Proposition~\ref{pro:vtf}.
\begin{proof}[Proof of Proposition~\ref{pro:vtf}]
By definition of $G^0$, we have
\[G^0(x,y)=\oG(Z(x,y)) \quad \text{with}\quad 
Z(x,y):=xe_i - y e_j\in \bar\Omega \quad \text{if}\quad (x,y)\in
J_i\times J_j
\]
where $(e_1,...,e_N)$ is the canonical basis of $\R^N$ and $\bar
\Omega$ is defined in \eqref{eq:defiomega}.

\paragraph{Step 1: Regularity.}
Then Lemmas~\ref{lem:oGloc} and \ref{lem:oGglob} imply immediately that
$G^0\in C(J^2)$ and $G^0 \in C^1(R)$
for each region $R$ given by
\begin{equation}\label{eq::71}
R=\begin{cases}
J_i\times J_j &\quad \text{if}\quad i\not=j,\\
T_i^\pm =\left\{(x,y)\in J_i\times J_i,\quad \pm (x-y)\ge 0 \right\} 
 &\quad \text{if}\quad i=j.
\end{cases}
\end{equation}
This regularity of $\oG$ implies in particular the regularity of $G^0$
given in i).

\paragraph{Step 2: Computation of the gradients.}
We also deduce from  Lemma~\ref{lem:oGglob} that 
\[\Lambda(x,y):=\oL(Z(x,y))\]
defines a continuous map $\Lambda:J^2\to [A,+\infty)$
which satisfies
\begin{equation}\label{eq::93}
\Lambda(x,x)=A
\end{equation}
because of Lemma~\ref{lem:oGloc}-\ref{iii}) and $Z(x,x)=0$.  Moreover,
for each $R$ given by \eqref{eq::71} and for all $(x,y)\in R \subset
J_i\times J_j$ we have
\[G^0(x,y)=p_i x -p_j y -\lambda\]
and
\[(G^0_{|R})_x(x,y)=p_i\quad \text{and}\quad (G^0_{|R})_y(x,y)=-p_j\]
with $\lambda=\Lambda(x,y)$ and $(p,\lambda)\in \mathcal{G}_A$ and 
\begin{equation}\label{eq:pipj}
(p_i,p_j)=\left\{\begin{array}{lll}
(\pi^+_i(\lambda),\pi^-_j(\lambda)) &\quad \text{if}\quad R=J_i\times J_j &\quad \text{with}\quad i\not=j,\\
(\pi^\pm_i(\lambda),\pi^\pm_i(\lambda)) &\quad \text{if}\quad R=T^\pm_i &\quad \text{with}\quad i=j. 
\end{array}\right.
\end{equation}

\paragraph{Step 3: Checking the compatibility condition on the gradients.}
Let us consider $(x,y)\in J^2$ with $x=y=0$ or $x\not= y$.
We have
\[(\partial_i G^0(\cdot, y))(x)\in \left\{\pi^\pm_i(\lambda)\right\}
\quad \text{and}\quad -(\partial_j G^0(x,\cdot))(y) \in
\left\{\pi^\pm_j(\lambda)\right\} \quad \text{with}\quad
\lambda=\Lambda(x,y)\ge A.\]

We claim that
\begin{equation}\label{eq::95}
H(x,G^0_x(x,y))=\lambda.
\end{equation}
If $x \neq 0$, then $H(x,G^0_x(x,y)) = H_i (\pi_i^\pm (\lambda)) = \lambda$. 
If $x =0$ and there exists $i$ such that $\sigma_i = -$, then 
$H_i^-(\partial_i^x G^0(0,y)) = H_i^- (\pi_i^- (\lambda)) = \lambda$
and $H_j^- (\partial_j^x G^0(0,y)) = H_j^- (\pi_j^{\sigma_j} (\lambda)) \le \lambda$. 
Hence we also have in this case that \eqref{eq::95} holds true.  
We are left with treating the case where
\begin{equation}\label{eq::94}
x=0\quad \text{and}\quad (\partial_i G^0(\cdot, y))(0) =
\pi^+_i(\lambda) \quad \text{for all}\quad i=1,...,N
\end{equation}
If $0\not= y \in J_j$, then $(x,y)=(0,y)\in T_j^-$ and $(\partial_j
G^0(\cdot, y))(0) = \pi^-_j(\lambda)$. Therefore \eqref{eq::94} only
happens if $y=0$ and then
\[H(0, G^0_x(0,0))=A\]
which still implies \eqref{eq::95}, because $\lambda= \Lambda(0,0)=A$.

In view of \eqref{eq::95}, \eqref{eq::17bis} with equality and
$\gamma=0$ is equivalent to
\begin{equation}\label{eq::96}
H(y,-G^0_y(x,y))=\lambda.
\end{equation}
Arguing like we did to get \eqref{eq::95}, we can treat all cases except the following one
\begin{equation}\label{eq::97}
y=0\quad \text{and}\quad -(\partial_j G^0(x,\cdot))(0) =
\pi^+_j(\lambda) \quad \text{for all}\quad j=1,...,N.
\end{equation}
If $x\in J_i$ and $N\ge 2$, then we can find $j\not=i$ such that
$-(\partial_j G^0(x,\cdot))(0) =\pi^-_j(\lambda)$.  Therefore
\eqref{eq::97} only happens if $N=1$ and then
\[H(0, -G^0_y(x,0))=A\le \lambda.\]

\paragraph{Step 4: Superlinearity.}
In view of the definition of $G^0$, we deduce from \eqref{eq:pipj} that for all $\lambda \ge A$, 
\[G^0(x,y)\ge \left\{\begin{array}{ll}
x \pi^+_i(\lambda)-y \pi^-_j(\lambda) - \lambda & \quad \text{if}\quad i\not=j,\\
(x-y) \pi^\pm_i(\lambda) -\lambda & \quad \text{if}\quad i=j \quad \text{and}\quad \pm (x-y)\ge 0
\end{array}\right.\]
Setting
\[ \pi^0(\lambda):=\min_{\pm, \ i=1,...,N} \pm \pi^\pm_i(\lambda) \ge
0, \]
we get
\[G^0(x,y)\ge d(x,y) \pi^0(\lambda) -\lambda.\]
From the definition \eqref{eq::21}  of $\pi^\pm_i$ and the
assumption~\eqref{eq::16} on the Hamiltonians, we deduce that
\[\pi^0(\lambda)\to +\infty \quad \text{as}\quad \lambda\to +\infty\]
which implies that for any $K\ge 0$, there exists a constant $C_K\ge
0$ such that
\[G^0(x,y)\ge Kd(x,y) - C_K.\]
Therefore we get \eqref{eq::20} with
\[g^0(a)=\sup_{K\ge 0} (Ka - C_K).\]

\paragraph{Step 5: Gradient bounds.}
Note that 
\[\sum_{i=1,...,N} |Z_i(x,y)|= d(x,y).\]
Because each component of the gradients of $G^0$ are equal to one of
the $\left\{\pi^\pm_k(\lambda)\right\}_{\pm,k=1,...,N}$ with $\lambda=
\oL(Z(x,y))$, we deduce \eqref{eq::19} from the continuity of $\oL$
and of the maps $\pi^\pm_k$.

\paragraph{Step 6: Saturation close to the diagonal.}
Point \ref{saturation}) in Proposition~\ref{pro:vtf} follows from
Lemma~\ref{lem:oGloc}-\ref{iii}), from the definition of $\oG$ and
from the regularity of $G^0$. In particular, for $(x,y) \in T_i^\pm$,
$Z=(x-y)e_i$ belongs to $P^\sigma (\lambda) \cup \Delta_\sigma$ with
$\sigma_i = \pm$. Hence, Lemma~\ref{lem:oGloc}-\ref{iii}) implies that
$G^0(x,y) = \pi_i^\pm (A) (x-y) -A$ for
$\pm (x-y) \in [0,\pm z_i^\pm]$ with $z_i^\pm = H_i' (\pi_i^\pm
(A))$.
We recall that $\bar z_i^\sigma = \pm z_i^\pm \ge 0$ appears in the
definition of $P^\sigma (\lambda)$ and $\Delta_\sigma$.
\end{proof}

\subsection{A second vertex test function}

In this subsection, we propose a construction of a second vertex test
function $G^\sharp$ (see Theorem \ref{th::p1} below), that can be seen
as a kind of approximation of the original vertex test function $G$.
This test function is somehow less natural than our previous test
function, but it has the advantage that it is easier to check its
properties. Moreover, it can be useful in applications. 

We introduce the following
\begin{defi}[Piecewise $C^1$ Regularity]\label{defi::p1}
  We say that a function $u$ belongs to $C^{1,\sharp}(J)$, if $u\in
  C(J)$, and if for any branch $J_i$ for $i=1,\dots,N$, there exists a
  sequence of points $(a^i_k)_{k\in \N}$ on the branch $J_i$
  satisfying
$$0=a^i_0< a^i_1< \dots < a^i_k<a^i_{k+1} \to +\infty \quad 
\mbox{as}\quad k \to +\infty$$ 
such that
$$u_{|[a^i_k,a^i_{k+1}]}\in C^1\left([a^i_k,a^i_{k+1}]\right) \quad 
\mbox{for all}\quad k\in \N,\quad i=1,\dots,N.$$
\end{defi}

\subsubsection*{The smooth convex case}

Following what we did in order to construct the first vertex test
function, we first construct $G^\sharp$ in the smooth convex case and
we then derive the general case by approximation. In the smooth convex
case, we first consider
\begin{equation}\label{eq::p5}
G^{0,\sharp}(x,y)=\sup_{k\in \N} \left(\sup_{(p,\lambda_k)\in {\mathcal G}_A}  
(p_ix-p_j y -\lambda_k)\right)\quad \mbox{if}\quad (x,y)\in J_i\times J_j
\end{equation}
for an increasing sequence $(\lambda_k)_{k\in \N}$ satisfying for some constant $\gamma_0>0$
\begin{equation}\label{eq::p8}
\left\{\begin{array}{l}
\lambda_0=A \quad \mbox{and}\quad \lambda_k\to +\infty \quad \mbox{as}\quad k\to +\infty\\
\lambda_{k+1}-\lambda_k\le \gamma_0 \quad \mbox{for all}\quad k\ge 0.
\end{array}\right.
\end{equation}
\begin{lem}[Piecewise linearity] 
The function $G^{0,\sharp}$ is piecewise linear. More precisely, 
\begin{itemize}
\item
For $(x,y)\in J_i\times J_i$,
\[ G^{0,\sharp}(x,y)=\ell_i(x-y)\]
  with $\ell_i \in C(\R)$ and
\[\ell_i(a)=\left\{\begin{array}{lll}
a\pi^+_i(\lambda_k) -\lambda_k & \quad \text{if}\quad & z_i^{k,+}\le a\le z_i^{k+1,+}\\
a \pi^-_i(\lambda_k) -\lambda_k & \quad \text{if}\quad &  z_i^{k+1,-}\le a \le  z_i^{k,-}\\
\end{array}\right| \quad \mbox{for all}\quad k\ge 0\]
and
\begin{equation}\label{eq::p10}
  z^{0,\pm}_i=0 \quad \mbox{and}\quad z^{k+1,\pm}_i=
\frac{\lambda_{k+1}-\lambda_k}{\pi^\pm_i(\lambda_{k+1})-\pi^\pm_i(\lambda_k)}
\quad \mbox{for all}\quad k\ge 0
\end{equation}
(recall that  $\pi^\pm_i$ is defined in \eqref{eq::21}).
We have in particular for all $k\ge 1$
\begin{equation}\label{eq::p11}
z_i^{k+1,-}<z_i^{k,-} < z_i^{0,-} =0=z_i^{0,+}< z_i^{k,+} < z_i^{k+1,+}.
\end{equation}
\item 
For $(x,y) \in J_i \times J_j$ with $i\not=j$, 
$$G^{0,\sharp}(x,y)=x\pi_i^{+}(\lambda_k)-y\pi^-_i(\lambda_k)-\lambda_k $$
for $(x,y)\in \Delta^k_{ij}$ with
\begin{equation}\label{eq::p7}
\Delta^k_{ij}=\left\{(x,y)\in J_i\times J_j,\quad \frac{x}{z^{k,+}_i}-\frac{y}{z^{k,-}_j}\ge 1,
\quad \frac{x}{z^{k+1,+}_i}-\frac{y}{z^{k+1,-}_j}\le 1\right\}.
\end{equation}
\end{itemize}
\end{lem}
\begin{proof}
  Remark that $\lambda_k=H_i(\pi^\pm_i(\lambda_k))$. Therefore the
  definition of $z_i^{k,\pm}$ and the convexity of $H_i$ imply
  inequalities (\ref{eq::p11}). It is then easy to check the explicit
  expressions of $G^{0,\sharp}$.
\end{proof}
We recall that if $u\in C^{1,\sharp}(J)$ and $u$ is not $C^1$
at a point $x\in J_i^*$, then Proposition~\ref{pro::gr42} can be used
in order to understand $H$ as follows
\begin{equation}\label{eq::p9}
H(x,u_x)=\max\left(H_i^+(\partial_iu(x^-)),H_i^-(\partial_iu(x^+))\right).
\end{equation}
This interpretation will be used to check inequality (\ref{eq::17bis})
at points where $G^{0,\sharp}(x,y)$ is not $C^1$ with $(x,y)\in
J_i\times J_j$ with $i\not= j$.

\begin{pro}[The second vertex test function -- the smooth convex case]\label{pro::p3}
  Let $A\ge A_0$ with $A_0$ given by \eqref{eq::A_0} and assume that
  the Hamiltonians satisfy \eqref{eq::16}.  Let $(\lambda_k)_{k\in
    \N}$ be any increasing sequence satisfying (\ref{eq::p8}) for some
  given $\gamma_0>0$.  Then the function $G^{0,\sharp}:J^2\to \R$
  defined in (\ref{eq::p5}) satisfies properties ii) and iv) listed in
  Proposition \ref{pro:vtf}, together with the following properties
\begin{enumerate}
\item[] \emph{i') (Regularity)}
\[G^{0,\sharp}\in C(J^2)\quad \text{and}\quad \left\{\begin{array}{l}
G^{0,\sharp}(x,\cdot)\in C^{1,\sharp}(J) \quad \text{for all}\quad x\in J,\\
G^{0,\sharp}(\cdot,y)\in C^{1,\sharp}(J) \quad \text{for all}\quad y\in J.
\end{array}\right.\]
\item[] \emph{iii') (Compatibility conditions)} On the one hand,
  \eqref{eq::85} holds with $\gamma=0$ for all $x\in J$.  On the other
  hand, \eqref{eq::17bis} holds with $\gamma=\gamma_0$, for all
  $(x,y)\in J^2$, except possibly for all points on the diagonals
  $x=y\in J_i^*$ for $i\in \left\{1,\dots,N\right\}$.

  Moreover, at points $(x,y)\in J_i\times J_j$ with $i\not= j$, where
  the functions $G^{0,\sharp}(x,\cdot)$ or $G^{0,\sharp}(\cdot,y)$ are
  not $C^1$, inequality \eqref{eq::17bis} has to be understood using
  convention (\ref{eq::p9});
\item[] \emph{v') (Gradient bounds)} Estimate~\eqref{eq::19} holds for
  all $(x,y) \in J^2$ if we understand it as a bound for both left and
  right derivatives, at points where the functions
  $G^{0,\sharp}(x,\cdot)$ and $G^{0,\sharp}(\cdot,y)$ are not $C^1$.
\end{enumerate}
\end{pro}

\begin{proof}
  The regularity i') follows immediately for the previous
  lemma. Moreover points ii) and iv) listed in Proposition
  \ref{pro:vtf} follow easily, and similarly for the gradient bounds
  v'). Also \eqref{eq::85} holds clearly for $\gamma=0$.

The only important point is to check inequality \eqref{eq::17bis} in
iii') with $\gamma=\gamma_0$.

\paragraph{Step 1: checking on $J_i^*\times J_i^*$}
Inequality \eqref{eq::17bis} is clearly true for $(x,y)\in J_i^*\times
J_i^*$, if $x-y\not= z^{k,\pm}_i$.  Let us check it if $x-y=
z^{k+1,\pm}_i\not= 0$. We distinguish two cases. \medskip

\noindent \textsc{Case 1: $(x,y)\in J_i^*\times J_i^*$ with $x-y=
  z^{k+1,+}_i>0$.}  The only novelty here is that the function
$G^{0,\sharp}$ is not $C^1$ at those points, and we have to use
interpretation (\ref{eq::p9}) to compute it. We get
\begin{equation}\label{eq::p20}
\begin{array}{ll}
H(x,G^{0,\sharp}_x(x,y))&=\max(H_i^+(G^{0,\sharp}_x(x^-,y)),H_i^-(G^{0,\sharp}_x(x^+,y)))\\
& = \max(H_i^+(\pi^+_i(\lambda_k)),H_i^-(\pi^+_i(\lambda_{k+1})))\\
& = \lambda_k
\end{array}
\end{equation}
and
\begin{equation}\label{eq::p21}
\begin{array}{ll}
H(y,-G^{0,\sharp}_y(x,y)) & = \max(H_i^+(-G^{0,\sharp}_y(x,y^-)),H_i^-(-G^{0,\sharp}_y(x,y^+)))\\
& = \max(H_i^+(\pi^+_i(\lambda_{k+1})),H_i^-(\pi^+_i(\lambda_{k})))\\
& = \lambda_{k+1}.
\end{array}
\end{equation}
This implies inequality \eqref{eq::17bis} for $\gamma=\gamma_0\ge
\lambda_{k+1}-\lambda_k$.\medskip

\noindent \textsc{Case 2: $(x,y)\in J_i^*\times J_i^*$ with $x-y= z^{k+1,-}_i<0$.}
We compute
\begin{equation}\label{eq::p22}
\begin{array}{ll}
H(x,G^{0,\sharp}_x(x,y))&=\max(H_i^+(G^{0,\sharp}_x(x^-,y)),H_i^-(G^{0,\sharp}_x(x^+,y)))\\
& = \max(H_i^+(\pi^-_i(\lambda_{k+1})),H_i^-(\pi^-_i(\lambda_{k})))\\
& = \lambda_k
\end{array}
\end{equation}
and
\begin{equation}\label{eq::p23}
\begin{array}{ll}
H(y,-G^{0,\sharp}_y(x,y)) & = \max(H_i^+(-G^{0,\sharp}_y(x,y^-)),H_i^-(-G^{0,\sharp}_y(x,y^+)))\\
& = \max(H_i^+(\pi^-_i(\lambda_{k})),H_i^-(\pi^-_i(\lambda_{k+1})))\\
& = \lambda_{k+1}
\end{array}
\end{equation}
which gives the result.

\paragraph{Step 2: checking on $\Delta^k_{ij}$ for $i\not=j$.}
This inequality is also obviously true if $(x,y)\in \mbox{Int}\
\Delta^k_{ij}$ for $i\not=j$. We then distinguish six cases.
\medskip

\noindent {\sc Case 1: $x=y=0$.}
This case is similar to the study of $G^0$ and we get immediately 
$$H(0,-G^{0,\sharp}_y(0,0))=-A=H(0,G^{0,\sharp}_x(0,0)).$$

\noindent {\sc Case 2: $(x,y)\in \Delta^k_{ij}$ with $y=0$ and $z^{k,+}_i<x<z^{k+1,+}_i$.}
$$H(0,-G^{0,\sharp}_y(x,0))=\lambda_k=H(x,G^{0,\sharp}_x(x,0)).$$

\noindent {\sc Case 3: $(x,y)\in \Delta^k_{ij}$ with $x=0$ and $-z^{k,-}_j<y<-z^{k+1,-}_j$.}
$$H(y,-G^{0,\sharp}_y(0,y))=\lambda_k=H(0,G^{0,\sharp}_x(0,y)).$$

\noindent {\sc Case 4: $(x,y)\in (\partial \Delta^k_{ij})\backslash
  \left((J_i\times \left\{0\right\})\cup (\left\{0\right\}\times
    J_j)\right)$.}
Let us consider the subcase where $\displaystyle \frac{x}{z^{k+1,+}_i}-\frac{y}{z^{k+1,-}_j}= 1$ 
(the other case with $k+1$ replaced by $k$ being of course similar).
We compute again:
$$\begin{array}{ll}
H(x,G^{0,\sharp}_x(x,y))&=\max(H_i^+(G^{0,\sharp}_x(x^-,y)),H_i^-(G^{0,\sharp}_x(x^+,y)))\\
& = \max(H_i^+(\pi^+_i(\lambda_k)),H_i^-(\pi^+_i(\lambda_{k+1})))\\
& = \lambda_k
\end{array}$$
and
$$\begin{array}{ll}
H(y,-G^{0,\sharp}_y(x,y)) & = \max(H_j^+(-G^{0,\sharp}_y(x,y^-)),H_j^-(-G^{0,\sharp}_y(x,y^+)))\\
& = \max(H_j^+(\pi^-_j(\lambda_k)),H_j^-(\pi^-_j(\lambda_{k+1})))\\
& = \lambda_{k+1}.
\end{array}$$
This implies again inequality \eqref{eq::17bis} for
$\gamma=\gamma_0\ge \lambda_{k+1}-\lambda_k$. \medskip

\noindent {\sc Case 5: $(x,y)\in \Delta^k_{ij}$ with $y=0$ and $x=z^{k+1,+}_i$.}
Again, we check easily that $H(0,-G^{0,\sharp}_y(x,0))=\lambda_{k+1}$,
and $H(x,G^{0,\sharp}_x(x,0))=\lambda_k$, as in Case~4.
\medskip

\noindent {\sc Case 6: $(x,y)\in \Delta^k_{ij}$ with $x=0$ and $y=-z^{k+1,-}_j$.}
We have $H(y,-G^{0,\sharp}_y(0,y))=\lambda_{k+1}$ as in Case~4, and
$H(0,G^{0,\sharp}_x(0,y))=\lambda_k$.
\end{proof}

\subsubsection*{The general case}

Then we have the following
\begin{theo}[The second vertex test function]\label{th::p1}
  Let $A\in \R\cup \left\{-\infty\right\}$ and $\gamma>0$. Assume that
  the Hamiltonians satisfy \eqref{assum:H} and \eqref{eq::24}.  Then
  there exists a function $G^\sharp:J^2\to \R$ enjoying properties ii)
  to vi) listed in Theorem \ref{th::G},
  and property i') given in Proposition \ref{pro::p3}.\\
  In particular, at points (different from the origin) where functions
  $G^\sharp(x,\cdot)$ and $G^\sharp(\cdot,y)$ are not $C^1$, we get
  bounds (\ref{eq::19}) on both left and right derivatives. Moreover,
  at those points, inequality (\ref{eq::17bis}) has to be interpreted
  in the sense of Proposition \ref{pro::gr42}.  Moreover, there exists
  some $\varepsilon>0$ such that we have
\begin{equation}\label{eq::p13}
  G^\sharp = G^{0,\sharp} \quad \mbox{on}\quad J^2 \backslash
  \delta_\varepsilon 
\quad \mbox{with}\quad \delta_\varepsilon=
  \left\{(x,y)\in \bigcup_{i=1,\dots,N}J_i^*\times J_i^*,\quad  |x-y|\le \varepsilon\right\}
\end{equation}
where $G^{0,\sharp}$ is given in Proposition \ref{pro::p3}, with $\gamma=\gamma_0$.
\end{theo}
\begin{proof}[Proof of Theorem \ref{th::p1}]
  In the smooth convex case, we define $G^\sharp$ as in
  (\ref{eq::p13}). On $J_i^*\times J_i^*$, we simply define $G^\sharp$
  as a regularization of $G^{0,\sharp}$ along each line
  $x=y\in J_i^*$, following the
  procedure described in the proof of Lemma \ref{lem:case-convex} for
  $\varepsilon\le \gamma=\gamma_0$.  The general case follows by
  approximation.  
\end{proof}

\begin{rem}\label{rem::p14}
  With the help of Proposition \ref{pro::gr42}, it is straighforward
  to check that the proof of the comparison principle works as well
  with this second vertex test function $G^\sharp$ given by Theorem
  \ref{th::p1}.
\end{rem}

\section{Extension to networks}
\label{s.n}

\subsection{Definition of a network}

A general abstract network ${\mathcal N}$ is characterized by the set
${\mathcal E}$ of its \emph{edges} and the set ${\mathcal V}$ of
its vertices (or nodes). It is endowed with a distance.  

\paragraph{Edges.}
${\mathcal E}$ is a finite or countable set of edges.  Each edge $e\in
{\mathcal E}$ is assumed to be either isometric to the half line
$[0,+\infty)$ with $\partial e = \left\{e^0\right\}$ (where the
endpoint $e^0$ can be identified to $\left\{0\right\}$), or to a
compact interval $[0,l_e]$ with 
\begin{equation}
  \label{eq:le-inf} 
  \inf_{e \in \mathcal{E}} l_e >0 
\end{equation}
and $\partial e= \left\{e^0,e^1\right\}$. Condition~\eqref{eq:le-inf}
implies in particular that the network is complete. The endpoints $\{e^0\},
\{e^1\}$ can respectively be identified to $\left\{0\right\}$ and
$\left\{l_e \right\}$.  The \emph{interior} $e^*$ of an edge $e$
refers to $e \setminus (\partial e)$.

\paragraph{Vertices.}
It is convenient to see vertices of the network as a partition of the
sets of all edge endpoints, 
\[\bigcup_{e\in {\mathcal E}} \partial e  = \bigcup_{n\in {\mathcal V}} n;\]
we assume that each set $n$ only contains a finite number of
endpoints.

Here each $n\in {\mathcal V}$ can be identified as a vertex (or node)
of the network as follows.  For every $x,y\in \bigcup_{e\in {\mathcal
    E}} e$, we define the equivalence relation:
\[x\sim y \quad \Longleftrightarrow \left(x=y\quad \text{or}\quad 
x,y\in n \in {\mathcal V}\right)\]
and we define the network as the quotient 
\begin{equation}\label{eq::N}
{\mathcal N} = \left(\bigcup_{e\in {\mathcal E}} e\right)/ \sim 
\quad = \quad \left(\bigcup_{e\in {\mathcal E}} e^*\right) \cup {\mathcal V}.
\end{equation}

We also define for $n\in {\mathcal V}$
\[{\mathcal E}_n =\left\{e\in {\mathcal E},\quad n\in \partial e \right\}\]
and its partition ${\mathcal E}_n = {\mathcal E}_n^-\cup {\mathcal E}_n^+$ with
\[{\mathcal E}_n^-=\left\{e\in {\mathcal E}_n, n=e^0\right\},\quad 
{\mathcal E}_n^+=\left\{e\in {\mathcal E}_n, n=e^1\right\}.\]

\paragraph{Distance.}
We also define the distance function $d(x,y)=d(y,x)$ as the minimal
length of a continuous path connecting $x$ and $y$ on the network,
using the metric of each edge (either isometric to $[0,+\infty)$ of to
a compact interval). Note that, because of our assumptions, if
  $d(x,y)<+\infty$, then there is only a finite number of minimal
  paths.

\begin{rem}\label{rem:bd-no-path}
For any $\eps>0$, there is a bound (depending on
$\eps$) on the number of minimal paths connecting $x$ to $y$
for all $y\in B(\bar y, \eps)=\left\{y\in {\mathcal N},\quad
d(\bar y, y) < \eps\right\}$.
\end{rem}

\subsection{Hamilton-Jacobi equations on a network}

Given a Hamiltonian $H_e$ on each edge $e \in \mathcal{E}$, we
consider the following HJ equation on the network $\mathcal{N}$,
\begin{equation}\label{eq::1terbis}
\left\{\begin{array}{lll}
u_t + H_e(t,x,u_x)= 0  &\text{for}\quad t\in (0,+\infty) 
&\quad \text{and}\quad x\in e^*,\\
u_t + F_A(t,x,u_x)=0   &\text{for}\quad  t\in (0,+\infty) 
&\quad \text{and}\quad x=n\in {\mathcal V}
\end{array}\right.
\end{equation}
supplemented with  an initial condition
\begin{equation}\label{eq::l15}
u(0,x)=u_0(x) \quad \text{for}\quad x\in {\mathcal N}. 
\end{equation}
The limited flux functions $F_A$ associated with the Hamiltonians
$H_e$ are defined below. We first make precise the meaning of
$u_x$ in \eqref{eq::1terbis}.

\paragraph{Gradients of real functions.} 
For a real function $u$ defined on the network ${\mathcal N}$, we
denote by $\partial_e u(x)$ the (spatial) derivative of $u$ at $x\in
e$ and define the ``gradient'' of $u$ by
\[u_x(x):=\begin{cases}
\partial_e u(x) & \quad \text{if} \quad x\in e^*= e \setminus (\partial e),\\
((\partial_e u(x))_{e\in {\mathcal E}_n^-},(\partial_e u(x))_{e\in
  {\mathcal E}_n^+}) & \quad \text{if} \quad x=n\in {\mathcal V}
\end{cases}.\]
The norm $|u_x|$ simply denotes $|\partial_e u|$ for $x\in e^*$ or
$\max \{ |\partial_e u|: e\in {\mathcal E}_n\}$ at the vertex $x=n$.

\paragraph{Limited flux functions.} 
We also define for $(t,x)\in \R\times \partial e$,
\[H^-_e(t,x,q)=\left\{\begin{array}{ll}
H_e(t,x,q) & \quad \text{if}\quad q\le p^0_e(t,x),\\
H_e(t,x,p^0_e(t,x)) & \quad \text{if}\quad q > p^0_e(t,x)
\end{array}\right.\]
and
\[H^+_e(t,x,q)=\left\{\begin{array}{ll}
H_e(t,x,p^0_e(t,x)) & \quad \text{if}\quad q\le p^0_e(t,x),\\
H_e(t,x,q) & \quad \text{if}\quad q > p^0_e(t,x).
\end{array}\right.\]
Given limiting functions $(A_n)_{n\in {\mathcal V}}$, we define
 for $p=(p_e)_{e\in {\mathcal E}_n}$, 
\[F_{A}(t,n,p)=\max\left(A_n(t),\quad\max_{e\in {\mathcal E}_n^-}
H^-_e(t,n,p_e), \quad \max_{e\in {\mathcal E}_n^+} H^+_e(t,n,p_e)\right).\]
In particular, for each $n\in {\mathcal V}$, the functions 
$F_A(t,n,\cdot)$ are the same for all $A_n(t)\in [-\infty, A_n^0(t)]$ with
\begin{equation}\label{eq::l21}
A_n^0(t):= \max\left(\max_{e\in {\mathcal E}_n^-} H^-_e(t,n,p^0_e(t,n)),
\quad \max_{e\in {\mathcal E}_n^+} H^+_e(t,n,p^0_e(t,n))\right).
\end{equation}

\paragraph{A shorthand notation.}
As in the junction case, we introduce
\begin{equation}\label{eq::l24bis}
H_{\mathcal N}(t,x,p)=\left\{\begin{array}{llll}
H_e(t,x,p) &\quad \text{for}\quad p\in\R, 
&\quad t\in\R, &\quad \text{if}\quad x\in e^*,\\
F_A(t,x,p) &\quad \text{for}\quad 
p=(p_e)_{e\in {\mathcal E}_n}\in \R^{\text{\small Card}\ {\mathcal E}_n}, 
&\quad t\in\R,  &\quad \text{if}\quad x=n\in {\mathcal V}
\end{array}\right.
\end{equation}
in order to rewrite \eqref{eq::1terbis}  as
\begin{equation}\label{eq::l16}
u_t + H_{\mathcal N}(t,x, u_x) =0 \quad \text{for all}\quad (t,x)\in 
(0,+\infty)\times {\mathcal N}.
\end{equation}

\subsection{Assumptions on the Hamiltonians}

For each $e\in {\mathcal E}$, we consider a Hamiltonian $H_e:
[0,+\infty) \times e\times \R \to \R$ satisfying
\begin{itemize}
\item \textbf{(H0)} (Continuity) $H_e \in C([0,+\infty) \times e\times
  \R)$.
\item \textbf{(H1)} (Uniform coercivity) For all $T>0$,  
\[\lim_{|q|\to +\infty} H_e(t,x,q)=+\infty\]
uniformly with respect to $t \in [0,T]$ and $x \in e \in
\mathcal{E}$.
\item \textbf{(H2)} (Uniform bound on the Hamiltonians for bounded gradients) 
For all $T,L>0$, there exists $C_{T,L}>0$ such that
\[ \sup_{t\in [0,T],\ p\in [-L,L], x \in \mathcal{N} \setminus
  \mathcal{V}}  |H_{\mathcal{N}}(t,x,p)| \le C_{T,L}.\]
\item \textbf{(H3)} (Uniform modulus of continuity for bounded gradients)
For all $T,L>0$, there exists a modulus of continuity $\omega_{T,L}$
such that for all $|p|,|q|\le L$, $t\in [0,T]$ and $x\in e\in
{\mathcal E}$, 
\[|H_e(t,x,p)-H_e(t,x,q)|\le \omega_{T,L}(|p-q|).\]
\item \textbf{(H4)} (Quasi-convexity) For all $n \in \mathcal{V}$,
  there exists a (possibly discontinuous) function $t
  \mapsto p^0_e (t,n)$ such that
\[\begin{cases}
H_e(t,n,\cdot) \quad \text{is nonincreasing on}\quad (-\infty,p^0_e(t,n)],\\
H_e(t,n,\cdot) \quad \text{is nondecreasing on}\quad [p^0_e(t,n),+\infty).
\end{cases}\]
\item \textbf{(H5)} (Uniform modulus of continuity in time)
For all $T>0$, there exists a modulus of continuity $\bar\omega_{T}$
such that for all $t,s\in [0,T]$, $p\in \R$, $x\in e\in {\mathcal E}$,
\[H_e(t,x,p)-H_e(s,x,p) \le  \bar\omega_{T}
\left( |t-s|(1+\max(H_e(s,x,p),0)) \right). \]
\item \textbf{(H6)} (Uniform continuity of $A^0$) 
For all $T>0$, there exists a modulus of continuity $\bar\omega_{T}$
such that for all $t,s\in [0,T]$ and $n\in {\mathcal V}$,
\[|A^0_n(t)-A^0_n(s)|\le \bar\omega_{T}(|t-s|).\]
\end{itemize}
As far as flux limiters are concerned, the following assumptions will
be used.
\begin{itemize}
\item \textbf{(A0)} (Continuity of $A$) For all $T>0$ and $n \in \mathcal{V}$,
  $A_n \in C([0,T])$.
\item \textbf{(A1)} (Uniform bound on $A$)
For all $T>0$, there exists a constant $C_T>0$ such that for all $t\in
[0,T]$ and $n\in {\mathcal V}$
\[|A_n(t)|\le C_T .\]
\item \textbf{(A2)} (Uniform continuity of $A$) For all $T>0$, there
  exists a modulus of continuity $\bar\omega_{T}$ such that for all
$t,s\in [0,T]$ and $n\in {\mathcal V}$,
\[|A_n(t)-A_n(s)|\le \bar\omega_{T}(|t-s|).\]
\end{itemize}
The proof of the following technical lemma is postponed until
appendix. 
\begin{lem}[Estimate on the difference of Hamiltonians]\label{lem::l146}
 Assume that the Hamiltonians satisfy \emph{(H0)-(H4)} and
 \emph{(A0)-(A1)}. Then for all $T>0$, there exists a constant $C_T>0$
 such that
\begin{eqnarray}
\label{eq::l136bis}
|p^0_e(t,x)| & \le C_T &  \quad \text{for all}
\quad t\in [0,T],\quad x\in \partial e, \quad e\in {\mathcal E}, \\
\label{eq::l156}
|A^0_n(t)| & \le C_T & \quad \text{for all}
\quad t\in [0,T],\quad n\in {\mathcal V}.
\end{eqnarray}
If we assume moreover \emph{(H5)-(H6)} and \emph{(A2)}, then there
exists a modulus of continuity $\tilde{\omega}_T$ such that for all
$t,s \in [0,T]$, and $x,p$
\begin{equation}\label{eq::l139}
 H_{\mathcal N}(t,x,p)-H_{\mathcal N}(s,x,p)\le 
\tilde{\omega}_T(|t-s|(1+\max(0,H_{\mathcal N}(s,x,p))).
\end{equation}
\end{lem}
\begin{rem}
From the proof, the reader can check that Assumptions (H5)-(H6) and (A2) in the
statement of Theorem~\ref{th::l2} can in fact be replaced with
\eqref{eq::l139}.
\end{rem}
\begin{rem}[Example of Hamiltonians with uniform modulus of time continuity]
Condition on the uniform modulus of continuity in time in (H5)
is for instance satisfied by Hamiltonians of the type for $q>0$ and $\delta>0$
such that for all $x\in e \in {\mathcal E}$ we have
\[
H_e(t,x,p)=c_e(t,x)|p|^q \quad \text{with}\quad 0<\delta \le
c_e(t,x)\le 1/\delta\]
with $c_e$ Lipschitz continuous in time and continuous in space.
\end{rem}

\subsection{Viscosity solutions on a network}

\paragraph{Class of test functions.}
For $T>0$, set ${\mathcal N}_T= (0,T)\times {\mathcal N}$. We define
the class of test functions on $(0,T)\times {\mathcal N}$ by
\[C^1({\mathcal N}_T)=\left\{\varphi\in C({\mathcal N}_T),\; 
\text{the restriction of $\varphi$ to $(0,T)\times e$ is $C^1$, 
for all $e\in {\mathcal E}$}\right\}.\]
\begin{defi}[Viscosity solutions]\label{defi::l1}
Assume the Hamiltonians satisfy \emph{(H0)-(H4)} and \emph{(A0)-(A1)}
and let $u:[0,T)\times {\mathcal N}\to \R$.
\begin{enumerate}[i)]
\item We say that $u$ is a \emph{sub-solution}
  (resp. \emph{super-solution}) of \eqref{eq::1bis} in $(0,T)\times
  {\mathcal N}$ if for all test function $\varphi\in
  C^1({\mathcal N}_T)$ such that
\[u^*\le \varphi \quad (\text{resp.}\quad u_*\ge \varphi) 
\quad \text{in a neighborhood of $(t_0,x_0)\in {\mathcal N}_T$}\]
with equality at $(t_0,x_0)$, we have
\[\varphi_t + H_{\mathcal N}(t,x,\varphi_x) \le 0 
\quad (\text{resp.}\quad \ge 0) \quad \text{at $(t_0,x_0)$}.\]
\item We say that $u$ is a \emph{sub-solution}
  (resp. \emph{super-solution}) of \eqref{eq::1bis}, \eqref{eq::2} in 
$[0,T)\times {\mathcal N}$ if additionally
\[u^*(0,x) \le u_0(x) \quad 
(\text{resp.}\quad u_*(0,x) \ge u_0(x))\quad \text{for all}
\quad x\in {\mathcal N}.\]
\item 
We say that $u$ is a \emph{(viscosity) solution} if $u$ 
is both a sub-solution and a super-solution.
\end{enumerate}
\end{defi}
\begin{rem}[Touching sub-solutions with semi-concave functions] 
\label{rem:sc-net}
When proving the comparison principle in the network setting,
sub-solutions (resp. super-solutions) will be touched from above
(resp. from below) by functions that will not be $C^1$, but only
semi-concave (resp. semi-convex).  We recall that a function is
semi-concave if it is the sum of a concave function and a smooth
($C^2$ say) function.  But it is a classical observation that, at a
point where a semi-concave function is not $C^1$, we can replace the
semi-concave function by a $C^1$ test function touching it from
above.
\end{rem}
As in the case of a junction (see Proposition~\ref{pro::2}), viscosity
solutions are stable through supremum/infimum.  We also have the
following existence result.
\begin{theo}[Existence on a network]\label{th::l3}
  Assume \emph{(H0)-(H4)} and \emph{(A0)-(A1)} on the Hamiltonians and
  assume that the initial data $u_0$ is uniformly continuous on
  ${\mathcal N}$.  Let $T>0$.  Then there exists a viscosity solution
  $u$ of \eqref{eq::l16},\eqref{eq::l15} on $[0,T)\times {\mathcal N}$
    and a constant $C_T>0$ such that
\[|u(t,x)-u_0(x)|\le C_T \quad \text{for all}\quad 
(t,x)\in [0,T)\times {\mathcal N}.\]
\end{theo}
\begin{proof}
The proof follows along the lines of the ones of Theorem~\ref{th::2}.
The main difference lies in the construction of barriers.  We proceed
similarly and get a regularized initial data $u_0^\eps$
satisfying
\[|u_0^\eps - u_0|\le \eps \quad \text{and}\quad
|(u_0^\eps)_x|\le L_\eps.\]
Then the functions 
\begin{equation}\label{eq::35bis}
u_\eps^\pm(t,x)=  u_0^\eps(x) \pm C_\eps t \pm \eps 
\end{equation}
are global super and sub-solutions with respect to the initial data
$u_0$ if $C_\eps$ is chosen as follows, 
\begin{equation}\label{eq::l114}
C_\eps = \max \left(\sup_{t\in [0,T]}\sup_{n\in {\mathcal V}}|
\max(A_n(t),A_n^0(t))|, \sup_{t\in [0,T]}\sup_{e\in {\mathcal E}} 
\sup_{x\in e,\ |p_e|\le L_\eps} |H_e(t,x,p_e)|\right);
\end{equation}
indeed, we use \eqref{eq::l156} in Lemma~\ref{lem::l146} to bound the first
terms in \eqref{eq::l114}. 
\end{proof}

\subsection{Comparison principle on a network}

\begin{theo}[Comparison principle on a network]\label{th::l2}
Assume the Hamiltonians satisfy \emph{(H0)-(H6)} and \emph{(A0)-(A2)}
and assume that the initial data $u_0$ is uniformly continuous on
${\mathcal N}$.  Let $T>0$.  Then for all sub-solution $u$ and
super-solution $w$ of \eqref{eq::l16}, \eqref{eq::l15} in $[0,T)\times
  {\mathcal N}$, satisfying for some $C_T>0$ and some $x_0\in
  {\mathcal N}$
\begin{equation}\label{eq::l27}
u(t,x)\le C_T (1+ d(x_0,x)),\quad w(t,x)\ge -C_T(1+d(x_0,x)),
\quad \text{for all}\quad (t,x) \in [0,T)\times {\mathcal N},
\end{equation}
we have
\[u\le w \quad \text{on}\quad [0,T)\times {\mathcal N}.\]
\end{theo}
As a straighforward corollary of Theorems~\ref{th::l2} and \ref{th::l3}, we get
\begin{cor}[Existence and uniqueness] \label{cor::l4}
Under the assumptions of Theorem~\ref{th::l2}, there exits a unique
viscosity solution $u$ of \eqref{eq::l16}, \eqref{eq::l15} in
$[0,T)\times {\mathcal N}$ such that there exists a constant $C>0$
  with
\[|u(t,x)-u_0(x)|\le C \quad \text{for all}\quad 
(t,x)\in [0,T)\times {\mathcal N}.\]
\end{cor}
In order to prove Theorem~\ref{th::l2}, we first need two technical
lemmas that are proved in appendix.
\begin{lem}[A priori control -- the network case]\label{lem::l3}
Let $T>0$ and let $u$ be a sub-solution and $w$ be a super-solution as
in Theorem~\ref{th::l2}. Then there exists a constant $C=C(T)>0$ such
that for all $(t,x),(s,y)\in [0,T)\times \mathcal{N}$, we have
\begin{equation}\label{eq::l29}
u(t,x)\le w(s,y) + C(1 + d(x,y)).
\end{equation}
\end{lem}
\begin{lem}[Uniform control by the initial data] \label{lem::l110}
Under the assumptions of Theorem~\ref{th::l2}, for any $T>0$ and
$C_T>0$, there exists a modulus of continuity $f:[0,T)\to [0,+\infty]$
  satisfying $f(0^+)=0$ such that for all sub-solution $u$
  (resp. super-solution $w$) of \eqref{eq::l16}, \eqref{eq::l15} on
  $[0,T)\times {\mathcal N}$, satisfying \eqref{eq::l27} for some
    $x_0\in {\mathcal N}$, we have for all $(t,x)\in [0,T)\times
      {\mathcal N}$,
\begin{equation}\label{eq::l111}
u(t,x)\le u_0(x) + f(t) \quad \left(\text{resp.}\quad w(t,x)\ge
u_0(x)-f(t)\right).
\end{equation}
\end{lem}
We can now turn to the proof of Theorem~\ref{th::l2}.  The proof is
similar the comparison principle on a junction
(Theorem~\ref{th::2}). Still, a space localization procedure has to be
performed in order to ``reduce'' to the junction case.  From a
technical point of view, a noticeable difference is that we will fix
the time penalization (for some parameter $\nu$ small enough), and
then will first take the limit $\eps\to 0$ ($\eps$ being the parameter
for the space penalization), and then take the limit $\alpha\to 0$
($\alpha$ being the penalizaton parameter to keep the optimization
points at a finite distance).
\begin{proof}[Proof of Theorem~\ref{th::l2}]
Let $\eta >0$ and $\theta>0$ and consider
\[M(\theta)=\sup\left\{u(t,x)-w(s,x)-\frac{\eta}{T-t},
\quad x\in{\mathcal N}, \quad t,s\in [0,T),\quad |t-s|\le \theta\right\}.\]
We want to prove that
\[M=\lim_{\theta \to 0} M(\theta) \le 0.\]
Assume by contradiction that $M>0$. 
From Lemma~\ref{lem::l3} we  know that $M$ is finite.

\paragraph{Step 1: The localization procedure.}
Let $\psi$ denote $\frac{d^2(x_0,\cdot)}{2}$.
\begin{lem}[Localization] \label{lem:localization}
The supremum  
\[M_{\alpha}=\sup_{\stackrel{t,s\in [0,T], t < T}{x\in {\mathcal N}}}\left\{u(t,x)-w(s,x)
-\alpha\psi(x)-\frac{\eta}{T-t}-\frac{(t-s)^2}{2\nu}\right\}\] is
reached for some point $(t_\alpha,s_\alpha,x_\alpha)$. Moreover, for
$\alpha$ and $\nu$ small enough, we have the following localization
estimates
\begin{align}
\label{eq::bfb}
M_\alpha \ge 3M/4>0\\
\label{eq::l117}
d(x_0,x_\alpha)\le \frac{C}{\sqrt{\alpha}} \\
\label{eq::l120} 
0 < \tau_\nu \le t_\alpha, s_\alpha  \le T - \frac{\eta}{2C} \\
\label{eq::l155}
\lim_{\nu\to 0} \left(\limsup_{\alpha\to 0}
\frac{(t_\alpha-s_\alpha)^2}{2\nu}\right)  =0
\end{align}
where $C$ is a constant which does not depend on $\alpha$, $\eps$,
 $\nu$ and $\eta$. 
\end{lem}
\begin{proof}[Proof of Lemma~\ref{lem:localization}]
Choosing $\alpha$ small enough, we have \eqref{eq::bfb} for all $\nu
>0$. Because the network is complete for its metric, the supremum in
the definition of $M_\alpha$ is reached at some point
$(t_\alpha,s_\alpha,x_\alpha)$. From Lemma~\ref{lem::l3}, we deduce
that 
\[0<\frac{3M}{4} \le M_\alpha 
\le C
-\alpha\psi(x_\alpha)-\frac{\eta}{T-t_\alpha}-\frac{(t_\alpha-s_\alpha)^2}{2\nu}\]
and then
\begin{equation}\label{eq::l130}
\alpha\psi(x_\alpha)+\frac{\eta}{T-t_\alpha}+\frac{(t_\alpha-s_\alpha)^2}{2\nu}\le C. 
\end{equation}
This implies \eqref{eq::l117} changing $C$ if necessary.

On the one hand, we get from \eqref{eq::l130} the second inequality in
\eqref{eq::l120} by choosing $\nu$ such that $\sqrt{2 \nu C} \le \eta
/2C$. On the other hand, we get from Lemma~\ref{lem::l110}  
\[0< M_\alpha  \le  f(t_\alpha)+f(s_\alpha) -\frac{\eta}{T}.\]
In particular,
\[\frac{\eta}{T}\le 2 f(\tau+ \sqrt{2 \nu C})\]
where $\tau = \min(t_\alpha,s_\alpha)$. 
If both $\tau$ and $\nu$ are too small, we get a contradiction. Hence
the first inequality in \eqref{eq::l120} holds for some constant
$\tau_\nu$ depending on $\nu$ but not on $\alpha$, $\eps$ and $\eta$.

We now turn to the proof of \eqref{eq::l155}. We know that for any
$\delta>0$, there exists $\theta(\delta)>0$ (with $\theta(\delta)\to
0$ as $\delta\to 0$) and $(t^\delta,s^\delta,x^\delta)\in [0,T)\times
  [0,T)\times {\mathcal N}$ such that
\[u(t^\delta,x^\delta)-w(s^\delta,x^\delta)
-\frac{\eta}{T-t^\delta}\ge M-\delta \quad \text{and}\quad 
|t^\delta-s^\delta|\le \theta(\delta).\]
Then from \eqref{eq::l130} we deduce that
\[M(\sqrt{2\nu C})-\frac{(t_\alpha-s_\alpha)^2}{2\nu}\ge M_\alpha\ge 
M-\delta
-\alpha\psi(x^\delta)-\frac{|\theta(\delta)|^2}{2\nu}\]
and then
\[\limsup_{\alpha\to 0} \frac{(t_\alpha-s_\alpha)^2}{2\nu} \le
M(\sqrt{2\nu C}) - M + \delta  + \frac{|\theta(\delta)|^2}{2\nu}.\]
Taking the limit $\delta\to 0$, we get
\[\limsup_{\alpha\to 0} \frac{(t_\alpha-s_\alpha)^2}{2\nu} \le M(\sqrt{2\nu C}) - M\]
which yields the desired result. 
\end{proof}

\paragraph{Step 2: Reduction when $x_\alpha$ is a vertex.} We adapt here
Lemma~\ref{pi0zero}.
\begin{lem}[Reduction]\label{lem:pi0zero-net}
Assume that $x_\alpha = n \in \mathcal{V}$. Without loss of
generality, we can assume that ${\mathcal E}_n^+=\emptyset$ and
$p^0_e(t_\alpha,x_\alpha)=0$ for each $e\in {\mathcal E}_n$ with
$n=x_\alpha$.
\end{lem}
\begin{proof}[Proof of Lemma~\ref{lem:pi0zero-net}] 
The orientation of the edges $e\in {\mathcal E}_n$ can be changed in
order to reduce to the case ${\mathcal E}_n^+=\emptyset$. In
particular, for $p=(p_e)_{e\in {\mathcal E}_n}$,
\[F_A(t,n,p)=\max\left(A_n(t),\quad\max_{e\in {\mathcal E}_n^-} H^-_e(t,n,p_e)\right).\]
We can then argue as in Lemma~\ref{pi0zero}.  This means that we
redefine the Hamiltonians (and the flux limiter $A_n$) only locally
for $e\in {\mathcal E}_{n}$.  Using
\eqref{eq::l136bis}, we can check that the new Hamiltonians (locally
for $e\in {\mathcal E}_{n}$) and $A_n$ still satisfy \emph{(H0)-(H6)}
and \emph{(A0)-(A2)} (with the same modulus of continuity, and with
some different controlled constants $C_{T,L}$).  We also have
\eqref{eq::l27} with some controlled different constants.
\end{proof}

\paragraph{Step 3: The penalization procedure.}
We now consider for $\eps>0$ and $\gamma\in (0,1)$
\begin{multline*}
{M}_{\alpha,\eps}=\sup_{\stackrel{(t,x),(s,y)\in [0,T]\times 
 \overline{B(x_\alpha,r)}}{t < T}} \left\{u(t,x)-w(s,y)
-\alpha\psi(x)-\frac{\eta}{T-t} \right.\\
 \left. -\frac{(t-s)^2}{2\nu}-G^{\alpha,\gamma}_\eps(x,y)
-\varphi^{\alpha}(t,s,x)\right\}
\end{multline*}
where the function $\varphi^\alpha$ 
\[\varphi^{\alpha}(t,s,x)= \frac{1}{2}\left(|t-t_\alpha|^2+ |s-s_\alpha|^2+
d^2(x,x_\alpha)\right)\] will help us to localize the problem around
$(t_\alpha,s_\alpha,x_\alpha)$, and $B(x_\alpha,r)$ is the open ball
of radius $r=r(\alpha)>0$ centered at $x_\alpha$; besides, we choose
$r\in(0,1)$ small enough such that $B(x_\alpha,r)\subset e$ if
$x_\alpha \in e \setminus {\mathcal V}$.  Lemma~\ref{lem::l20} ensures
that $\psi$ and $\varphi^\alpha$  are semi-concave and therefore can
be used as test functions, see Remark~\ref{rem:sc-net}. 

We choose
\[G_\eps^{\alpha,\gamma}(x,y)=\eps 
G^{\alpha,\gamma}(\eps^{-1}x,\eps^{-1}y)\]
with
\[G^{\alpha,\gamma}(x,y)=\begin{cases}
\displaystyle \frac{(x-y)^2}{2} &\quad \text{if}\quad 
x_\alpha\in {\mathcal N}\setminus {\mathcal V},\\
G^{x_\alpha,\gamma}(x,y) &\quad \text{if}\quad x_\alpha\in {\mathcal V},
\end{cases}\]
where $G^{x_\alpha,\gamma}\ge 0$ is the vertex test function of
parameter $\gamma>0$ given by Theorem~\ref{th::G}, built on the
junction problem associated to the vertex $x_\alpha$ at time
$t_\alpha$, \textit{i.e.} associated to junction problem for the
Hamiltonian $H^{t_\alpha,x_\alpha}_{\mathcal V}$ given by
\begin{equation}\label{eq::l127}
H^{t_\alpha,n}_{\mathcal V}(x,p):=\left\{\begin{array}{ll}
H_e(t_\alpha,n,p) & \quad \text{if}\quad x\in e\setminus \left\{n\right\} \quad \text{with}\quad e\in {\mathcal E}_n,\\
F_{A}(t_\alpha,n,p) & \quad \text{if}\quad x=n.
\end{array}\right.
\end{equation}
The supremum in the definition of $M_{\alpha,\eps}$ is reached at some
point $(t,x),(s,y)\in [0,T]\times \overline{B(x_\alpha,r)}$ with
$t<T$. These maximizers satisfy the following penalization estimates.
\begin{lem}[Penalization]\label{lem:penal}
For $\eps \in (0,1)$ and $\gamma \in (0,M/4)$, we have
\begin{align}
\label{eq::l116}
{M}_{\alpha,\eps}\ge M_\alpha - \eps \gamma\ge M/2>0\\
\label{eq::l42}
d(x,y) \le \omega(\eps) \\
\nonumber
0 < \tau_\nu \le s,t \le T -\sigma_\eta
\end{align}
for some modulus of continuity $\omega$ (depending on $\alpha$ and
$\gamma$) and $\tau_\nu$ and $\sigma_\eta$ not depending on
$(\eps,\gamma)$. Moreover,
\[(t,s,x,y) \to (t_\alpha,s_\alpha,x_\alpha,x_\alpha) \quad \text{ as }
(\eps,\gamma)\to (0,0).\]
In particular, we have $x,y\in B(x_\alpha,r)$ for
$\eps,\gamma>0$ small enough.
\end{lem}
\begin{proof}[Proof of Lemma~\ref{lem:penal}]
For all $\eps,\nu >0$, the compatibility on the diagonal
\eqref{eq::85} of the vertex test function $G^{x_\alpha,\gamma}$
yields the first inequality in \eqref{eq::l116}. Then for $\eps
\in (0,1]$, with a choice of $\gamma$ such that $0<\gamma <M/4$, we
have the second one.

\paragraph{Bound on $d(x,y)$.}
Remark that 
\[\eps g\left(\frac{d(x,y)}{\eps}\right)\le 
G_\eps^{x_\alpha,\gamma}(x,y)\]
where
\[g(a)=\begin{cases}
\displaystyle \frac{a^2}{2} &\quad \text{if}\quad x_\alpha\in 
{\mathcal N}\setminus {\mathcal V},\\
g^{x_\alpha,\gamma}(a) &\quad \text{if}\quad x_\alpha\in {\mathcal V},
\end{cases}\]
and where $g^{x_\alpha,\gamma}$ is the superlinear function associated to
$G^{x_\alpha,\gamma}$ and given by Theorem~\ref{th::G}.  
Thanks to Lemma~\ref{lem::l3}, we deduce that the maximiser $(t,x),(s,y)$ satisfies
\begin{equation}\label{eq::l40}
\begin{array}{ll}
0<M/2 &\displaystyle \le  C(1+d(x,y))-G^{\alpha,\gamma}_\eps(x,y) 
-\frac{(t-s)^2}{2\nu}-\frac{\eta}{T-t} -\alpha\psi(x)\\
& \displaystyle \le  C(1+d(x,y)) 
-\eps g\left(\frac{d(x,y)}{\eps}\right) 
-\frac{(t-s)^2}{2\nu}-\frac{\eta}{T-t} -\alpha\psi(x)
\end{array}
\end{equation}
which implies in particular that
\[\eps g\left(\frac{d(x,y)}{\eps}\right)\le C(1+d(x,y)).\]
This gives \eqref{eq::l42} as in Step 1 of the proof of Theorem~\ref{th::2}.

\paragraph{First time estimate.}
From \eqref{eq::l40} with $G^{\alpha,\gamma}_\eps\ge 0$ and
\eqref{eq::l42}, we deduce in particular that for $\eps\in
(0,1]$
\[0<M/2\le C' -\frac{(t-s)^2}{2\nu}-\frac{\eta}{T-t}.\]
This implies in particular  that  
\begin{equation}\label{eq::l120bis}
T-t\ge \frac{\eta}{C'}, \quad T-s \ge \frac{\eta}{C'} 
- \sqrt{2\nu C'} \ge \frac{\eta}{2C'} =: \sigma_\eta >0
\end{equation}
for $\nu>0$ small enough, and up to redefine $\sigma_\eta$ for the new 
constant $C'\ge C$.

\paragraph{Second time estimate.}
From Lemma~\ref{lem::l110}, we  have with 
\[\begin{array}{ll}
0< \displaystyle M/2 &  
\le  f(t)+f(s) + u_0(x)-u_0(y) -\frac{\eta}{T} -\frac{(t-s)^2}{2\nu}\\
& \le  \displaystyle f(t)+f(s) + \omega_0\circ \omega(\eps) 
-\frac{\eta}{T} -\frac{(t-s)^2}{2\nu}
\end{array}\]
where $\omega_0$ is the modulus of continuity of $u_0$. Let us choose
$\eps>0$ small enough such that
\begin{equation}\label{eq::l115}
\omega_0 \circ \omega(\eps)\le \frac{M}2.
\end{equation}
As in the proof of Lemma~\ref{lem:localization}, for $\tau = \min (t,s)$, we get
\[\frac{\eta}{T}\le 2 f(\tau + \sqrt{2 \nu C'}).\]
For $\nu$ small enough (with $\eta$ fixed), we then get a
contradiction if $\tau$ converges to $0$ as $\nu$ does.

\paragraph{Convergence of maximizers.}
Because of \eqref{eq::l116} and using the fact that 
$G^{\alpha,\gamma}_\eps\ge 0$, we get for $\eps\in (0,1]$
\[M_\alpha - \gamma \le M_{\alpha,\eps} \le u(t,x)-w(s,y)
-\alpha\psi(x)-\frac{\eta}{T-t}
-\frac{(t-s)^2}{2\nu}-\varphi^{\alpha}(t,s,x).\] 
Extracting a subsequence if needed, we can assume
\[(t,x,s,y)\to (\bar t,\bar x,\bar s,\bar x) \quad \text{as}\quad
(\eps,\gamma)\to (0,0)\]
for some $\bar t,\bar s\in [\tau_\nu,T-\sigma_\eta]$, $\bar
x \in \overline{B(x_\alpha,r)}$.
We get
\[M_\alpha \le u(\bar t,\bar x)-w(\bar s,\bar x)
-\alpha\psi(\bar x)-\frac{\eta}{T-\bar t} 
-\frac{(\bar t-\bar s)^2}{2\nu} -\varphi^{\alpha}(\bar t,\bar s,\bar x)
\le M_\alpha -\varphi^{\alpha}(\bar t,\bar s,\bar x)\]
which implies that $(\bar t,\bar s,\bar x) = (t_\alpha,s_\alpha, x_\alpha)$. 
\end{proof}

\paragraph{Step 4: Viscosity inequalities.}
Then we can write the viscosity inequalities at $(t,x)$ and $(s,y)$
using the shorthand notation \eqref{eq::l24bis},
\begin{align}
\label{eq::l100}
\frac{\eta}{(T-t)^2} + \frac{t-s}{\nu}+ (t-t_\alpha) 
+ H_{\mathcal  N}(t,x,p_x^{\alpha,\gamma,\eps}
+ \alpha \psi_x(x)+ \varphi^\alpha_x(t,s,x))\le 0 \\
\nonumber \frac{t-s}{\nu} - (s-s_\alpha) 
+ H_{\mathcal N}(s,y,p_y^{\alpha,\gamma,\eps})\ge 0
\end{align}
where
\[ \begin{cases}
p_x^{\alpha,\gamma,\eps} =G^{\alpha,\gamma}_x(\eps^{-1}x,\eps^{-1}y),\\
p_y^{\alpha,\gamma,\eps} = -G^{\alpha,\gamma}_y(\eps^{-1}x,\eps^{-1}y).
\end{cases}
\]
We choose $\eps,\gamma$ small enough such that
(Lemma~\ref{lem:penal}) we have
\[|t-t_\alpha|, \quad |s-s_\alpha|\quad \le \frac{\eta}{4T^2}.\]
Substracting the two viscosity inequalities, we get
\begin{equation}\label{eq::l125}
 \frac{\eta}{2T^2}\le H_{\mathcal  N}(s,y,p_y^{\alpha,\gamma,\eps}) 
- H_{\mathcal N}(t,x,p_x^{\alpha,\gamma,\eps}+\alpha
\psi_x(x)+\varphi^\alpha_x(t,s,x)).
 \end{equation}

\paragraph{Step 5: Gradient estimates.}
We deduce from \eqref{eq::l100} that 
\[\tilde{p}^{\alpha,\gamma,\eps}_x=p^{\alpha,\gamma,\eps}_x+  \alpha \psi_x(x)+\varphi^\alpha_x(t,s,x)\]
satisfies
\begin{equation}\label{eq::l135}
H_{\mathcal N}(t,x,\tilde{p}^{\alpha,\gamma,\eps}_x)
\le \frac{s-t}{\nu} + t_\alpha -t\le \frac{T}{\nu} + T.
\end{equation}
Hence (H1) implies that there exists a
constant $C'_\nu$ (independent of $\alpha$, $\eps$, $\gamma$,
but depending on $\eta, \nu$) such that
\[\begin{cases}
|\tilde{p}^{\alpha,\gamma,\eps}_x| \le C'_\nu & 
\quad \text{if}\quad x\not= x_\alpha 
\quad \text{or} \quad x_\alpha \notin \mathcal{V},\\
\tilde{p}^{\alpha,\gamma,\eps}_x\ge - C'_\nu & 
\quad \text{if}\quad x= x_\alpha \quad \text{and}\quad x_\alpha\in {\mathcal V}.
\end{cases}\]
From \eqref{eq::l117}, we deduce that
\begin{equation}\label{eq::l147}
|\alpha \psi_x(x)+\varphi^\alpha_x(t,s,x)|\le C\sqrt{\alpha} + d(x,x_\alpha) \le C
\end{equation}
for $\alpha\le 1$ (using \eqref{eq::l117}).  Therefore, we have for
some constant $C_\nu$ (independent of $\alpha$, $\eps$,
$\gamma$):
\[\begin{cases}
|p^{\alpha,\gamma,\eps}_x| \le C_\nu & \quad \text{if}\quad
x\not= x_\alpha \quad \text{or} \quad x_\alpha \notin \mathcal{V},\\
p^{\alpha,\gamma,\eps}_x\ge - C_\nu & \quad \text{if}\quad x= x_\alpha \quad \text{and}\quad x_\alpha\in {\mathcal V}.
\end{cases}\]
From the compatibility condition of the Hamiltonians satisfied by
$G^{\alpha,\gamma}$ if $x_\alpha\in {\mathcal V}$, or the definition
of $G^{\alpha,\gamma}$ if $x_\alpha\notin {\mathcal V}$, we have in
both cases, 
\begin{equation}\label{eq::l128}
H^{t_\alpha,x_\alpha}(y,p_y^{\alpha,\gamma,\eps}) \le
H^{t_\alpha,x_\alpha}(x,p_x^{\alpha,\gamma,\eps}) + \gamma
\end{equation}
where
\[H^{t_\alpha,x_\alpha}(x,p)= \begin{cases}
H^{t_\alpha,n}_{\mathcal V}(x,p) & \text{if}\quad x_\alpha=n \in
{\mathcal V} ,\\
H_e(t_\alpha,x_\alpha,p) &   \text{if}\quad x_\alpha \notin
{\mathcal V}, x_\alpha \in  e^*.
\end{cases}\]
We deduce that
$p^{\alpha,\gamma,\eps}_y$
satisfies (modifying $C_\nu$ if necessary)
\[ \begin{cases}
|p^{\alpha,\gamma,\eps}_y| \le C_\nu & \quad \text{if}\quad
y\not= x_\alpha \quad \text{ or } x_\alpha \notin \mathcal{V},\\
p^{\alpha,\gamma,\eps}_y\ge - C_\nu & \quad \text{if}\quad 
y= x_\alpha \quad \text{and}\quad x_\alpha\in {\mathcal V}.
\end{cases}\]
For $z=x,y \in \mathcal{V}$, $p_z^{\alpha,\gamma,\eps}$ is a vector
and its components are only bounded from below, see above. But when
writing viscosity inequalities, they appear as variables of the
non-increasing part of Hamiltonians. Hence, if they are too large,
they can be replaced with the point minimizing the Hamiltonian,
without changing the viscosity inequalities.  This is the reason why
we truncate each component of this vector by a well chosen constant
$K$. Precisely, we define for $z=x,y$,
\[\bar p_z^{\alpha,\gamma,\eps} = \left\{\begin{array}{ll}
\left(
\min\left(K, (p^{\alpha,\gamma,\eps}_z)_{\tilde{z}}\right)
\right)_{\tilde{z}\in x_\alpha} &\quad \text{if}\quad z= x_\alpha \quad 
\text{and}\quad x_\alpha\in {\mathcal V} \\
p^{\alpha,\gamma,\eps}_z  &\quad \text{if not.}
\end{array}\right.\]
with, in the case where $x_\alpha \in \mathcal{V}$, the constant $K$
given by
\[ K = \max_{e \in \mathcal{E}_{x_\alpha}} (p_e^0(s,x_\alpha),
p_e^0(t_\alpha,x_\alpha),p_e^0 (t,x_\alpha)+C)) \le C_T+C\]
 ($C$ comes from  \eqref{eq::l147} and $C_T$ from \eqref{eq::l136bis}). 
We then have 
\[|\bar p_z^{\alpha,\gamma,\eps}|\le C_\nu+C_T+C=:C_{\nu,T}\]
and
\begin{align}
\label{eq::l125bis}
\frac{\eta}{2T^2}\le H_{\mathcal  N}(s,y,\bar p_y^{\alpha,\gamma,\eps})- 
H_{\mathcal N}(t,x,\bar p_x^{\alpha,\gamma,\eps}+\alpha\psi_x(x)
+\varphi^\alpha_x(t,s,x)),\\
\label{eq::l135bis}
H_{\mathcal N}(t,x,\bar p^{\alpha,\gamma,\eps}_x 
+ \alpha \psi_x(x)+\varphi^\alpha_x(t,s,x))
\le \frac{s-t}{\nu} + t_\alpha -t\le \frac{T}{\nu} + T,
\\
\label{eq::l128bis}
H^{t_\alpha,x_\alpha}(y,\bar p_y^{\alpha,\gamma,\eps}) \le
H^{t_\alpha,x_\alpha}(x,\bar p_x^{\alpha,\gamma,\eps}) + \gamma.
\end{align}

\paragraph{Step 6: The limit $(\eps,\gamma)\to (0,0)$ 
and conclusion as $\alpha\to 0$.}
Up to a subsequence, we get in the limit $(\eps,\gamma)\to
(0,0)$ for $z=x,y$:
\[\bar p_z^{\alpha,\gamma,\eps} \to \bar p_z^{\alpha} \quad \text{with}\quad |\bar p_z^{\alpha}|\le C_{\nu,T}.\]
Moreover, passing to the limit in \eqref{eq::l125bis} and
\eqref{eq::l135bis}, we get respectively
\[\displaystyle \frac{\eta}{2T^2}\le H_{\mathcal N}(s_\alpha,x_\alpha,\bar p_y^{\alpha})
- H_{\mathcal N}(t_\alpha,x_\alpha,\bar p_x^{\alpha}+ \alpha \psi_x(x_\alpha))\]
and
\[H_{\mathcal N}(t_\alpha,x_\alpha,\bar {p}^{\alpha}_x+ \alpha \psi_x(x_\alpha))\le \frac{s_\alpha-t_\alpha}{\nu} \le \frac{T}{\nu}.\]
On the other hand, passing to the limit in \eqref{eq::l128bis} gives
\[H^{t_\alpha,x_\alpha}(x_\alpha,\bar p_y^{\alpha}) \le H^{t_\alpha,x_\alpha}(x_\alpha, \bar p_x^{\alpha}).\]
Because
\[H_{\mathcal N}(t_\alpha,x_\alpha,p)=H^{t_\alpha,x_\alpha}(x_\alpha,p)\]
we get for any $p$,
\[\frac{\eta}{2T^2} \le I_1 + I_2\]
with
\begin{align*}
I_1 &=H_{\mathcal N}(s_\alpha,x_\alpha,\bar p_x^{\alpha})-H_{\mathcal N}(s_\alpha,x_\alpha,\bar p_x^{\alpha}+ \alpha \psi_x(x_\alpha)),\\
I_2 & =H_{\mathcal N}(s_\alpha,x_\alpha,\bar p_x^{\alpha}+ \alpha \psi_x(x_\alpha))
- H_{\mathcal N}(t_\alpha,x_\alpha,\bar p_x^{\alpha}+ \alpha
\psi_x(x_\alpha)).
\end{align*}
Thanks to (H3) and \eqref{eq::l117}, we have $|\alpha
\psi_x(x_\alpha)|\le C_{\nu,T}$ and we thus get
\begin{equation}\label{eq::l150}
I_1 \le \omega_{T,2C_{\nu,T}}(\alpha \psi_x(x_\alpha)) \le
\omega_{T,2C_\nu} (C\sqrt{\alpha}).
\end{equation}
Now thanks to Lemma~\ref{lem::l146}, we also have
\begin{align*}
I_2 &
\le \tilde{\omega}_T(|t_\alpha-s_\alpha|(1+\max(H_{\mathcal N}(t_\alpha,x_\alpha,\bar p_x^{\alpha}+ \alpha \psi_x(x_\alpha)),0)))\\
 &\le \tilde{\omega}_T(|t_\alpha-s_\alpha|(1+\max(\frac{s_\alpha-t_\alpha}{\nu},0))).
\end{align*}
Then taking first the limit $\alpha\to 0$ and then taking the limit
$\nu\to 0$, we use \eqref{eq::l155} to get the desired contradiction.
This achieves the proof of Theorem~\ref{th::l2}. 
\end{proof}


\section{First application: link with optimal control theory}
\label{s:oct}

This section is devoted to the study of the value function of an
optimal control problem associated with trajectories running over the
junction. 

\subsection{Assumptions on dynamics and running costs}

As before, we consider a junction $J=\bigcup_{i=1,\dots,N} J_i$.  We
consider compact metric spaces ${\mathbb A}_i$  for $i=0,\dots,N$ and
functions $b_i, \ell_i: [0,T] \times J_i \times {\mathbb A}_i\to \R$
for $i=1,\dots,N$ and $b_0, \ell_0:[0,T] \times \mathbb{A}_0 \to \R$. 
The sets $\mathbb{A}_i$ are the sets of controls on each branch
$J_i^*$ for $i=1,\dots,N$, while the set ${\mathbb A}_0$ is the set of
controls at the junction point $x=0$.  The functions $b_i$ represent
the dynamics and the $\ell_i$'s are the running cost functions.

For $i=1,\dots,N$, we follow \cite{bbc2} by assuming the following
\begin{equation}\label{eq::pr19}
\left\{\begin{array}{ll}
b_i \text{ and } \ell_i \text{ are continuous and bounded } \medskip \\
b_i \text{ is uniformly continuous w.r.t. } (t,x) \text{ uniformly
  w.r.t. } \alpha_i \medskip \\
\ell_i \text{ is uniformly continuous w.r.t. } (t,x) \text{ uniformly
  w.r.t. } \alpha_i \medskip \\
{\mathcal B}_i(t,x) := \{ (b_i (t,x,\alpha_i),\ell_i (t,x,\alpha_i)): 
\alpha_i \in {\mathbb A}_i \} \text{ is closed and convex} \medskip \\
 B_i(t,x) = \{ b_i (t,x,\alpha_i): \alpha_i \in
     {\mathbb A}_i \} \text{ contains } [-\delta,\delta]
\end{array}\right.
\end{equation}
for some $\delta$ independent of $(t,x)$. 

It is easy to check the following lemmas.
\begin{lem}[Hamiltonians]\label{lem::pp45}
Assume (\ref{eq::pr19}). 
Then given $i \in \{1,\dots,N\}$, the Hamiltonian $H_i$ defined by
\[H_i(t,x,p_i)=\sup_{\alpha_i\in {\mathbb A}_i}\left(
  b_i(t,x,\alpha_i) p_i - \ell_i(t,x,\alpha_i)\right)\] satisfies
Assumption~\eqref{assum:H}.
\end{lem}
\begin{lem}[Non-increasing Hamiltonians]\label{lem:hj-dec}
  Assume (\ref{eq::pr19}). Given $i \in \{1,\dots,N\}$, then the
  non-increasing part of $H_i(t,0,p_i)$ with respect to $p_i$, is
  given by
\begin{align*}
 H_i^- (t,p_i) & = \sup_{\alpha_i \in \mathbb{A}_i^-} \left(
b_i(t,0,\alpha_i) p_i - \ell_i(t,0,\alpha_i)\right)\\
& = \sup_{\alpha_i \in \mathbb{A}_i^<} \left(
b_i(t,0,\alpha_i) p_i - \ell_i(t,0,\alpha_i)\right) 
\end{align*}
where 
$\mathbb{A}_i^-  = \{ \alpha_i \in \mathbb{A}_i :
b_i(t,0,\alpha_i) \le 0\}$ and 
$\mathbb{A}_i^<  = \{ \alpha_i \in \mathbb{A}_i :
b_i(t,0,\alpha_i) < 0\}$. 
\end{lem}

As far as the dynamics and running costs at the junction point are
concerned, we also assume that 
\begin{equation}\label{assum:0cont}
 b_0 \text{ and } l_0 \text{ are continuous  bounded, \; }  \mathbb A_0 \subset \R^{d_0}
\end{equation}
for some $d_0\ge 1$, and define
\[ B_0(t) = \{ b_0 (t,\alpha_0) : \alpha_0 \in \mathbb{A}_0\}.\] 

We also define 
\begin{equation}\label{eq::pp84}
A_0(t) = \max_{i=1,\dots,N} \min_{p \in \R} H_i (t,0,p).
\end{equation}
We set
\begin{equation}\label{eq::pp72}
H_0(t)=\begin{cases}
\displaystyle  \sup_{\alpha_0\in {\mathbb A}_{0}(t)}\left(- \ell_0(t,\alpha_0)\right) &
\quad \mbox{if}\quad {\mathbb A}_{0} (t)\not=\emptyset,\\
-\infty &\quad \mbox{if}\quad {\mathbb A}_{0}(t)=\emptyset
\end{cases}
\end{equation}
with 
\begin{equation}\label{eq::pp62}
{\mathbb A}_{0}(t)=\left\{\alpha_0\in {\mathbb A}_0,\quad b_0(t,\alpha_0)=0\right\},
\end{equation}
and we assume that
\begin{equation}\label{assum:hot}
\bar H_0 : t \mapsto \max (H_0(t),A_0(t)) \text{ is continuous in } [0,T].
\end{equation}

\subsection{The value function}

We then define the general set of controls,
\[{\mathbb A} = {\mathbb A}_0\times \dots\times {\mathbb A}_N\]
and define for $\alpha=(\alpha_0,\dots,\alpha_N)\in {\mathbb A}$ and $(t,x)\in [0,T]\times J$,
\[b(t,x,\alpha)=\left\{\begin{array}{ll}
b_i(t,x,\alpha_i) & \quad \mbox{if}\quad x\in J_i^*,\\
b_0(t,\alpha_0) & \quad \mbox{if}\quad x=0.
\end{array}\right.\]
Similarly, we define
\[\ell(t,x,\alpha)=\left\{\begin{array}{ll}
\ell_i(t,x,\alpha_i) & \quad \mbox{if}\quad x\in J_i^*,\\
\ell_0(t,\alpha_0) & \quad \mbox{if}\quad x=0.
\end{array}\right.\]

For $0\le s<t\le T$ and $y,x\in J$, we define the set of admissible dynamics
\begin{equation}\label{eq::pr60}
{\mathcal T}^{t,x}_{s,y}=\left\{\begin{array}{l}
(X(\cdot),\alpha(\cdot))\in \Lip ([s,t];J)\times
L^\infty([s,t];{\mathbb A}),\medskip \\
\left\{\begin{array}{l}
X(s)=y,\quad X(t)=x,\\
\dot{X}(\tau)=b(\tau,X(\tau),\alpha(\tau)) \quad \mbox{for a.e.}\quad \tau\in (s,t)
\end{array}\right.
\end{array}\right\}.
\end{equation}

Then we consider the value function of the optimal control problem,
\begin{equation}\label{eq::pr20}
u(t,x)=\inf_{z\in J}\quad \inf_{(X(\cdot),\alpha(\cdot))\in {\mathcal T}^{t,x}_{0,z}}E_0^t(X,\alpha) 
\end{equation}
with
\[E_0^t(X,\alpha)=u_0(X(0))+\int_0^t \ell(\tau,X(\tau),\alpha(\tau))\ d\tau\]
where the initial datum $u_0$ is assumed to be globally Lipschitz
continuous. 

Note that if ${\mathcal T}^{t,x}_{0,z}=\emptyset$, then we have
$\displaystyle \inf_{{\mathcal T}^{t,x}_{0,z}} (\dots)=+\infty$.  More
generally and for later use, we set
\begin{equation}\label{eq::pp40}
E_s^t(X,\alpha)=u(s,X(s))+\int_s^t \ell(\tau,X(\tau),\alpha(\tau))\
d\tau .
\end{equation}

\subsection{Dynamic programming principle}

The following result is expected and quite standard.
\begin{pro}[Dynamic programming principle]\label{pro:dpp}
  For all $x \in J$, $t \in (0,T]$ and $s \in [0,t)$, the value
  function $u$ defined in (\ref{eq::pr20}) satisfies
  \[ u (t,x) = \inf_{y \in J} \inf_{(X(\cdot),\alpha(\cdot)) \in
    \mathcal{T}^{t,x}_{s,y}} E_s^t(X,\alpha)\] where $E_s^t$ and
  $\mathcal{T}^{t,x}_{s,y}$ are defined respectively in
  (\ref{eq::pp40}) ansd (\ref{eq::pr60}).
\end{pro}
\begin{proof}
  Let $V(t,x)$ denote the right hand side of the desired equality.
  Consider $(X(\cdot),\alpha(\cdot)) \in \mathcal{T}^{s,y}_{0,z}$ and
  $(\tilde{X}(\cdot),\tilde{\alpha}(\cdot)) \in
  \mathcal{T}_{s,y}^{t,x}$. Then
\[ (\bar X (\tau),\bar \alpha(\tau)) = 
\begin{cases} (X (\tau),\alpha(\tau)) & \text{ if } \tau \in  [0,s] \\ 
(\tilde{X}(\tau),\tilde{\alpha}(\tau)) & \text{ if } \tau \in (s,t] \end{cases} \]
lies in $\mathcal{T}^{t,x}_{0,z}$. In particular, 
\begin{align*} u(t,x) &\le u_0(z) + \int_0^t  \ell(\tau,\bar X(\tau),\bar \alpha(\tau))
 \, d\tau \\
& \le u_0(z) + \int_0^s  \ell(\tau,X(\tau),\alpha(\tau))\, d\tau +
  \int_s^t  \ell(\tau,\tilde{X}(\tau),\tilde{\alpha}(\tau)) \,  d\tau. 
\end{align*}
Taking the infimum, first with respect to $(X(\cdot),\alpha(\cdot))$ and $z$, and then with respect to $(\tilde{X} (\cdot),\tilde{\alpha}(\cdot))$
yields $u(t,x) \le V(t,x)$.  

To get the reversed inequality, consider, for all $\eps >0$,
an admissible dynamics $(X^\eps(\cdot),\alpha^\varepsilon(\cdot)) \in \mathcal{T}^{t,x}_{0,z}$ such that 
\begin{align*}
u(t,x) & \ge u_0 (X^\eps (0)) + \int_0^t \ell(\tau,{X}^\varepsilon(\tau),{\alpha}^\varepsilon(\tau)) \, d\tau - \eps \\
& \ge u_0 (X^\eps (0)) + \int_0^s  \ell(\tau,{X}^\varepsilon(\tau),{\alpha}^\varepsilon(\tau)) \, d\tau
 + \int_s^t  \ell(\tau,{X}^\varepsilon(\tau),{\alpha}^\varepsilon(\tau))  \, d\tau -\eps \\
& \ge u(s,X^\eps (s)) + \int_s^t  \ell(\tau,{X}^\varepsilon(\tau),{\alpha}^\varepsilon(\tau)) 
 \, d\tau -\eps \\
& \ge V(t,x) - \eps. 
\end{align*}
Since $\eps$ is arbitrary, we conclude. 
\end{proof}

\subsection{Derivation of the Hamilton-Jacobi-Bellman equation}

We will show that the value function $u$ solves the following problem
\begin{equation}\label{eq::pp42}
\left\{\begin{array}{ll}
u_t + H_i(t,x,u_x) = 0 & \quad \mbox{for all}\quad (t,x)\in (0,T)\times J_i^*,\\
u_t +F_{\bar H_0(t)}(t,u_x)=0 & \quad \mbox{for all}\quad (t,x)\in (0,T)\times \left\{0\right\}
\end{array}\right.
\end{equation}
with
$$F_{\bar H_0(t)}(t,u_x(t,0^+)) :=  \max\left({\bar H_0(t)},\ 
\displaystyle \max_{i=1,\dots,N} H_i^-(t,\partial_i u(t,0^+))\right)$$
and with initial condition
\begin{equation}\label{eq::pp43}
u(0,x)=u_0(x) \quad \mbox{for all}\quad x\in  J.
\end{equation}
We also consider the following condition for $i=1,\dots,N$
\begin{equation}\label{eq::toto}
b_i \text{ is Lipschitz continuous w.r.t. } t \text{ uniformly
  w.r.t. } (x,\alpha_i).
\end{equation}
\begin{theo}[The value function is a flux-limited solution]\label{th::21}
  Assume (\ref{eq::pr19}), (\ref{assum:0cont}) and (\ref{assum:hot}).
  Let us also consider $H_i$, $H_i^-$ and $\bar H_0$ respectively
  defined in Lemmas~\ref{lem::pp45} and \ref{lem:hj-dec} and in
  (\ref{assum:hot}). Assume also that the initial datum $u_0$ is
  globally Lipschitz on $J$.
\begin{enumerate}[\upshape i)]
\item \emph{(Existence)}.  The value function $u$ defined by
  \eqref{eq::pr20} is a solution of \eqref{eq::pp42},
  \eqref{eq::pp43}.
\item \emph{(Uniqueness)}.  If we assume moreover (\ref{eq::toto}), then $u$
  is the unique solution of \eqref{eq::pp42}, \eqref{eq::pp43}.
\end{enumerate}
\end{theo}
In order to prove this theorem, two technical results are
needed. Their proofs is postponed until the end of the proof of Theorem~\ref{th::21}.
\begin{lem}[A measurable selection result]\label{lem::pp85}
Assume that $b_0$ and $\ell_0$ satisfy (\ref{assum:0cont}).
For some $[a,b]\subset (0,T)$, let us also assume that 
$$\emptyset \not=\mathbb A_0(\tau):=\left\{\alpha_0\in \mathbb A_0,\quad b_0(\tau,\alpha_0)=0\right\}  \quad 
\mbox{for all}\quad \tau\in [a,b]$$ 
and that
$$\tau\mapsto H_0(\tau):= \sup_{\alpha_0\in \mathbb A_0(\tau)}(-\ell_0(\tau,\alpha_0)) \quad \mbox{is continuous onø}\quad [a,b].$$
Then there exists a mesurable selection $\bar \alpha_0\in L^\infty([a,b];\mathbb A_0)$ such that
$$\bar \alpha_0(\tau)\in \mathbb A_0(\tau) \quad \mbox{and}\quad H_0(\tau)=-\ell_0(\tau,\bar \alpha_0(\tau)) \quad \mbox{for a.e.}\quad \tau\in [a,b].$$
\end{lem}
\begin{pro}[Checking assumptions for the comparison principle]\label{pro::pp95}
  Assume (\ref{eq::pr19}), (\ref{assum:0cont}), (\ref{assum:hot}) and
  (\ref{eq::toto}).  Let us also consider $H_i$, $H_i^-$ and $\bar
  H_0$ respectively defined in Lemmas~\ref{lem::pp45} and
  \ref{lem:hj-dec} and in (\ref{assum:hot}).  Using notation from
  Section~\ref{s.n} on networks, let us consider the network $\mathcal
  N = J$, with edges ${\mathcal
    E}=\left\{J_1,\dots,J_N\right\}={\mathcal E}^-_n$ where the unique
  vertex $n$ is identified to the junction point $0$.  We set
  $H_e(t,x,p):=H_i(t,x,p)$ and $H_e^-(t,p)=H_i^-(t,p)$ for $e=J_i$ for
  each $i=1,\quad, N$.  We also set $A_n(t):=\bar H_0(t)$. Then
  assumptions (H0)-(H6) and (A0)-(A2) are satisfied.
\end{pro}
\begin{proof}[Proof of Theorem~\ref{th::21}]
We will show that $u^*$ is a super-solution and $u_*$ is a sub-solution on $(0,T)\times J$.
Deriving the Hamilton-Jacobi-Bellman equation outside the junction
point is known and standard. This is the reason why we will focus on the
junction condition. As in the standard case, it relies on the dynamic
programming principle.

\paragraph{Step 1: the super-solution property.}
Consider any test function $\varphi$ such that
\[\varphi\le u_* \text{ in } (0,+\infty)\times J\quad
\text{ and } \quad \varphi = u_* \text{ at } (\bar t,0)\quad
\mbox{with}\quad \bar t\in (0,T).\] Our goal is to show that
\begin{equation}\label{eq::pp50}
\varphi_t(\bar t,0)+F_{\bar H_0(\bar t)}(\bar t,\varphi_x(\bar t,0^+)) \ge 0
\end{equation}
The proof of this inequality proceeds in several substeps. \medskip

\noindent {\sc Step 1.1: the basic optimal control inequality.}
Let $(t_n,x_n)\in (0,T)\times J$ be such that
$$(t_n,x_n)\to (\bar t,0) \quad \mbox{and}\quad u(t_n,x_n)\to u_*(\bar
t,0) 
\quad \mbox{as}\quad n\to +\infty.$$
Let $s\in (0,\bar t)$. Then the dynamic programming principle yields
$$u(t_n,x_n)=\inf_{y \in J} \inf_{(X(\cdot),\alpha(\cdot)) \in \mathcal{T}^{t_n,x_n}_{s,y}}
\left\{u(s,X(s))+\int_s^{t_n} \ell(\tau,X(\tau),\alpha(\tau))\ d\tau\right\}$$
This implies that
$$\varphi(t_n,x_n) + o_n(1) \ge \inf_{y \in J} \inf_{(X(\cdot),\alpha(\cdot)) \in \mathcal{T}^{t_n,x_n}_{s,y}}
\left\{\varphi(s,X(s))+\int_s^{t_n} \ell(\tau,X(\tau),\alpha(\tau))\ d\tau\right\}$$
where $o_n(1)\to 0$ as $n\to +\infty$. Therefore, we have
\begin{equation}\label{eq::pp65}
S_n:=\sup_{y \in J} \sup_{(X(\cdot),\alpha(\cdot)) \in \mathcal{T}^{t_n,x_n}_{s,y}} K_s^{t_n}(X,\alpha)
\ge -o_n(1)
\end{equation}
where
\begin{equation}\label{eq::pp77}
K_s^{t_n}(X,\alpha):=\varphi(t_n,X(t_n)) -\varphi(s,X(s))
-\int_s^{t_n}  \ell(\tau,X(\tau),\alpha(\tau))\ d\tau
\end{equation}
with
$$\varphi(t_n,X(t_n)) -\varphi(s,X(s))= \int_s^{t_n}d\tau\
\left\{\varphi_t(\tau,X(\tau)) 
+ \varphi_x(\tau,X(\tau))b(\tau,X(\tau),\alpha(\tau))\right\}.$$
Here, we take the convention that the product $\varphi_x b$ equals $0$ if
$X(\tau)=0$.  This makes sense for almost every $\tau$, because by
Stampacchia's truncation theorem, we have
\begin{equation}\label{eq::pp60}
  0=\dot{X}(\tau)=b(\tau,X(\tau),\alpha(\tau))
  =b_0(\tau,\alpha_0(\tau))
\quad \mbox{a.e. on}\quad  \left\{\tau\in(s,t_n),\ X(\tau)=0\right\}
\end{equation}
which implies in particular
\begin{equation}\label{eq::pp61}
\alpha_0(\tau)\in \mathbb A_0(\tau) \quad \mbox{a.e. on}\quad  \left\{\tau\in(s,t_n),\ X(\tau)=0\right\}
\end{equation}
where $\mathbb A_0$ is defined in (\ref{eq::pp62}).
This shows that we can write
$$K_s^{t_n}(X,\alpha)= \int_s^{t_n}d\tau\ \kappa(\tau,X(\tau),\alpha(\tau))$$
with for $(\tau,x)\in (0,T)\times J$ and $\beta=(\beta_0,\dots,
\beta_N) \in \mathbb A$:
$$\kappa(\tau,x,\beta) 
= \varphi_t(\tau,x) + \varphi_x(\tau,x)b(\tau,x,\beta) -\ell(\tau,x,\beta)$$
with the convention that 
$$\left\{\begin{array}{ll}
\varphi_x(\tau,x)b(\tau,x,\beta) = 0\\
\beta_0\in \mathbb A_0(\tau)
\end{array}\right| \quad \mbox{if}\quad  x=0.$$

\noindent {\sc Step 1.2: freezing the coefficients.}
We now freeze the coefficients at the point $(\bar t,0)\in (0,T)\times J$, defining for any  
$(\tau,x)\in (0,T)\times J$ and $\beta\in \mathbb A$:
\begin{equation}\label{eq::pp79}
\bar \kappa (\tau,x,\beta) :=\left\{\begin{array}{ll}
\varphi_t(\bar t,0) + \partial_i\varphi(\bar t,0)b_i(\bar t,0,\beta_i) -\ell_i(\bar t,0,\beta_i) & \quad \mbox{if}\quad x\in J_i^*,\\
\varphi_t(\bar t,0)  -\ell_0(\tau,\beta_0)  & \quad \mbox{if}\quad x=0,
\end{array}\right.
\end{equation}
with the convention that $\beta_0\in \mathbb A_0(\tau)$ if $x=0$.
From structural assumptions (\ref{eq::pr19}) and (\ref{assum:0cont}),
there exists a (monotone continuous) modulus of continuity $\omega$
(depending only on $\varphi$ and the quantities $b_i$, $\ell_i$ for
$i=0,\dots,N$) such that
$$|\bar \kappa (\tau,x,\beta)  -\kappa(\tau,x,\beta)|\le \omega(|\bar
t-\tau| + d(x,0)) \quad \mbox{for all}\quad (\tau,x,\beta)\in
(0,T)\times J\times \mathbb A.$$ Since trajectories are uniformly
Lipschitz, there exists a constant $C_0>0$ such that for all $\tau \in
(s,t_n)$, 
$$d(X(\tau),0)\le d(x_n,0) + C_0|t_n-\tau| = o_n(1) + C_0|\bar t -\tau|.$$
Defining
\begin{equation}\label{eq::pp66}
\bar K_s^{t_n}(X,\alpha)= \int_s^{t_n}d\tau\ \bar \kappa(\tau,X(\tau),\alpha(\tau))
\end{equation}
we get that
\begin{equation}\label{eq::pp78}
  |\bar K_s^{t_n}(X,\alpha) - K_s^{t_n}(X,\alpha)|\le
  |t_n-s|\omega(o_n(1)+C_1|\bar t -s|) 
\quad \mbox{with}\quad  C_1=1+C_0.
\end{equation}

\noindent {\sc Step 1.3: application to a quasi-optimizer.} 
Let us consider a quasi-optimizer $(X^n,\alpha^n)\in
\mathcal{T}^{t_n,x_n}_{s,y_n}$ for some $y_n\in J$ such that
$$K_s^{t_n}(X^n,\alpha^n) \ge S_n -o_n(1).$$
By \eqref{eq::pp65} and estimate \eqref{eq::pp78}, this implies
\begin{equation}\label{eq::pp75}
\bar K_s^{t_n}(X^n,\alpha^n) \ge -o_n(1)- |t_n-s|\omega(o_n(1)+C_1|\bar t -s|).
\end{equation}
In order to evaluate $\bar K_s^{t_n}(X^n,\alpha^n)$, we naturally define the following sets. 
Let
$$\mathbb T_0^n = \left\{\tau\in (s,t_n),\quad X^n(\tau)=0\right\}$$
which is a (relative) closed set of $(s,t_n)$, and let us set   for $i=1,\dots,N$:
$$\mathbb T_i^n = \left\{\tau\in (s,t_n),\quad X^n(\tau) \in J_i^*\right\}$$
which are open sets.
We have
$$\bar K_s^{t_n}(X^n,\alpha^n) = \sum_{i=0,\dots,N} \bar K_i^n \quad \mbox{with}\quad 
\bar K_i^n:= \int_{\mathbb T_i^n}d\tau\  \bar
\kappa(\tau,X^n(\tau),\alpha^n(\tau)).$$
We next study each term $\bar K_i^n$ of the previous sum. \medskip

\noindent {\sc Step 1.3.1: convergence for $i=1,\dots,N$.}
We now use an argument that we found in \cite{bbc2}.
For $i=1,\dots,N$, by convexity of the set ${\mathcal B}_i(\bar t,0)$ defined in (\ref{eq::pr19}), 
we deduce that there exists some $\bar \alpha_i^n\in \mathbb A_i$
such that
\begin{equation}\label{eq::pp67}
\frac{1}{|\mathbb T_i^n|}\int_{\mathbb T_i^n}d\tau\  (b_i(\bar t,0,\alpha^n(\tau)), \ell_i(\bar t,0,\alpha^n(\tau))) = 
(b_i(\bar t,0,\bar \alpha_i^n), \ell_i(\bar t,0,\bar \alpha_i^n))
\end{equation}
and then
$$ \bar K_i^n  = |\mathbb T_i^n| \left\{\varphi_t(\bar t,0) + \partial_i\varphi(\bar t,0)b_i(\bar t,0,\bar \alpha_i^n) -\ell_i(\bar t,0,\bar \alpha_i^n)\right\}.$$
Moreover, decomposing the set $\mathbb T_i^n$ in a (at most countable) union of intervals $(a_k,b_k)$ (with possibly $a_k=s$ or $b_k=t_n$ for some particular value of $k$), we see that we have with $x_n=X(t_n)$
\begin{multline}\label{eq::pp68}
\int_{\mathbb T_i^n}d\tau\  b_i(\bar t,0,\alpha^n(\tau))  =
\int_{\mathbb T_i^n}d\tau\ \dot{X}^n(\tau) \\
= \left\{\begin{array}{lll}
0-X^n(s) & \quad \mbox{if} \quad X^n(t_n)\not \in J_i^*, &\quad X^n(s)\in J_i^*,\\
X(t_n)-X^n(s) & \quad \mbox{if} \quad X^n(t_n)\in J_i^*, &\quad X^n(s)\in J_i^*,\\
X(t_n)-0 & \quad \mbox{if} \quad X^n(t_n)\in J_i^*, &\quad X^n(s)\not\in J_i^*.
\end{array}\right.
\end{multline}
Up to a subsequence, we have $\bar \alpha_i^n \to \bar \alpha_i$,
$|\mathbb T_i^n|\to T_i$ for some $T_i \ge 0$. It is
convenient to write $T_i$ as $|\mathbb T_i|$. Remark in particular
that we have
\[ \sum_{i=0}^N |\mathbb T_i | = \bar t -s. \]
Next, we get that the sequence of 
trajectories $X^n(\cdot)$ converges uniformly to some $X(\cdot)$  such
that
$$|\mathbb T_i| b_i(\bar t,0, \bar \alpha_i) = \left\{\begin{array}{ll}
0-X(s) & \quad \mbox{if} \quad X(s)\in J_i^*,\\
0  & \quad \mbox{if} \quad X(s)\not\in J_i^*
\end{array}\right.$$
and therefore
$$b_i(\bar t,0, \bar \alpha_i) \le 0 \quad \mbox{if}\quad |\mathbb T_i|\not=0.$$
This implies
$$\bar K_i^n \to \bar K_i$$
with
\begin{equation}\label{eq::pp70}
\begin{array}{lll}
\bar K_i  & := & |\mathbb T_i| \left\{\varphi_t(\bar t,0) 
+ \partial_i\varphi(\bar t,0)b_i(\bar t,0,\bar \alpha_i) -\ell_i(\bar t,0,\bar \alpha_i)\right\}\\
& \le & |\mathbb T_i|\left\{\varphi_t(\bar t,0)  + H_i^-(\bar t,\partial_i\varphi(\bar t,0)) \right\}\\
& \le &  |\mathbb T_i|\left\{\varphi_t(\bar t,0)  + F_{\bar H_0(t)}(t,\varphi_x(t,0^+))\right\}.
\end{array}
\end{equation}

\noindent {\sc Step 1.3.2: convergence for $i=0$.}
We have
$$\bar K_0^n =  \int_{\mathbb T_0^n}d\tau\  \bar \kappa(\tau,X^n(\tau),\alpha^n(\tau)) 
= \int_{\mathbb T_0^n}d\tau\ \left\{\varphi_t(\bar t,0)
  -\ell_0(\tau,\alpha_0^n(\tau)) \right\}.$$ 
Because of (\ref{eq::pp61}), we know that $\alpha_0^n(\tau)\in \mathbb
A_0(\tau)$ for almost every $\tau\in \mathbb T_0^n$ which implies
$$\bar K_0^n \le   \int_{\mathbb T_0^n}d\tau\  \left\{\varphi_t(\bar t,0)  +H_0(\tau) \right\} 
\le   \int_{\mathbb T_0^n}d\tau\  
\left\{\varphi_t(\bar t,0)  +\bar H_0(\tau)) \right\}$$
where $H_0$ and $\bar H_0$ are defined in (\ref{eq::pp72}) and
\eqref{assum:hot} respectively.  Since the function $\bar H_0$ is
assumed to be continuous, see (\ref{assum:hot}), there exists some
(monotone continuous) modulus of continuity, that we still denote by
$\omega$, such that
$$\bar K_0^n \le |\mathbb T_0^n| \left\{\varphi_t(\bar t,0)  +\bar H_0(t_n) + \omega(|t_n-s|)\right\}$$
Up to a subsequence, we have $|\mathbb T_0^n| \to |\mathbb T_0|$ and then
\begin{equation}\label{eq::pp74}
\begin{array}{lll}
\limsup_{n\to +\infty} \bar K_0^n & \le & |\mathbb T_0| \left\{\varphi_t(\bar t,0)  +\bar H_0(\bar t) + \omega(|\bar t-s|)\right\}\\
& \le &   |\mathbb T_0| \left\{\varphi_t(\bar t,0)  + F_{\bar H_0(t)}(t,\varphi_x(t,0^+))+ \omega(|\bar t-s|)\right\}.
\end{array}
\end{equation}
\noindent {\sc Step 1.4: conclusion.}
From (\ref{eq::pp75}) on the one hand, and from (\ref{eq::pp70}),
(\ref{eq::pp74}) on the other hand, we deduce that
\begin{multline*}
- |\bar t-s|\omega(C_1|\bar t -s|)\le \limsup_{n\to +\infty} \sum_{i=0,\dots,N} \bar K_i^n \\
\le \left(\sum_{i=0,\dots,N} |\mathbb T_i|\right) \left\{\varphi_t(\bar t,0) +
  F_{\bar H_0(t)}(t,\varphi_x(t,0^+)\right\} + |\mathbb T_0|
\omega(|\bar t-s|).
\end{multline*}
 Using the fact that $\sum_{i=0,\dots,N} |\mathbb
T_i| = |\bar t-s|$ and $C_1\ge 1$, and dividing by $|\bar t -s|$, we
deduce that
$$- 2 \omega(C_1|\bar t -s|)\le \varphi_t(\bar t,0)   + F_{\bar H_0(t)}(t,\varphi_x(t,0^+).$$
Passing to the limit $s\to \bar t$, we deduce (\ref{eq::pp50}).

\paragraph{Step 2: the sub-solution property.}
Consider any test function $\varphi$ such that
\[\varphi \ge u^* \text{ in } (0,+\infty)\times J
\quad \text{ and } \quad \varphi = u^* \text{ at } (\bar t,0)\in
(0,T)\times J, \quad \mbox{with}\quad \bar t\in (0,T).\] Our goal is
to show that
\begin{equation}\label{eq::ppq50}
\varphi_t(\bar t,0)+F_{\bar H_0(\bar t)}(\bar t,\varphi_x(\bar t,0^+)) \le 0.
\end{equation}

\noindent {\sc Step 2.1: the basic optimal control inequality.}
Let $(t_n,x_n)\in (0,T)\times J$ such that
$$(t_n,x_n)\to (\bar t,0) \quad \mbox{and}\quad u(t_n,x_n)\to u^*(\bar
t,0) \quad \mbox{as}\quad n\to +\infty.$$ 
From the dynamic programming principle, we get that for all $(s,y) \in
(0,t_n)\times J$ and all $(X
(\cdot),\alpha(\cdot))\in \mathcal{T}^{t_n,x_n}_{s,y}$,
$$u(t_n,x_n)\le E^{t_n}_s(X,\alpha)=u(s,X(s))+\int_s^{t_n} \ell(\tau,X(\tau),\alpha(\tau))\ d\tau.$$
This implies
$$\varphi(t_n,x_n) - o_n(1)\le \varphi(s,X(s))+\int_s^{t_n} \ell(\tau,X(\tau),\alpha(\tau))\ d\tau$$
\textit{i.e.}
$$K_s^{t_n}(X,\alpha) \le o_n(1)$$
with $K_s^{t_n}(X,\alpha)$ defined in (\ref{eq::pp77}).
\medskip

\noindent {\sc Step 2.2: freezing the coefficients.}
Using (\ref{eq::pp78}), this implies
\begin{equation}\label{eq::pp88}
\int_s^{t_n}d\tau \  \bar \kappa(\tau,X(\tau),\alpha(\tau)) 
= \bar K_s^{t_n}(X,\alpha) \le o_n(1)+  |t_n-s|\omega(o_n(1)+C_1|\bar t -s|)
\end{equation}
with $\bar \kappa$ defined in (\ref{eq::pp79}).

\noindent {\sc Step 2.3: inequalities for $i_0=1,\dots,N$.}
For each $i=1,\dots,N$, let us choose some $\bar \alpha_i, \underline{\alpha}_i\in \mathbb A_i$ such that
\begin{equation}\label{eq::pp80}
b_i(\bar t,0,\bar \alpha_i)<0\quad \mbox{and} \quad b_i(\bar t,0,\underline{\alpha}_i)>0.
\end{equation}
We now fix some index $i_0\in \left\{1,\dots,N\right\}$.

Assume first that $x_n\in J_j^*$ with $j\not=i_0$. Then we look for a solution with
terminal condition $X^n(t_n)=x_n$, which solves backward the following
ODE
$$\dot{X}^n(\tau)=b_j(\tau,{X}^n(\tau),\underline{\alpha}_j) \quad \mbox{for}\quad \tau<t_n$$
up to the first time $\tau_n^j$ where $X^n$ reaches the junction
point, where $\tau_n^j$ is precisely defined by
\begin{equation}\label{eq::pp87}
  \tau_n^j\in (0,t_n)\quad \mbox{such that}\quad  X^n(\tau_n^j)=0 
\quad \mbox{and}\quad X^n(\tau)\in J_j^* \quad \mbox{for all}\quad \tau\in (\tau_n^j,t_n].
\end{equation}
Note that such a trajectory $X^n(\cdot)$ always exists, even if it 
may not be unique, because $b_j$ is not Lipschitz in the space variable $x$.
By assumption (\ref{eq::pp80}) and the continuity of $b_j$, we know
that we will have $\tau_n^j \to \bar t$ as $n\to +\infty$.  Then we
consider some $\alpha^n(\cdot)\in L^\infty([s,t_n];\mathbb A)$ such
that
$$
\left\{\begin{array}{ll}
\alpha^n_{i_0}(\tau)=\bar \alpha_{i_0} & \quad \mbox{if}\quad \tau\in
[s,\tau_n^j], \\
\alpha^n_j(\tau)=\underline{\alpha}_j & \quad \mbox{if}\quad \tau\in (\tau_n^j,t_n].
\end{array}\right.
$$

Assume now that $x_n\in J_{i_0}$. In  this case, we require
$$\alpha^n_{i_0}(\tau) = \bar \alpha_{i_0} \quad \mbox{for all}\quad
\tau\in [s,t_n].$$

In both cases, we call $X^n(\cdot)$ a trajectory such that
$(X^n,\alpha^n)\in {\mathcal T}^{t_n,x_n}_{s,X^n(s)}$.  

Up to a subsequence, we get that $X^n$ converges uniformly towards
some $X$, and $\alpha^n$ converges to $\alpha=\bar \alpha_{i_0}$, such
that (using (\ref{eq::pp88})),
$$
|\bar t -s| \left\{\varphi_t(\bar t,0) 
+ \partial_{i_0}\varphi(\bar t,0)b_{i_0}(\bar t,0,\bar \alpha_{i_0}) 
-\ell_{i_0}(\bar t,0,\bar \alpha_{i_0})\right\}  
= \bar K_s^{\bar t}(X,\alpha) \le |\bar t-s|\omega(C_1|\bar t -s|).
$$
Dividing by $|\bar t-s|$ and passing to the limit $s\to \bar t$, and taking the supremum on $\bar \alpha_{i_0}\in \mathbb A_{i_0}$
such that $b_{i_0}(\bar t,0,\bar \alpha_{i_0})<0$, we get
\begin{equation}\label{eq::pp81}
\varphi_t(\bar t,0) + H_{i_0}^-(\bar t,\partial_{i_0}\varphi(\bar t,0))  \le 0.
\end{equation}

\noindent {\sc Step 2.4: inequality for $i_0=0$.}
We now assume that (\ref{eq::ppq50}) does not hold true.
Then (\ref{eq::pp81}) implies that
\begin{equation}\label{eq::pp90}
\varphi_t(\bar t,0) + H_0(\bar t) >0
\end{equation}
and
$$
H_0(\bar t)=\bar H_0(\bar t) >  
\displaystyle \max_{i=1,\dots,N} H_i^-(\bar t,\partial_i \varphi(\bar
t,0^+)) \ge A_0(\bar t).
$$
By continuity of $\bar H_0=\max(H_0,A_0)$ with $A_0$ continuous
defined in (\ref{eq::pp84}), we deduce that there exists some
$s_0<\bar t$ such that $H_0$ is continuous on $[s_0,\bar t]$. In
particular, we have $\mathbb A_0(\tau)\not=\emptyset$ for all $\tau\in
[s_0,\bar t]$.  By Lemma \ref{lem::pp85}, there exists a measurable
selection $\bar \alpha_0 \in L^\infty([s_0,\bar t];\mathbb A_0)$ such
that
$$
\bar \alpha_0(\tau)\in \mathbb A_0(\tau) \quad \mbox{and}\quad
H_0(\tau)=-\ell_0(\tau,\bar \alpha_0(\tau)) \quad \mbox{for a.e.}\quad
\tau\in [s_0,\bar t].
$$
If $x_n\in J_j^*$, we now use the defintion of $\tau_n^j$ given in (\ref{eq::pp87}) and 
consider some $\alpha^n(\cdot)\in L^\infty([s_0,t_n];\mathbb A)$ such that
$$\left\{\begin{array}{ll}
\alpha^n_j(\tau)=\underline{\alpha}_j & \quad \mbox{if}\quad \tau\in (\tau_n^j,t_n],\\
\alpha^n_{0}(\tau)=\bar \alpha_{0}(\tau) & \quad \mbox{if}\quad \tau\in [s_0,\tau_n^j].
\end{array}\right.$$
If $x_n=0$, then we simply choose some $\alpha^n(\cdot)\in
L^\infty([s_0,t_n];\mathbb A)$ such that
$$\alpha^n_{0}(\tau)=\bar \alpha_{0}(\tau) \quad \mbox{if}\quad \tau\in [s_0,t_n].$$
Let $s\in [s_0,\bar t)$.  In any cases, we call again $X^n(\cdot)$ a
trajectory such that $(X^n,\alpha^n)\in {\mathcal
  T}^{t_n,x_n}_{s,X^n(s)}$.  Similarly to Step 2.3, up to a
subsequence, we get that $X^n$ converges uniformly towards $X=0$, and
$\alpha^n$ converges to $\alpha=\bar \alpha_{i_0}$, such that (using
(\ref{eq::pp88})):
$$\begin{array}{ll}
|\bar t-s|\omega(C_1|\bar t -s|) & \ge  \bar K_s^{\bar t}(X,\alpha)\\
&= \displaystyle \int_s^{\bar t}d\tau\  \left\{\varphi_t(\bar t,0)  -\ell_{0}(\tau,\bar \alpha_{0}(\tau))\right\}\\
& = \displaystyle \int_s^{\bar t}d\tau\  \left\{\varphi_t(\bar t,0)  +H_{0}(\tau))\right\}\\
& \ge  |\bar t-s| \left\{\varphi_t(\bar t,0)  +H_{0}(\bar t) -\omega(|\bar t -s|))\right\}\\
\end{array}$$
where $\omega$ still denotes some modulus of continuity of $H_0$ on $[s_0,\bar t]$.
Dividing by $|\bar t-s|$ and passing to the limit $s\to \bar t$, we get
$$\varphi_t(\bar t,0)  +H_{0}(\bar t) \le 0$$
which contradicts (\ref{eq::pp90}). This finally shows that
(\ref{eq::ppq50})  holds true.

\paragraph{Step 3: checking the initial condition and a priori bounds.} 
From the fact that $u_0$ is  continuous and the fact that 
$b_i,\ell_i$ are bounded for $i=0,\dots,N$, we deduce easily from the
representation formula (\ref{eq::pr20}) that the value function $u$
satisfies
$$u^*(0,x)=u_0(x)=u_*(0,x) \quad \mbox{for all}\quad x\in J.$$
Again from the representation formula (\ref{eq::pr20}), the fact that 
$b_i,\ell_i$ are bounded for $i=0,\dots,N$, and the fact that $u_0$ is
globally Lipschitz continuous, we also easily see that there exists a constant
$C>0$ such that $|u(t,x)-u_0(x)|\le Ct$. In particular
\begin{equation}\label{eq::pp96}
|u(t,x)|\le C_T(1+d(x,0)) \quad \mbox{for all}\quad (t,x)\in [0,T]\times J.
\end{equation}

\paragraph{Step 4: conclusion.}
The previous steps show that $u$ solves (\ref{eq::pp42}) with initial condition (\ref{eq::pp43}).
We also have the sublinear property (\ref{eq::pp96}).
Then, we apply Proposition \ref{pro::pp95} (which is postponed)  
which claims that our PDE satisfies the assumptions of Corollary \ref{cor::l4}.
This implies the indentification of the function $u$ to the unique solution of (\ref{eq::pp42}), (\ref{eq::pp43}).
This ends the proof of the theorem.
\end{proof}
We now turn to proofs of Lemma~\ref{lem::pp85} and Proposition~\ref{pro::pp95}.
\begin{proof}[Proof of Lemma~\ref{lem::pp85}]
We consider the map $f: [a,b]\times \mathbb A_0 \to \R^2$ defined by
$$f(\tau,\alpha_0)=(b_0(\tau,\alpha_0),H_0(\tau)+\ell_0(\tau,\alpha_0))$$
Recall that by (\ref{assum:0cont}), we have $\mathbb A_0\subset
\R^{d_0}$, with $\mathbb A_0$ compact.  Then we define the
multifunction $\Gamma : [a,b]\rightrightarrows \mathbb \R^{d_0}$
defined by
$$\Gamma(\tau)=\left\{\alpha_0\in \mathbb A_0,\quad f(\tau,\alpha_0)=(0,0)\right\}$$
Because $f$ is continuous, $\Gamma(\tau)$ is closed. Moreover our
assumptions guarantee that $\Gamma(\tau)$ is nonempty.  We recall (see
\cite{rockafellar}, page 314, beginning of section 2) that $\Gamma$ is
said to be $\mathcal L$-measurable (Lebesgue measurable) if and only
if its graph
$$G(\Gamma)=\left\{(\tau,\alpha_0)\in [a,b]\times \mathbb
  \R^{d_0},\quad 
\alpha_0\in \Gamma(\tau)\right\}$$
is ${\mathcal L} \otimes {\mathcal B}$-mesurable, \textit{i.e.}
belongs to the $\sigma$-algebra generated by the product of Lebesgue
sets in $[a,b]$ and Borel sets in $\R^{d_0}$.  Here $G(\Gamma)
=f^{-1}((0,0))$ is a closed set of $[a,b]\times \mathbb \R^{d_0}$, so
this set is obviously $\mathcal L \otimes {\mathcal B}$-measurable.
We now apply the mesurable selection result cited as the corollary on
page 315 in \cite{rockafellar}.  This result states that for any
$\mathcal L$-measurable multifunction $\Gamma: [a,b]\rightrightarrows
\mathbb \R^{d_0}$, which is closed-valued with $\Gamma(\tau)$ nonempty
for almost every $\tau\in [a,b]$, there exists a $\mathcal
L$-measurable function $\bar \alpha_0 : [a,b]\to \R^{d_0}$ such that
$$\bar \alpha_0(\tau)\in \Gamma(\tau) \quad \mbox{for almost every}\quad \tau\in [a,b]$$
This implies the result stated in the lemma and ends its proof.
\end{proof}
\begin{proof}[Proof of Proposition~\ref{pro::pp95}]
We check successively all assumptions.

\noindent \textsc{Step 1: Checking  (H0) and (H3).}
We set
$$P=(t,x,p) \quad \mbox{and}\quad \Phi_i(\alpha_i,P)=p b_i(t,x,\alpha_i) - \ell_i(t,x,\alpha_i).$$
We recall that
$$H_i(P)=\sup_{\alpha_i\in \mathbb A_i} \Phi_i(\alpha_i,P)= \Phi_i(\bar \alpha_i(P),P).$$
Let $P'=(t',x',p')$. 
We assume that 
$$|p|,|q|\le L.$$
Using the fact that $b_i$, $\ell_i$ are uniformly continuous with
respect to $(t,x)$, uniformly with respect to $\alpha_i\in \mathbb
A_i$, we deduce that there exists a modulus of continuity
$\omega_{T,L}$ such that
$$H_i(P') \ge \Phi_i(\bar \alpha_i(P),P') \ge  \Phi_i(\bar
\alpha_i(P),P) - \omega_{T,L}(|P-P'|) = H_i(P) - \omega_{T,L}(|P-P'|).$$
Exchanging $P$ and $P'$, we get the reverse inequality, which yields
\begin{equation}\label{eq::pp102}
|H_i(P')-H_i(P)|\le \omega_{T,L}(|P-P'|)
\end{equation}
In particular, this gives the continuity of $H_i$.

\noindent \textsc{Step 2: Checking  (H1).}
By assumption (\ref{eq::pr19}), there exists some $\delta>0$ and
controls $\alpha_i^\pm=\alpha_i^\pm(t,x)$ such that
$$\pm b_i(t,x,\alpha_i^\pm)\ge \delta>0.$$
Using the fact that $\ell_i$ is bounded, this implies that
\begin{equation}\label{eq::pp99}
H_i(t,x,p)\ge \delta |p| -C
\end{equation}
for some constant $C>0$.

\noindent \textsc{Step 3: Checking  (H2).}
Again, using the boundedness of $b_i$ and $\ell_i$, we get the uniform coercivity estimate
\begin{equation}\label{eq::pp100}
|H_i(t,x,p)|\le C(|p| +1).
\end{equation}

\noindent \textsc{Step 4: Checking  (H4).}
The quasi-convexity of $H_i(t,x,\cdot)$ follows from its convexity.

\noindent \textsc{Step 5: Checking  (H5).}
We write with $p'=p$, $x'=x$, $\bar \alpha_i:= \bar \alpha_i(P')$
$$\begin{array}{ll}
H_i(P')-H_i(P) & =\Phi(\bar \alpha_i(P'),P')-H_i(P)\\
& \le \Phi(\bar \alpha_i,P')-\Phi(\bar \alpha_i,P)\\
& = p(b_i(t',x,\bar \alpha_i)-b_i(t,x,\bar \alpha_i)) 
-(\ell_i(t',x,\bar \alpha_i)-\ell_i(t,x,\bar \alpha_i))\\
& \le L|p| |t'-t| +\bar \omega(|t'-t|)\\
& \le L\delta^{-1} (C+\max(0,H_i(t,x,p)))|t'-t| +\bar \omega(|t'-t|)\\

\end{array}$$
where in the fourth line, we have used the fact that $b_i$ is
$L$-Lipschitz continuous (by (\ref{eq::toto})) with respect to $t$, uniformly with respect
to $\alpha_i$. We have also used the fact that there exists a modulus
of continuity $\bar \omega$ for $\ell_i$ with respect to $(t,x)$,
uniformly in $\alpha_i$. In the fifth line, we have used the uniform
coercivity estimate (\ref{eq::pp99}).  The previous inequality implies easily
(H5).

\noindent \textsc{Step 6: Checking  (H6).}
Recall that $H_i$ is uniformly coercive by (H1), and continuous by (H0).
This implies that the map $t\mapsto \min H_i(t,0,\cdot)$ is also continuous. 
This implies the continuity of 
$$A^0_0(t) = \max_{i=1,\dots,N}  \min H_i(t,0,\cdot).$$

\noindent \textsc{Step 7: Checking  (A0).}
The continuity of $A_0(t)=\bar H_0(t)$ follows from (\ref{assum:hot}).

\noindent \textsc{Step 8: Checking  (A1) and (A2).}
The bound on $A_0(t)$ and the uniform continuity of $A_0(t)$ are
trivial since there is only one vertex.

This ends the proof of the proposition.
\end{proof}

\section{Second application: study of Ishii solutions}
\label{sec:bbc}

This section is strongly inspired by the work \cite{bbc2} where one of
the main contribution of the authors was to identify the maximal and
minimal Ishii solutions (in any dimensions), in the framework of
convex Hamiltonians, and using tools of optimal control theory.  With
our PDE theory in hands, we revisit this problem in dimension one, but
for quasi-convex Hamiltonians (in the sense of (\ref{assum:H})) that
can be non-convex.  As a by-product of our approach, we give a PDE
characterization of both the maximal and the minimal Ishii solutions.

\begin{rem}\label{rem::gr80}
Combining results from Subsection \ref{svs} with the ones from this
Section, we can easily see that for one-dimensional problems, the
solutions in \cite{bbc}, \cite{bbc2}, \cite{rz} and \cite{rsz} fall
naturally in our theoretical framework; they coincide with some 
$A$-flux-limited solutions for $A$ well chosen.
\end{rem}

\subsection{The framework}

Let us consider two Hamiltonians $H_i$ for $i=1,2$ which are level-set
convex in the sense of (\ref{assum:H}).  In particular $H_i$ is
assumed to be minimal at $p_i^0$.

\paragraph{Ishii solutions on the real line.}
In \cite{bbc2}, Ishii solutions are considered. A function $u$ is said
to be a Ishii sub-solution if its upper semi-continuous envelope $u^*$
solves
\[\left\{\begin{array}{ll}
u_t + H_1(u_x) \le 0 &\quad \mbox{for } x <0,\\
u_t + H_2(u_x) \le 0 &\quad \mbox{for } x >0,\\
u_t + \min(H_1(u_x),H_2(u_x))\le 0  &\quad \mbox{for } x=0
\end{array}\right.\]
A function $u$ is said to be a Ishii super-solution if its lower
semi-continuous envelope $u_*$  solves
\[\left\{\begin{array}{ll}
u_t + H_1(u_x) \ge 0 &\quad \mbox{for } x <0,\\
u_t + H_2(u_x) \ge 0 &\quad \mbox{for } x >0,\\
u_t + \max(H_1(u_x),H_2(u_x))\ge 0 &\quad \mbox{for } x=0.
\end{array}\right.\]
An Ishii solution is a function $u$ which is both an Ishii
sub-solution and an Ishii super-solution. 

\paragraph{Translation of flux-limited solutions in the real line setting.}
The notion of solutions $\tilde{u}(t,x)$ from Section~\ref{s.v} on two
branches $J_1\cup J_2$ with two Hamiltonians
\[\tilde H_1(q)=H_1(-q)\quad \mbox{and}\quad \tilde H_2(q)=H_2(q)\]
is translated in the framework of the real line into functions $u$ defined for $(t,x)\in
\left[0,+\infty\right)\times \R$ by
\[u(t,x)=\left\{\begin{array}{ll}
\tilde u(t,x) &\quad \mbox{for}\quad  0\le x \in J_2,\\
\tilde u(t,-x)  &\quad \mbox{for}\quad  0\le -x \in J_1.\\
\end{array}\right.\]
Then $\tilde u$ solves \eqref{eq::1bis} with Hamiltonians $\tilde H_i$
if and only if $u$ solves
\begin{equation}\label{eq::pr1}
\left\{\begin{array}{ll}
u_t + H_1(u_x)=0 &\quad \mbox{for}\quad (t,x)\in (0,+\infty)\times (-\infty,0),\\
u_t + H_2(u_x)=0 &\quad \mbox{for}\quad (t,x)\in (0,+\infty)\times (0,+\infty),\\
u_t  + \check F_A(u_x(t,0^-),u_x(t,0^+))=0  &\quad \mbox{for}\quad (t,x)\in (0,+\infty)\times \left\{0\right\}
\end{array}\right.
\end{equation}
with
\[\check F_A (q_1,q_2) = \max (A, H_1^+(q_1),H_2^-(q_2))\] 
where
\[H_i^-(q)=\left\{\begin{array}{ll}
H_i(q) &\quad \mbox{if}\quad q < p_i^0,\\
H_i(p_i^0) &\quad \mbox{if}\quad q\le p_i^0
\end{array}\right.
\quad \mbox{and}\quad
H_i^+(q)=\left\{\begin{array}{ll}
H_i(p_i^0) &\quad \mbox{if}\quad q\le p_i^0,\\
H_i(q) &\quad \mbox{if}\quad q > p_i^0.
\end{array}\right.\]
We have the following correspondance 
\[ \tilde{H}_1^\pm (p_1) = H_1^\mp (-p_1) \quad \text{ and } \quad \tilde{H}_2^\pm (p_2) = H_2^\pm (p_2). \]

Viscosity inequalities are now naturally written by touching 
$u$ with test functions $\phi: [0,+\infty) \times \R \to \R$ that are
  continuous, and $C^1$ in $[0,+\infty) \times (-\infty,0]$ and in
  $[0,+\infty) \times [0,+\infty)$.  
  
\paragraph{Ishii flux-limiters.}
We recall the quantity
$$\displaystyle A_0 = \max_{i=1,2} \left(\min_{q\in
  \R} H_i(q)\right)=\max_{i=1,2}  H_i(p^0_i).$$
and define
\[ A^* =  \max_{q \in \mbox{ch}\left[p_1^0,p_2^0\right]}
(\min(H_1 (q),H_2 (q))).\]
with the chord
\[\mbox{ch}\left[p_1^0,p_2^0\right]=[\min (p_1^0,p_2^0), \max (p_1^0,p_2^0)].\]
Then we set
\begin{equation}\label{eq::gr57}
A_I^+ = \max (A^*,A_0)
\end{equation}
and 
\begin{equation}\label{eq::gr58}
A_I^- = \left\{\begin{array}{ll}
A_I^+ & \quad \mbox{if}\quad p_2^0< p_1^0,\\
A_0 & \quad \mbox{if}\quad p_2^0\ge  p_1^0,\\
\end{array}\right.
\end{equation} 

\begin{rem}\label{rem::gr60}
Notice that even if the points of minimum $p_i^0$ of $H_i$ may be not unique,
it is easy to see that the quantities $A_I^\pm$ are uniquely defined.
\end{rem}
These two quantities $A_I^\pm$ will play a crucial role here; they
have been identified first in \cite{bbc2}, in a different way (see
below).  

\subsection{Identification of maximal and minimal Ishii solutions}

The main result of this section is the following.
\begin{theo}[Identification of maximal and minimal Ishii
    solutions] \label{th::gr59} We assume that the Hamiltonians $H_i$
  satisfy \eqref{assum:H} for $i=1,2$.  We have $A_I^- \le A_I^+$
  and the following holds.
\begin{enumerate}[\upshape i)]
\item \label{pro::i4} {\upshape(Ishii sub-solution)}
Every Ishii sub-solution is a $\check{F}_{A_I^-}$-sub-solution.
\item \label{pro::i5} {\upshape(Ishii super-solution)}
Every Ishii super-solution is a $\check{F}_{A_I^+}$-super-solution.
\item \label{pro::i6} {\upshape(Particular Ishii solutions)}
Every $\check{F}_{A}$-solution is a Ishii solution if $A\in \left[A_I^-, A_I^+\right]$.
\item \label{pro::i7} {\upshape(Maximal and minimal Ishii solutions)} For a given
   uniformly continuous initial data, the
  $\check{F}_{A_I^+}$-solution is the minimal Ishii solution, and the
  $\check{F}_{A_I^-}$-solution is the maximal Ishii solution.
  Moreover the Ishii solution is unique if and only if
  $A_I^+=A_I^-$.
\end{enumerate}
\end{theo}
We prove successively i)-iv) from Theorem~\ref{th::gr59}.
\begin{proof}[Proof of Theorem~\ref{th::gr59}-\ref{pro::i4})] Let $u$
  be a Ishii sub-solution. We want to check that $u$ is a
  $\check F_{A_I^-}$-sub-solution. Lemma~\ref{lem:wc-c1} implies the
  ``weak continuity'' condition.  The only difficulty is on the
  junction point $x=0$.  If $A_I^-=A_0$, then the result follows from
  Theorem \ref{th::gr1} i).

Assume now that 
\[A_I^->A_0.\]
Then $A_I^-=A^*$, and $p_2^0<p_1^0$. In particular, we can choose
$p^*\in \left[p_2^0,p_1^0\right]$ such that
\begin{equation}\label{eq::gr61}
H_1(p^*)=H_1^+(p^*)=A^*=A_I^-=H_2(p^*)=H_2^-(p^*).
\end{equation}
Now from Theorem \ref{th::gr1} i), we see that, in order to show that
$u$ is a $\check F_{A_I^-}$-sub-solution, it is sufficient to consider
a test function $\varphi$ touching $u$ from above at $(t_0,0)$ for
$t_0>0$, with
\[\varphi(t,x)=\psi(t) + p^*x\]
with $\psi\in C^1$, and to show that 
\begin{equation}\label{eq::gr62}
\varphi_t + A_I^-\le 0\quad \mbox{at}\quad (t_0,0).
\end{equation}
Indeed, such $\varphi$ is now an admissible test function for Ishii
sub-solutions. So we deduce that
\[\varphi_t + \min(H_1^+(\varphi_x(t_0,0^-)),\  
H_2^-(\varphi_x(t_0,0^+)))\le 0 \quad \mbox{at}\quad (t_0,0)\] which
implies (\ref{eq::gr62}). We conclude that $u$ is a $\check
F_{A_I^-}$-sub-solution and this ends the proof.
\end{proof}
\begin{proof}[Proof of Theorem~\ref{th::gr59}-\ref{pro::i5})]
Let $u$ be a Ishii super-solution. We want to show that $u$ is a
$\check F_{A_I^+}$-super-solution.

\paragraph{Step 1: preliminaries.}
We distinguish two cases.

\noindent {\sc Case 1: $A^*\ge A_0$.}  Then we have $A_I^+=A^*$. In
particular, there exists $p^*\in \mbox{ch}\left[p_1^0,p_2^0\right]$
such that 
\begin{equation}\label{eq:a*}
A^* = H_1 (p^*)= H_2 (p^*).
\end{equation}
 We set
\begin{equation}\label{eq::gr70}
\varphi(t,x):=\psi(t)+ p^*x=:\tilde{\varphi}(t,x)
\end{equation}
with $\psi\in C^1$.

\noindent {\sc Case 2: $A^*<A_0$.}  This implies that there is a
unique $\alpha\in \left\{1,2\right\}$ such that
\[A_I^+=A_0=H_\alpha(p_\alpha^0)\]
and  for $\bar \alpha\in  \left\{1,2\right\}\backslash \left\{\alpha\right\}$ we have
\[
H_\alpha(p_\alpha^0)> H_{\bar \alpha}(p_\alpha^0).
\]
In particular,
\begin{equation}\label{eq::gr74}
\max (H_\alpha(p_\alpha^0), H_{\bar \alpha}(p_\alpha^0)) = A_I^+.
\end{equation}

If $\alpha=1$, then we set $(p_1,p_2) = (p_1^0,\pi_2^+(A_0))$; if $\alpha=2$, then
we set $(p_1,p_2) = (\pi_1^-(A_0),p_2^0)$. 
We remark that we have
\[H_2(p_2)=H_2^+(p_2)=A_0=A_I^+=H_1(p_1)=H_1^-(p_1)\]
and
\[p_2>p_1.\]
We set
\begin{equation}\label{eq::gr71}
\varphi(t,x):=\psi(t) + p_1 x 1_{\left\{x<0\right\}}+ p_2 x
1_{\left\{x>0\right\}} \ge \tilde{\varphi}(t,x) := \psi(t) + p_\alpha^0 x
\end{equation}
with $\psi\in C^1$.

\paragraph{Step 2: conclusion.}
Now from Theorem \ref{th::gr1} iii), we see that, in order to show that
$u$ is a $\check F_{A_I^+}$-super-solution, it is sufficient to consider
a test function $\varphi$ (given either in (\ref{eq::gr70}) in case 1
or (\ref{eq::gr71}) in case 2) touching $u$ from below at $(t_0,0)$
for $t_0>0$, and to show that
\begin{equation}\label{eq::gr73}
\varphi_t + A_I^+\ge 0\quad \mbox{at}\quad (t_0,0)
\end{equation}
Because we have $\varphi\ge \tilde{\varphi}$ with equality at
$(t_0,0)$, we deduce that $\tilde{\varphi}$ is an admissible test
function for the Ishii super-solution $u$.  Therefore, we have
$$\tilde{\varphi}_t +
\max(H_1(\tilde{\varphi}_x),H_2(\tilde{\varphi}_x))\ge 0 \quad
\mbox{at}\quad (t_0,0)$$ Using either \eqref{eq:a*} in case 1, or
(\ref{eq::gr74}) in case 2, we deduce that
$$\psi_t + A_I^+\ge 0 \quad \mbox{at}\quad (t_0,0)$$ which implies
(\ref{eq::gr73}). This implies that $u$ is a $\check
F_{A_I^+}$-super-solution and ends the proof.
\end{proof}
We now state and prove a proposition which is more precise than
Theorem~\ref{th::gr59}-\ref{pro::i6}).
\begin{pro}[Relation between $\check F_A$ and Ishii  sub/super-solutions]\label{pro::pr4}
Under the assumptions of Theorem \ref{th::gr59}, every $\check
F_A$-sub\-solution (resp. $\check F_A$-super-solution) is a Ishii
sub-solution (resp. Ishii super-solution) if $A\ge A_I^-$
(resp. $A\le A_I^+$).

  Moreover for every $A\in \left[A_0, A_I^-\right)$, there exists
    a $\check F_A$-sub-solution which is not a Ishii sub-solution.
    For every $A>A_I^+$, there exists a $\check F_A$-super-solution
    which is not a Ishii super-solution.
\end{pro}
\begin{proof}
We treat successively sub-solutions and super-solutions. 

\textsc{Sub-Solutions.} Let $u$ be a $\check F_A$-sub-solution with $A
\ge A_I^-$.  Consider a $C^1$ function $\phi$ touching
$u$ from above at $(t,0)$ for some $t>0$. Then
\[ \lambda + \check F_A (q,q) \le 0 \]
where $\lambda = \partial_t \phi (t,0)$ and $q = \partial_x \phi
(t,0)$. In particular, $\lambda +A \le 0$.
We want to prove that 
\[ \lambda + \min (H_1(q),H_2 (q)) \le 0. \]
If $q \le p_2^0$, then 
\[ \min (H_1(q),H_2(q)) \le H_2^-(q) \le \check F_A (q,q) \le - \lambda.\]
Similarly, if  $q \ge p_1^0$, then 
\[ \min (H_1(q),H_2(q)) \le  H_1^+(q) \le  \check F_A (q,q) \le - \lambda.\]
If $p_2^0<p_1^0$, and $q\in \left[p_2^0,p_1^0\right]$, then by
definition of $A^*$, we have
\[ \min (H_1(q),H_2 (q)) \le A^* \le A_I^+=A_I^-\le A \le - \lambda.\]
This shows that $u$ is a Ishii sub-solution.

If $A^* \le A_0$ or $p_2^0\ge p_1^0$, there is nothing additional to
prove.  Assume now that $p_2^0< p_1^0$ with $A_I^-=A^* > A_0$, and
we claim that for any $A \in \left[A_0,A_I^-\right)=
  \left[A_0,A^*\right)$, there exists a $\check F_A$-sub-solution which
    is not an Ishii sub-solution. Indeed, let us consider $p^*\in
    \left[p_2^0,p_1^0\right]$ such that
$$A^*=H_1(p^*)=H_2(p^*).$$
Then there exists $p_2^0\le p_2<p^*<p_1\le p_1^0$ such that
\begin{equation}\label{eq::pr5}
A=H_1(p_1)=H_2(p_2)= \check F_A(p_1,p_2)
\end{equation}
Let us now consider 
\[ u (t,x) = - A t  + p_1 x 1_{\left\{x<0\right\}} + p_2 x1_{\left\{x\ge 0\right\}} \]
In particular $u$ is $\check F_A$-sub-solution because of (\ref{eq::pr5}).
Now the test function $\phi (t,x) = -A t + p^* x$ touches $u$ at
$(t,0)$ from above and  does not satisfy
the inequality
\[ \partial_t \phi (t,0) + \min (H_1(\partial_x \phi (t,0)), H_2
(\partial_x \phi (t,0))) \le 0.\]
This shows that $u$ is not a Ishii sub-solution. \smallskip

\textsc{Super-Solutions.} Let $u$ be a $\check F_A$-super-solution with
$A \le A_I^+$.  Consider a $C^1$ function $\phi: \R \to \R$ touching $u$
from below at $(t,0)$ for some $t>0$. Then
\[ \lambda + F_A (q,q) \ge 0 \]
where $\lambda = \partial_t \phi (t,0)$ and $q = \partial_x \phi
(t,0)$. Without loss of generality, we can assume that $A\ge A_0$. We want to prove that 
\[ \lambda + \max (H_1(q),H_2 (q)) \ge 0.\]
If $F_A (q,q) = A$, then we deduce from Lemma~\ref{A*} below that 
\[ 0 \le \lambda + A \le \lambda + A_I^+ \le \lambda + \max (H_1(q),H_2 (q)). \]
If now $F_A (q,q) = H_1^+(q)$, then 
\[ 0 \le \lambda + F_A (q,q) \le \lambda + H_1 (q) \le \lambda + \max
(H_1(q),H_2(q)).\] 
If finally $F_A (q,q) = H_2^-(q)$, then 
\[ 0 \le \lambda + F_A (q,q) \le \lambda + H_2 (q) \le \lambda + \max
(H_1(q),H_2(q)).\] 
This shows that $u$ is a Ishii super-solution.

Assume next that $A > A_I^+$.
If $A^*\ge A_0$, let $p^*\in \mbox{ch}\left[p_1^0,p_2^0\right]$ such that
\[A^*=H_1(p^*)=H_2(p^*).\]
Let us choose an index $\alpha\in \left\{1,2\right\}$ such that
\[\max_{i=1,2} H_i(p_i^0)= H_\alpha(p_\alpha^0).\]
Then we set
\[\bar p = \left\{\begin{array}{ll}
p^* & \quad \mbox{if}\quad A^*\ge A_0,\\
p_1^0 & \quad \mbox{if}\quad A^*< A_0 \quad \mbox{and}\quad \alpha=1,\\
p_2^0 & \quad \mbox{if}\quad A^*< A_0 \quad \mbox{and}\quad \alpha=2.
\end{array}\right.\]
In particular we have
\begin{equation}\label{eq::pr6}
\max(H_1(\bar p),H_2(\bar p))=A_I^+.
\end{equation}
Then for $A>A_I^+$, there exist $p_1$ and $p_2$ such that
\[ p_2\ge \max(p_1^0,p_2^0)\ge \bar p\ge \min(p_1^0,p_2^0)\ge p_1 \] 
and 
\[H_2(p_2)=A=H_1(p_1).\]
Let us now define
\[ u (t,x) = - A t + p_1 x1_{\left\{x<0\right\}} + p_2 x1_{\left\{x\ge 0\right\}} .\]
Then $u$ is a $\check F_A$-super-solution because $\check F_A(p_1,p_2)=A$.
Now the test function $\phi (t,x) = -A t + \bar p x$ touches $u$ at
$(t,0)$ from below and  does not satisfy
the inequality
\[ \partial_t \phi (t,0) + \max (H_1(\partial_x \phi (t,0)), H_2
(\partial_x \phi (t,0))) \ge 0\] because of (\ref{eq::pr6}).  This
shows that $u$ is not a Ishii super-solution. This achieves the
proof.\end{proof}

In the previous proof, we used the following elementary lemma. 
\begin{lem}[Bound from above for $A_I^+$]\label{A*}
For all $q \in \R$, $A_I^+ \le \max (H_1(q),H_2(q))$.
\end{lem}
\begin{proof}
We recall that $A_I^+=\max(A^*,A_0)$.  Assume first that
$\max(A^*,A_0)=A_0$, then $A_0 = \min H_\alpha$ for some $\alpha \in
\{1,2\}$. In particular, for all $q \in \R$, we have $A_I^+=A_0 \le H_\alpha
(q) \le \max (H_1(q),H_2(q))$.\\ If now $\max(A^*,A_0)=A^* > A_0$,
then there exists $p^* \in [p_i^0,p_j^0]$ for some $i,j \in \{1,2\}$
($i \neq j$), such that
\[ A^* = H_i(p^*) = H_j(p^*).\]
Moreover, $H_j$ is non-increasing in $(-\infty,p^*]$ hence 
\[ H_j (q) \ge A^* \text{ for } q \le p^*;\]
similarly, $H_i$ is non-decreasing in $[p^*,+\infty)$ hence 
\[ H_i (q) \ge A^* \text{ for } q \ge p^*.\] 
This implies the expected inequality. 
\end{proof}

We finally state a proposition which implies Theorem~\ref{th::gr59}-\ref{pro::i7}).
\begin{cor}[Conditions for uniqueness of Ishii solution]\label{cor::pr16}
We work under the assumptions of Theorem \ref{th::gr59}. Recall that
$A_I^+\ge A_I^-$, and let $g$ be a uniformly continuous initial data.
\begin{itemize}
\item
If $A_I^+=A_I^-$, then there is uniqueness of the Ishii solution with
initial data $g$.  
\item
If $A_I^+> A_I^-$, then there exists a Lipschitz
continuous initial data $g$ such that there are two different Ishii
solutions with the same initial data $g$.
\end{itemize}
\end{cor}
\begin{proof}
If $A_I^+=A_I^-$, then Theorem \ref{th::gr59} \ref{pro::i4}) and \ref{pro::i5})
imply that every Ishii solution $u$ is a $\check{F}_A$-solution for
$A=A_I^+$.  Given some uniformly continuous initial data, such a
solution is then unique.

 On the contrary, if $A_I^+> A_I^-$, then
\[U^-(t,x)=-At + p_1 x 1_{\left\{x<0\right\}}+ p_2 x 1_{\left\{x\ge 0\right\}}\]
is a $\check F_A$-solution with $A=A_I^+$ with initial data
$g(x)=U^-(0,x)$ if
\[A_I^+=A=H_1(p_1)=H_2(p_2),\quad p_2\ge p_2^0,\quad p_1\le p_1^0.\]
On the other hand, $U^-$ is not a $\check F_{A_I^-}$-solution
because $\check F_{A_I^-}(p_1,p_2)=A_I^-<A_I^+$.
\end{proof}

\subsection{Link with regional control}

In this subsection, we shed light on the consequence of our results in
the interpretation of the results from \cite{bbc2} when both
frameworks coincide. Roughly speaking, the one-dimensional framework
from \cite{bbc2} reduces to our framework with two branches. In this
case, the value function ${U}^-$ defined in \cite[Eq.~(2.7)]{bbc2}
(see also (\ref{eq::pr17}) in the present paper) and characterized in
\cite[Theorem~4.4]{bbc2} corresponds to the unique solution of
\eqref{eq::1bis} for $A=A_I^+$.  Similarly, the function ${U}^+$
defined in \cite[Eq.~(2.8)]{bbc2} (see also (\ref{eq::pr18}) in the
present paper) corresponds to the unique solution of \eqref{eq::1bis}
for $A=A_I^-$. This is shown in this subsection. We also provide the
link between our definition of $A_I^+$ and $A_I^-$ and the tangential
Hamiltonians introduced in \cite{bbc2}, coming from optimal control
theory.

\subsubsection{The optimal control framework}
\label{subsub}

The one dimensional framework of \cite{bbc2} corresponds to 
\[\Omega_1
= (-\infty,0), \quad \mathcal{H}=\{0\}, \quad \Omega_2 =
(0,+\infty).\] 
In this case, $(\mathbf{H}_\Omega)$ in \cite{bbc2}  is satisfied. We refer to this
framework as \textit{the common framework}. 

\paragraph{Hamiltonians.} 
As far as the Hamiltonian is concerned, the $(t,x)$-dependence is not
relevant for what we discuss now; for this reason we consider the
simplified case of convex Hamiltonians given for $i=1,2$ by
\[ H_i (p) = \sup_{\alpha_i \in A_i} ( - b_i (\alpha_i) p - \ell_i
(\alpha_i)) \]
for some compact metric space $A_i$ and $b_i,\ell_i: A_i \to \R$. In this
simplified framework, $(\mathbf{H}_C)$ reduces to the following
assumptions for $i=1,2$: 
\begin{equation}
\left\{\begin{array}{l}
b_i \text { and } \ell_i \text{ are continuous and bounded }\medskip\\
\{ (b_i (\alpha_i),\ell_i (\alpha_i)): \alpha_i \in A_i \} \text{ is
  closed and convex} \medskip \\
B_i = \{ -b_i (\alpha_i): \alpha_i \in A_i \} \text{ contains } [-\delta,\delta].
\end{array}
\right.
\end{equation}
In particular, we see that $B_i$ is a compact interval.
Introducing the Legendre-Fenchel transform $L_i$ of $H_i$, it is possible to see that
this problem can be reformulated by assuming that for $i=1,2$
\[ H_i (p) = \sup_{q \in B_i} ( q p - L_i (q))\]
where $L_i: B_i \to \R$ is convex where we recall that $B_i$ is a
compact interval containing $[-\delta,\delta]$. Indeed the graph of
$L_i$ on $B_i$ is the lower boundary of the closed convex set $\{ (b_i
(\alpha_i),\ell_i (\alpha_i)): \alpha_i \in A_i \}$ in the plane
$\R^2$. In particular, we see that $H_i$ is convex, Lipschitz
continuous and $H_i (p)\to + \infty$ as $|p| \to +\infty$. This last
fact comes from the fact that $\pm \delta \in B_i$.  Moreover $H_i$
reaches its minimum at any convex subgradient $p_i^0$ of $L_i$ at $0$
and satisfies
\[\left\{\begin{array}{l}
H_i \quad \mbox{is non-increasing on}\quad (-\infty,p_i^0],\\
H_i \quad \mbox{is non-decreasing on}\quad [p_i^0, +\infty).
\end{array}\right.\]
Hence, $H_i$ satisfies \eqref{assum:H}.

\paragraph{Tangential Hamiltonians.}
Using notation similar to the one of \cite{bbc2}, we define
$$\hat A =  A_1 \times A_2 \times [0,1]$$
Now, for $a=(\alpha_1,\alpha_2,\mu)\in \hat A$, we define
\[\left\{\begin{array}{l}
b_{\mathcal H}(a)=\mu b_1(\alpha_1)+(1-\mu)b_2(\alpha_2),\\
\ell_{\mathcal H}(a)=\mu \ell_1(\alpha_1)+(1-\mu)\ell_2(\alpha_2)
\end{array}\right.\]
and set
\begin{align*}
 {\hat A}_0 &= \{ a=(\alpha_1,\alpha_2,\mu) \in \hat A:
0 = b_{\mathcal H}(a)\}, \\
{\hat A}_0^{\mathrm{reg}} &= \{ a=(\alpha_1,\alpha_2,\mu) \in \hat A:
b_1(\alpha_1) \le 0, b_2(\alpha_2) \ge 0 \text{ and } 0 = b_{\mathcal H}(a)\}. 
\end{align*}
In the common framework, the tangential Hamiltonians  
given in \cite{bbc2} reduce to constants, and we can see that we can write them as follows
\begin{equation}\label{eq::pp1}
\left\{\begin{array}{ll}
H_T & = \displaystyle \sup_{a=(\alpha_1,\alpha_2,\mu)\in {\hat A}_0} (- \ell_{\mathcal H}(a) ),\\
H_T^{\mathrm{reg}} & = \displaystyle \sup_{a=(\alpha_1,\alpha_2,\mu)\in {\hat  A}_0^{\mathrm{reg}}} (- \ell_{\mathcal H}(a) ).
\end{array}\right.
\end{equation}

\paragraph{The value functions $U^-$ and $U^+$.}
We consider the following initial condition
\[u(0,x)=g(x) \quad \mbox{for}\quad x\in \R\]
with $g$ globally Lipschitz continuous.

For $a=(\alpha_1,\alpha_2,\mu)\in \hat A$, and for $x\in \R$, we set
\[b(x,a)=\left\{\begin{array}{ll}
b_1(\alpha_1) & \quad \mbox{if}\quad x\in (-\infty,0)=\Omega_1,\\
b_2(\alpha_2) & \quad \mbox{if}\quad x\in (0,+\infty)=\Omega_2,\\
b_{\mathcal H}(a) & \quad \mbox{if}\quad x\in {\mathcal H}=\left\{0\right\}
\end{array}\right.\]
and
\[\ell(x,a)=\left\{\begin{array}{ll}
\ell_1(\alpha_1) & \quad \mbox{if}\quad x\in (-\infty,0)=\Omega_1,\\
\ell_2(\alpha_2) & \quad \mbox{if}\quad x\in (0,+\infty)=\Omega_2,\\
\ell_{\mathcal H}(a) & \quad \mbox{if}\quad x\in {\mathcal H}=\left\{0\right\}.
\end{array}\right.\]
We consider admissible controlled dynamics starting from the point $(0,x)$ and ending at time $t>0$ defined by
\[{\mathcal T}_{t,x}=\left\{\begin{array}{l}
(X(\cdot),a(\cdot))\in \mbox{Lip}(0,t;\R)\times L^\infty(0,t;\hat A) \quad \mbox{such that}\medskip \\
\left\{\begin{array}{l}
X(0)=x,\\
\dot{X}(s)=b(X(s),a(s)) \quad \mbox{for a.e.}\quad s\in (0,t)
\end{array}\right.
\end{array}\right\}\]
and define the set of regular controlled dynamics as
\[{\mathcal T}^{reg}_{t,x}=\left\{\begin{array}{l}
(X(\cdot),a(\cdot))\in  {\mathcal T}_{t,x}\quad \mbox{such
  that}\medskip \\
a(s)\in {\hat A}_0^{reg} \quad \mbox{for a.e.}\quad s\in (0,t) \quad \mbox{such that} \quad X(s)=0
\end{array}\right\}.\]
Notice that the definition of ${\mathcal T}_{t,x}$ differs from the
one given in (\ref{eq::pr60}), where now $X$ takes the value $x$ at
time $0$ instead of at time $t$.  Then we define
\begin{equation}\label{eq::pr17}
U^-(x,t)=\inf_{(X(\cdot),a(\cdot))\in {\mathcal T}_{t,x}}\left\{g(X(t))+ \int_0^t \ell(X(s),a(s))\ ds\right\}
\end{equation}
and
\begin{equation}\label{eq::pr18}
U^+(x,t)=\inf_{(X(\cdot),a(\cdot))\in {\mathcal T}^{reg}_{t,x}}\left\{g(X(t))+ \int_0^t \ell(X(s),a(s))\ ds\right\}.
\end{equation}
Then we have the following characterization of $U^-$ and $U^+$:
\begin{theo}[Characterization of $U^-$ and $U^+$]\label{th::pr8}
Under the previous assumptions, $U^-$ is the unique $\check
F_A$-solution with initial data $g$ for $A=H_T$.  Similarly, $U^+$ is
the unique $\check F_A$-solution with initial data $g$ for
$A=H_T^{reg}$.
\end{theo}
\begin{proof}
Theorem \ref{th::pr8} is a straightforward application of Theorem
\ref{th::21}. 
\end{proof}

\subsubsection{Tangential Hamiltonians and Ishii flux-limiters}

In this paragraph, we show that the tangential Hamiltonians from
\cite{bbc2} coincide with the Ishii flux-limiters. 

We start with  defining
\begin{align*}
 {\mathcal A} & =  B_1 \times B_2 \times [0,1], \\
 {\mathcal A}_0 &= \{ (v_1,v_2,\mu) \in {\mathcal A}:
v_1 v_2 \le 0 \text{ and } 0 = \mu v_1 + (1-\mu)
v_2 \}, \\
{\mathcal A}_0^{\mathrm{reg}} &= \{ (v_1,v_2,\mu) \in {\mathcal A}:
v_1 \le 0, v_2 \ge 0 \text{ and } 0 = \mu v_1 + (1-\mu)
v_2 \}. 
\end{align*}
Then we can see (with $v_i=b_i(\alpha_i)$) that the tangential
Hamiltonians given in (\ref{eq::pp1}) can be written as follows
\begin{align*}
H_T & = \sup_{(v_1,v_2,\mu)\in {\mathcal A}_0} (- \mu L_1 (v_1) - (1-\mu) L_2 (v_2) ),\\
H_T^{\mathrm{reg}} & = \sup_{(v_1,v_2,\mu)\in {\mathcal A}_0^{\mathrm{reg}}} (- \mu L_1 (v_1) - (1-\mu) L_2 (v_2) ).
\end{align*}
Indeed, we use here the construction of $L_1$ and $L_2$ explained in
the previous Paragraph~\ref{subsub}. In particular, for $-b_i \in B_i$,
there exists $\alpha_i \in A_i$ such that $v_i = -b (\alpha_i)$. There
are several possible $\alpha_i$ and hence several possible
$\ell_i (\alpha_i)$. The construction $L_i (v_i)=\ell_i (\alpha^*_i)$ 
which is smaller than all the possible $\ell_i (\alpha_i)$.

\begin{pro}[Characterization of $H_T$]\label{pro:ht}
\[ H_T = A_I^+ .\]
\end{pro}
\begin{proof}
\textsc{Reduction.}
Remark that there exists $p_c \in
\R$ such that $A_I^+ = H_{i_c} (p_c)$ for some $i_c \in \{1,2\}$. We then
consider 
\[ \tilde{H}_i (v_i) = H_i (p_c + v_i) - A_I^+. \]
In this case, using obvious notation, $\tilde{A}_I^+ = 0$ and
$\tilde{p}_c = 0$. 
Remark that 
\begin{align*}
\tilde{L}_i (v_i)&= \sup_q ( v_i q - \tilde{H}_i (q) )  \\
& = \sup_q ( v_iq - H_i (p_c + q)) + A_I^+ \\
& = \sup_q (v_i q - H_i (q))  - p_c v_i + A_I^+\\
& = L_i (v_i) -p_c v_i + A_I^+. 
\end{align*}
Then
\begin{align*}
 \tilde{H}_T & = \sup_{(v_1,v_2,\mu)\in A_0} (- \mu
\tilde{L}_1 (v_1) - (1-\mu) \tilde{L}_2 (v_2) )  \\
& = \sup_{(v_1,v_2,\mu)\in A_0} (- \mu
L_1 (v_1) - (1-\mu) L_2 (v_2) ) -A_I^+ \\
& = H_T - A_I^+. 
\end{align*}
Hence, it is enough to prove 
\[\tilde{H}_T = 0.\] 

From now on, we assume that $A_I^+ = 0$ and $p_c =0$. We distinguish two
cases. \medskip

\textsc{First case.}
Assume first that $0= A_I^+ = A^* \ge A_0$. Then $0=A^*=H_1(p^*)=H_2(p^*)=H_{i_c}(p_c)$
with $p^*\in \mbox{ch}\left[p_1^0,p_2^0\right]$. Choosing initially
$p_c=p^*$, we can assume that
$A^* =H_1 (0) = H_2(0)=0$. In particular, $L_1 \ge 0$ and $L_2 \ge 0$. Hence $H_T \le
0$. To get the reverse inequality, we observe that there exists 
$v_i^* \in \partial H_i (0)$, $i=1,2$, with 
\[ v_1^* v_2^* \le 0. \] 
Indeed, if this is not true,  this
implies that for all $v_i \in \partial H_i (0)$, 
\[ v_1 v_2 > 0 \]
which is impossible because the graphs of $H_1$ and $H_2$ cross at $p^*$ and 
$p^*$ lies between $p_1^0$ and $p_2^0$ where $H_1$ and $H_2$ reach their minimum.

Pick now $\mu \in [0,1]$ such that $\mu v_1^* +
(1-\mu)v_2^*=0$. Then $(v_1^*,v_2^*,\mu) \in {\mathcal A}_0$ and
consequently, 
\[ H_T \ge - \mu L_1 (v_1^*) - (1-\mu) L_2 (v_2^*) =  \mu
H_1 (0) + (1-\mu) H_2 (0) = 0. \]
Hence $H_T=0$ in the first case, as desired.

\textsc{Second case.}  We now assume that $0 = A_I^+=A_0 > A^*$. In this
case, there exists $a \in \{1,2\}$ such that
\[  \min H_{a}  = H_{a}(0) = 0,\]
with the initial choice $p_c=p_a^0$.
This implies in particular 
\[ L_{a} \ge L_{a}(0)=0.\]
Moreover, for $b \neq a$, 
\[ \min L_b = - H_b (0) \ge 0,\] 
where we have used the fact that $A^*< A_0$.
Hence, $L_{a} \ge 0$ and $L_b \ge 0$ and consequently, $H_T \le 0$. 
Moreover with $v_i^*\in \partial H_i(0)$, we have,  $(0,v_2^*,1) \in {\mathcal A}_0$ when $a = 1$ 
and  $(v_1^*,0,0) \in {\mathcal A}_0$ when $a =2$. Hence, in both cases,
\[ H_T \ge -L_a (0) = 0 .\]
Hence $H_T=0$ in the second case too. The proof is now complete.
\end{proof}

\begin{pro}[Characterization of $H_T^{reg}$]\label{pro:htreg}
\[H_T^{reg} = A_I^-.\]
\end{pro}
\begin{proof}
The proof is similar to the proof of Proposition \ref{pro:ht}.  We make
precise how to adapt it.

\textsc{Reduction.} The reduction to the case $A_I^- = 0$ and $p_c =0$
is completely analogous.  We now have to prove that
$H_T^{reg}=0$.\smallskip

\textsc{First case.}  Assume first that $0= A_I^- = A^* \ge A_0$.  Note
that this case only makes sense either when $p_2^0<p_1^0$ or when
$p_2^0 \ge p_1^0$ and $0= A_I^- = A^* = A_0$. Similarly, we get
$H_T^{reg}\le 0$.  To get the reverse inequality, we observe that
there exists $v_i^* \in \partial H_i (0)$, $i=1,2$, with
\[ v_1^* v_2^* \le 0. \]  
We deduce that we can choose $v_2^*\ge 0$ and $v_1^*\le 0$,
both in the case $p_2^0<p_1^0$ and the case $p_2^0 \ge p_1^0$ and $0=
A_I^- = A^* = A_0$.  This implies that we can find
$(v_1^*,v_2^*,\mu)\in {\mathcal A}_0^{reg}$ and similarly,
we conclude that $H_T^{reg}\ge 0$.  Hence $H_T=0$ in the first case,
as desired.

\textsc{Second case.}  We now assume that $0 = A_I^-=A_0$.
We set again for some $a\in \left\{1,2\right\}$:
\[\min H_{a}  = H_{a}(0) = 0.\]
From our definition of $a$, we have again
\[ L_{a} \ge L_{a}(0)=0 \quad \mbox{and}\quad p_a^0=0.\]
We first prove that $H_T^{reg} \le 0$. In order to do so, we now
distinguish three subcases.

Assume first $p_2^0<p_1^0$.  Then we can assume that $A_0>A^*$
(otherwise we have $A_0=A^*$ and we fall into the first case).  Then
we deduce, as in the proof of Proposition \ref{pro:ht}, that
$H_T^{reg}\le 0$.

Assume now that $p_2^0\ge p_1^0$ and $a=1$.  We deduce that
$0=p_1^0\le p_2^0$. But because $H_2$ is minimal at $p_2^0$, we have
$0\in \partial H_2(p_2^0)$, and we deduce that $0\le p_2^0\in \partial
L_2(0)$.  This implies that $L_2\ge L_2(0) = -H_2(p_2^0)\ge 0$ on
$\R^+$. By definition of $H_T^{reg}$, this implies that $H_T^{reg} \le
0$.

Assume finally that $p_2^0\ge p_1^0$ and $a=2$.  This subcase is
symmetric with respect to the previous one.  We deduce that
$0=p_2^0\ge p_1^0$. But because $H_1$ is minimal at $p_1^0$, we deduce
that $0\ge p_1^0\in \partial L_1(0)$.  This implies that $L_1\ge
L_1(0) = -H_1(p_1^0)\ge 0$ on $\R^-$. Again, by definition of
$A_I^-$, this implies that $A_I^- \le 0$.

We now prove that $H_T^{reg} \ge 0$.  To do so pick some $(0,v_2,1)
\in {\mathcal A}_0^{reg}$ when $a = 1$ and some $(v_1,0,0) \in
    {\mathcal A}_0^{reg}$ when $a =2$. Hence, in both cases, we get
\[ H_T^{reg} \ge -L_a (0) = 0 .\]
Hence $H_T=0$ in the second case too.  The proof is now complete.
\end{proof}

\section{Third application: a homogenization result for a network}
\label{s.h}

In this section, we present an application of the comparison principle
of viscosity sub- and super-solutions on networks. 

\subsection{A homogenization problem}

We consider the simplest periodic network generated by $\eps \Z^d$. It
is in fact a lattice.  Hence, the network (or lattice) is naturally embedded in
$\R^d$. Let us be more precise now. At scale $\eps =1$, the edges are
the following subsets of $\R^d$: for $k,l \in \Z^d$, $|k-l|=1$,
\[ e_{k,l} = \{ \theta k + (1-\theta) l : \theta \in [0,1]\}. \] 
If $(e_1,\dots,e_d)$ denotes the canonical basis of $\R^d$, then for
$l = k + e_i$, $e_{k,l}$ is oriented in the direction of $e_i$.
The network $\mathcal{N}_\eps$ at scale $\eps >0$ is the one
corresponding to 
\[ \begin{cases}
 \mathcal{E}_\eps = \{ \eps e_{k,l}, k,l \in \Z^d, |k-l|=1 \}\\
 \mathcal{V}_\eps = \eps \Z^d \\
 \end{cases} \]
endowed with the metric induced by the Euclidian norm.
We next consider the following ``oscillating'' Hamilton-Jacobi equation
on this network
\begin{equation}\label{eq:hj-eps}
\left\{\begin{array}{ll}
u^\eps_t + H_{\frac{e}\eps} (u^\eps_x) = 0, & t >0, 
\; x \in e^*, e \in \mathcal{E}_\eps, \\
u^\eps_t + F_A(\frac{x}\eps,u^\eps_x) =0, 
& t >0, \; x \in \mathcal{V}_\eps 
\end{array}\right.
\end{equation}
(for some $A \in \R$) subject to the initial condition
\begin{equation}\label{eq:ic-eps}
u^\eps(0,x) = u_0 (x), \qquad x \in \mathcal{N}_\eps.
\end{equation}
\begin{rem}
In this section, we choose the simplest periodic homogenization problem
but much more can be done. For instance, the cell can be larger or have
a different shape, Hamiltonians can depend on $x$ etc. 
\end{rem}

For $m \in  \Z^d$, it is convenient to define
\[ \eps e_{k,l} + \eps m = \eps e_{k+m,l+m}.\]

\paragraph{Assumptions on $H$ for the homogenization problem}

For  each $e \in \mathcal{N}_1$, we associate a Hamiltonian
$H_e$ and we assume
\begin{itemize}
\item \textbf{(H'0)} (Continuity) For all $e \in \mathcal{E}_1$, $H_e
  \in C(\R)$.
\item \textbf{(H'1)} (Coercivity) $e \in \mathcal{E}_1$,
\[\liminf_{|q|\to +\infty} H_e(q)=+\infty.\]
\item \textbf{(H'2)} (Quasi-convexity) For all $e \in \mathcal{E}_1$,
  there exists a $p^0_e \in \R$ such that
\[\begin{cases}
H_e \quad \text{is nonincreasing on}\quad (-\infty,p^0_e],\\
H_e \quad \text{is nondecreasing on}\quad [p^0_e,+\infty).
\end{cases}\]
\item \textbf{(H'3)} (Periodicity) For all $m \in  \Z^d$, $H_{e+m}
  (p) = H_e (p)$. 
\end{itemize}

\paragraph{A homogenization result}

The goal of this section is to prove the following convergence result
for the oscillating Hamilton-Jacobi equation.
\begin{theo}[Homogenization of a network]\label{thm:conv}
Assume \emph{(H'0)-(H'3)}. Let $u_0$ be Lipschitz continuous and
$u^\eps$ be the solution of \eqref{eq:hj-eps}-\eqref{eq:ic-eps}.
There exists a continuous function $\bar{H} : \R^d \to \R$ such that
$u^\eps$ converges locally uniformly towards the unique solution $u^0$
of
\begin{eqnarray}
 u^0_t + \bar H (\nabla_x u^0) = 0, & &  t>0, x \in \R^d \label{eq:hj-h}\\
 u^0(0,x) = u_0 (x), & & x \in \R^d. \label{eq:ic-h}
\end{eqnarray}
\end{theo}
\begin{rem}
The meaning of the convergence $u^\eps$ towards $u^0$ is
\[ \lim_{\stackrel{(s,y) \to (t,x)}{y \in \mathcal{N}_\eps} } u^\eps (s,y) = u^0(t,x).\]
\end{rem}

\subsection{The cell problem}

Keeping in mind the definitions of networks and derivatives of
functions defined on networks, solving the cell problem consists in
finding specific global solutions of \eqref{eq:hj-eps} for $\eps=1$,
\textit{i.e.}
\begin{equation}\label{eq:hj-1}
\left\{\begin{array}{ll}
w_t + H_e (w_y) = 0, & t \in \R, 
\; y \in e^*, e \in \mathcal{E}_1, \\
w_t + F_A(y,w_y) =0, 
& t \in \R, \; y \in \mathcal{V}_1. 
\end{array}\right.
\end{equation}
Precisely, for some $P \in \R^d$, we look for solutions
$w(t,y)=\lambda t + P \cdot y + v(y)$ with a $\Z^d$-periodic function $v$; in
other words, we look for $(\lambda,v)$ such that
\begin{equation}\label{eq:cell}
\left\{\begin{array}{ll}
\lambda + H_e ((P\cdot y+v)_y) = 0, &  
\; y \in e^*, e \in \mathcal{E}_1, \\
\lambda + F_A(y,(P\cdot y+ v)_y) =0, 
&  y \in \mathcal{V}_1. 
\end{array}\right.
\end{equation}
\begin{theo}\label{thm:ergo}
  For all $P \in \R^d$ there exists a unique $\lambda \in \R$ for
  which there exists a $\Z^d$-periodic solution $v$ of
  \eqref{eq:cell}. Moreover, the function $\bar H$ which maps $P$ to
  $-\lambda$ is continuous.
\end{theo}
\begin{proof}
We consider the following $\Z^d$-periodic stationary problem
\begin{equation} \label{eq:alpha}
\left\{\begin{array}{ll}
\alpha v^\alpha + H_e ((P\cdot y+v^\alpha)_y) = 0, &  
\; y \in e^*, e \in \mathcal{E}_1, \\
\alpha v^\alpha + F_A(y,(P\cdot y+ v^\alpha)_y) =0, 
&  y  \in \mathcal{V}_1.
\end{array}\right.
\end{equation}
We consider 
\[ C = \max_{e \in \mathcal{E}_1} |H_e ((P \cdot y)_y)|.\]
Then the existence result and the comparison principle for the
stationary equation (see Appendix~\ref{s.stat}) imply that there
exists a (unique) $\Z^d$-periodic solution $v^\alpha$ of
\eqref{eq:alpha} such that
\[ |\alpha v^\alpha | \le C. \]
Since $H_e$ is coercive, this implies that there exists a constant
$\tilde{C}$ such that for all $\alpha >0$, $v_\alpha$ is
Lipschitz-continuous and
\[ |v^\alpha_y | \le \tilde{C};\]
in other words, the family $(v^\alpha)_{\alpha >0}$ is equi-Lipschitz continuous.
We then consider 
\[ \tilde{v}_\alpha = v_\alpha - v_\alpha (0).  \]
By Arzel\`a-Ascoli theorem, there exists $\alpha_n \to 0$ such that
$\tilde{v}^n := \tilde{v}_{\alpha_n}$ converges uniformly towards
$v$. Moreover, we can also assume that
\[ \alpha_n v_{\alpha_n} (0) \to \lambda.\]
Passing to the limit into the equation yields that $(\lambda,v)$
solves the cell problem~\eqref{eq:cell}.

The continuity of $\lambda$ is completely classical too. Consider $P_n
\to P_\infty$ as $n \to \infty$ and consider $(\lambda_n,v_n)$ solving
\eqref{eq:cell}. We proved above that
\[ |\lambda_n| \le C.\]
Hence, arguing as above, we can extract a subsequence from
$(\lambda_n,v_n)$ converging towards $(\lambda_\infty,v_\infty)$.
Passing to the limit into the equation implies that
$(\lambda_\infty,v_\infty)$ solves the cell
problem~\eqref{eq:cell}. The uniqueness of $\lambda$ yields the
continuity of $\bar H$. The proof is now complete.
\end{proof}

\subsection{Proof of convergence}

Before proving the convergence, we state without proof the following
elementary lemma. 
\begin{lem}[Barriers]
There exists $C>0$ such that for all $\eps>0$, 
\[ |u^\eps (t,x) - u_0(x) | \le Ct .\]
\end{lem}
We can now turn to the proof of convergence. 
\begin{proof}[Proof of Theorem~\ref{thm:conv}]
We classically consider the relaxed semi-limits 
\[\begin{cases}
 \overline u (t,x) = \limsup_{\stackrel{\eps \to 0,  
 (s,y) \to (t,x)}{y \in \mathcal{N}_\eps}} u^\eps (s,y),  \\
 \underline u (t,x) = \liminf_{\stackrel{\eps \to 0,  
 (s,y) \to (t,x)}{y \in \mathcal{N}_\eps}} u^\eps (s,y).  
 \end{cases}
 \]
 In order to prove convergence of $u^\eps$ towards $u^0$, it is enough
 to prove that $\overline{u}$ is a sub-solution of \eqref{eq:hj-h} and
 $\underline{u}$ is a super-solution of \eqref{eq:hj-h}. We only prove
 that $\overline{u}$ is a sub-solution since the proof for
 $\underline{u}$ is very similar.

We consider a test function $\varphi$ touching (strictly)
$\overline{u}$ from above at $(t_0,x_0)$: there exists $r_0>0$ such
that for all $(t,x) \in B_{r_0} (t_0,x_0)$, $(t,x) \neq (t_0,x_0)$,
\[ \varphi (t,x) > \overline{u} (t,x) \]
and $\varphi(t_0,x_0) = \overline{u}(t_0,x_0)$.  We argue by
contradiction by assuming that there exists $\theta >0$ such that
\begin{equation}\label{eq:1}
  \partial_t \varphi(t_0,x_0) - \lambda = \partial_t \varphi(t_0,x_0) + \bar H (\nabla_x \varphi (t_0,x_0)) = \theta >0. 
 \end{equation}
 We then consider the following ``perturbed test'' function
 $\varphi^\eps \colon \R^+ \times \mathcal{N}_\eps \to \R$
 \cite{evans},
\[ \varphi^\eps (t,x) = \varphi (t,x)  + \eps v (\eps^{-1}x)\]
where $(\lambda,v)$ solves the cell problem~\eqref{eq:cell} for $P= \nabla_x \varphi (t_0,x_0)$. 
\begin{lem}
  For $r \le r_0$ small enough, the function $\varphi^\eps$ is a
  super-solution of \eqref{eq:hj-eps} in
  $B((t_0,x_0),r) \subset (0,T) \times \mathcal{N}_\eps$ and
  $\varphi^\eps \ge u^\eps + \eta_r$ in $\partial B( (t_0,x_0),r)$ for
  some $\eta_r >0$.
\end{lem}
\begin{proof}
  Consider a test function $\psi$ touching $\varphi^\eps$ from below
  at $(t,x) \in ]0,+\infty[ \times \mathcal{N}_\eps$.  Then the
  function
\[ \psi_\eps (s,y) = \eps^{-1} (\psi (s,\eps y) - \varphi (s,\eps y))\]
touches $v$ from below at $y=\frac{x}\eps \in e$. In particular, 
\begin{align}
\label{eq:2}
 \partial_t \psi (t,x) = \partial_t \varphi (t,x), \\ \lambda +
 H_{\mathcal{N}_1} (y,\varphi_x(t_0,x_0) + \psi_x (t,x) -\varphi_x
 (t,x) ) \ge 0.
\label{eq:3}
\end{align}
Combine now \eqref{eq:1}, \eqref{eq:2} and \eqref{eq:3} and get
\[ \partial_t \psi (t,x) + H_{\mathcal{N}_1}(y,\psi_x (t,x)) \ge  \theta + E \]
where 
\[ E = 
 ( \varphi_t (t,x) -\varphi_t (t_0,x_0)) + (H_{\mathcal{N}_1}
(y,\psi_x (t,x)) - H_{\mathcal{N}_1}(y,\psi_x (t,x) +
\varphi_x(t_0,x_0)-\varphi_x (t,x))).\] The fact that $\varphi$ is
$C^1$ implies that we can choose $r >0$ small enough so that for all
$(t,x) \in B((t_0,x_0),r)$,
\[ E \ge -\theta.\]
Moreover, since $\varphi$ is strictly above $\overline{u}$, we
conclude that $\varphi^\eps \ge u^\eps +\eta_r$ on $\partial
B((t_0,x_0),r)$ for some $\eta_r >0$.  This achieves the proof of the
lemma.
\end{proof}
From the lemma, we deduce thanks to the (localized) comparison principle that 
\[ \varphi^\eps (t,x) \ge u^\eps (t,x)+\eta_r.\]
In particular, this implies 
\[ u(t_0,x_0)=\varphi (t_0,x_0) \ge u(t_0,x_0) + \eta_r>u(t_0,x_0)\]
which is the desired contradiction.
\end{proof}

\subsection{Characterization of the effective Hamiltonian}

We remark that, in view of {\bf (H'3)}, there are exactly $d$
  different Hamiltonians $H_1, \dots, H_d$ corresponding to
  $e_{0,b_i}$ where $(b_i)_i$ denotes the canonical basis of
  $\R^d$. With such a remark in hand, we can know give the explicit
  form of the effective Hamiltonian $\bar H$.

\begin{pro}[Characterization of the effective Hamiltonian]\label{pro::p1}
  Under assumptions of Theorem \ref{thm:conv}, for all $P=(p_1,\dots,p_d)\in \R^d$,
\[\bar H(P)=\max(A,\ \max_{i=1,\dots,d} H_i(p_i)).\]
\end{pro}
\begin{proof}
Let $\bar \mu$ denote $\max(A,\ \max_{i=1,\dots,d} H_i(p_i))$ and
$\mu$ denote $\bar H(P)$. We prove successively that $\mu \le \bar
\mu$ and $\bar \mu \le \mu$. 

\paragraph{Step 1: bound from above.}
Consider the following  sub-solution of (\ref{eq:hj-1})
\[\bar w(t,y)=-\bar \mu t + P\cdot y.\]
By comparison with 
\[w(t,y)=-\mu t +  P\cdot y + v(y)\]
where the bounded corrector $v$ is a solution of (\ref{eq:cell}) with
$\lambda=-\mu$, we deduce  that
\[\bar H(P)=\mu \le \bar \mu\]
by letting $t \to + \infty$.

\paragraph{Step 2: bound from below.}
To deduce the reverse inequality, we first notice that the periodic
corrector $v$ is Lipschitz continuous (by coercivity of the
Hamiltonians), which implies
$$-\mu + H_e(p_e +v_y) =0 \quad \mbox{for a.e.}\quad y\in e\in {\mathcal E}_1.$$
If $H_e$ is convex, we deduce that
$$\int_0^1 \mu \ dy \ge H_e(\int_0^1 (p_e+v_y(y))\ dy)$$
which implies
\begin{equation}\label{eq::p2}
\mu \ge  H_e(p_e).
\end{equation}
When $H_e$ is only quasi-convex, we still get the same inequality, because for any
$\varepsilon>0$, we can find a Hamiltonian $\tilde{H}_e^\varepsilon$
such that $|\tilde{H}_e^\varepsilon -H_e|\le \varepsilon$ with
$\tilde{H}_e$ satisfying (\ref{eq::3}).  By Lemma~\ref{lem::1}, we
know that there exists a convex increasing function
$\beta_\varepsilon$ such that $\beta_\varepsilon\circ
\tilde{H}_e^\varepsilon$ is convex for all $e\in {\mathcal E}_1$,
which implies again
$$\beta_\varepsilon(\mu+\varepsilon)\ge \beta_\varepsilon\circ \tilde{H}_e^\varepsilon(p_e).$$
Composing by $\beta_\varepsilon^{-1}$ and letting $\varepsilon$ go to
zero, we recover (\ref{eq::p2}).

Let us now consider what happens at the junction point
$y=0$. Since $w(t,0)=v(t,0)-\mu t$, Theorem~\ref{th::gr6}
  implies
$$-\mu + A \le 0.$$
Together with (\ref{eq::p2}), this implies
$$\bar H(P)=\mu\ge \bar \mu.\qedhere$$
\end{proof}


\appendix

\section{Appendix: proofs of some technical results}\label{s.a}

\subsection{Technical results on a junction}

In order to prove Lemma~\ref{lem::3}, we need the following one. 
\begin{lem}[A priori control at the same time]\label{lem::4}
Let $T>0$ and let $u$ be a sub-solution and $w$ be a super-solution as
in Theorem~\ref{th::2}. Then there exists a constant $C_T>0$ such that
for all $t\in [0,T), x,y \in J$, we have
\begin{equation}\label{eq::28}
u(t,x)\le w(t,y) + C_T(1 + d(x,y)).
\end{equation}
\end{lem}
We first derive Lemma~\ref{lem::3} from Lemma~\ref{lem::4}. 
\begin{proof}[Proof of Lemma~\ref{lem::3}]
Let us fix some $\eps>0$ and let us consider the sub-solution
$u^-_\eps$ and super-solutions $u^+_\eps$ defined in
\eqref{eq::35}. Using \eqref{eq::34}, we see that we have for all
$(t,x),(s,y)\in [0,T)\times J$
\begin{equation}\label{eq::36}
u^+_\eps(t,x)-u^-_\eps(s,y)\le 2C_\eps T + 2\eps + L_\eps d(x,y)
\end{equation}
We first apply Lemma~\ref{lem::4} to control $u(t,x)-u^+_\eps(t,x)$,
and then apply Lemma~\ref{lem::4} to control $u^-_\eps(s,y)-w(s,y)$.
Finally we get the control on $u(t,x)-w(s,y)$, using \eqref{eq::36}.
\end{proof}
We now turn to the proof of Lemma~\ref{lem::4}.
\begin{proof} [Proof of Lemma~\ref{lem::4}]
Let us define
\[\varphi(x,y)=\sqrt{1+ d^2(x,y)}.\]
Then $\varphi\in C^1(J^2)$ and satisfies
\begin{equation}\label{eq::32}
|\varphi_x(x,y)|,\ |\varphi_y(x,y)|\le 1.
\end{equation}
For constants $C_1,C_2>0$ to be chosen, let us consider
\[M=\sup_{t\in [0,T),\ x,y\in J} \left(u(t,x)-w(t,y) -C_2t -C_1\varphi(x,y)\right).\]
The result follows if we show that $M$ is non-positive for $C_1$ and
$C_2$ large enough. Assume by contradiction that $M>0$ for any $C_1$
and $C_2$.  Then for $\eta,\alpha >0$ small enough, we have
$M_{\alpha,\eta}\ge M/2>0$ with
\begin{equation}\label{eq::30}
M_{\eta,\alpha}=\sup_{t\in [0,T),\ x,y\in J} \left(u(t,x)-w(t,y) -C_2t
  -C_1\varphi(x,y)-\frac{\eta}{T-t}-\alpha\frac{d^2(x_0,x)}{2}\right).
\end{equation}
From \eqref{eq::27}, we have 
\[u(t,x)-w(t,y) \le C_T(2 + d(0,x) + d(0,y))\]
which shows that the supremum in \eqref{eq::30} is reached at a point
$(t,x,y)$, assuming $C_1>C_T$. Moreover, we have (for $0<\alpha\le 1$)
\begin{equation}\label{eq::31}
\alpha d(0,x)\le C=C(C_T).
\end{equation}
From the uniform continuity of the initial data $u_0$, there exists a
constant $C_0>0$ such that
\[u_0(x)-u_0(y)\le C_0\varphi(x,y)\]
and therefore $t>0$, assuming $C_1>C_0$. Then the classical time
penalization (or doubling variable technique) implies the existence of
$a,b\in\R$ (that play the role of $u_t$ and $v_t$) such that we have
the following viscosity inequalities
\[\left\{\begin{array}{l}
\displaystyle a+H\left(x,C_1\varphi_x(x,y)+ \alpha d(x_0,x)\right)\le 0,\\
b + H(y, -C_1\varphi_y(x,y))\ge 0
\end{array}\right.\]
(using the shorthand notation \eqref{eq::17} and writing $\alpha
d(x_0,x)$ for $\alpha \left(d^2(x_0,x)/2\right)_x$ for the purposes of
notation) with $a-b=C_2+\eta (T-t)^{-2}$. Substracting these inequalities
yields
\begin{equation}\label{eq::33}
C_2+\frac{\eta}{(T-t)^2}\le H(y, -C_1\varphi_y(x,y))- H\left(x,C_1\varphi_x(x,y)+ \alpha d(0,x)\right).
\end{equation}
Using bounds \eqref{eq::32} and \eqref{eq::31}, this yields a
contradiction in \eqref{eq::33} for $C_2$ large enough.
\end{proof}

\subsection{Technical results on a network}

\subsubsection*{Proof of Lemma~\ref{lem::l146}}

\begin{proof}[Proof of Lemma~\ref{lem::l146}]
(H1) and (H2) imply the uniform boundedness of the $p_e^0(t,x)$,
  \textit{i.e.} \eqref{eq::l136bis}. We also notice that because of
  \eqref{eq::l136bis}, there exists a constant $C_0>0$ such that for
  all $t\in [0,T]$, $e\in {\mathcal E}$ and $n\in \partial e$,
\begin{equation}\label{eq::l141}
|H_e(t,n,p^0_e(t,n))| \le C_0
\end{equation}
from which \eqref{eq::l156} is easily derived.

We now turn to the proof of \eqref{eq::l139}. In view of the
definition of $F_A$ and (A2), (H5), we see that it is enough to prove
that for all for $n\in {\mathcal V}$, $t,s\in [0,T]$, $p=(p_e)_{e\in
  {\mathcal E}_n} \in \R^{\text{Card}\ {\mathcal E}_n}$, $x\in
{\mathcal V}$,
\begin{equation}\label{eq::l140}
A_n^0(t,p)-A_n^0(s,p)\le\tilde\omega_T\Big(|t-s|(1+\max(0,A_n^0(s,p)))\Big).
\end{equation}
where
\[A_n^0(t,p):=\max_{e\in {\mathcal E}_n^-} H_e^-(t,n,p_e) \ge A^0_n(t)\]
or 
\[A_n^0(t,p):=\max_{e\in {\mathcal E}_n^+} H_e^+(t,n,p_e) \ge A^0_n(t).\]
We only treat the first case, since the second case reduces to the
first one by a simple change of orientation of the network.

We have 
\[A_n^0(a,p)=  H_{e_a}^-(a,n,p_{e_a})\quad \text{for}\quad a = t,s.\]
Let us assume that we have (otherwise there is nothing to prove)
\[0\le I(t,s):=A_n^0(t,p)-A_n^0(s,p).\]
We also have
\[H_{e_s}^-(t,n,p_{e_s})\le A_n^0(t,p)=H_{e_t}^-(t,n,p_{e_t})\]
and
\[H_{e_t}^-(s,n,p_{e_t})\le A_n^0(s,p)=H_{e_s}^-(s,n,p_{e_s}).\]
We now distinguish three cases.

\paragraph{Case 1: $H_{e_t}^-(s,n,p_{e_t})<H_{e_t}(s,n,p_{e_t})$.}
We first note that
\begin{equation}\label{eq::l145}
0\le I(t,s)\le A_n^0(t,p)-A_n^0(s).
\end{equation}
Let us define 
\[\tau = \left\{\begin{array}{ll}
\inf \left\{\sigma\in [t,s],\quad H_{e_t}^-(\sigma,n,p_{e_t})
<H_{e_t}(\sigma,n,p_{e_t})\right\} & \quad \text{if}\quad t<s,\\
\sup \left\{\sigma\in [s,t],\quad H_{e_t}^-(\sigma,n,p_{e_t})
<H_{e_t}(\sigma,n,p_{e_t})\right\} & \quad \text{if}\quad t\ge s.
\end{array}\right.\]
Let us consider a optimizing sequence $\sigma_k\to \tau$ such that
\[H_{e_t}^-(\sigma_k,n,p_{e_t})<H_{e_t}(\sigma_k,n,p_{e_t}).\]
Then we have
\[H_{e_t}^-(\sigma_k,n,p_{e_t})=
H_{e_t}(\sigma_k,n,p_{e_t}^0(\sigma_k,n)) \le A_n^0(\sigma_k)\le
A_n^0(\sigma_k,p).\] Then passing to the limit $k\to +\infty$, we get
(by convergence of the minimum values of the Hamiltonians, even if the
map $\bar t\mapsto p^0_e(\bar t,n)$ is discontinuous)
\begin{equation}\label{eq::l143}
H_{e_t}^-(\tau,n,p_{e_t})= H_{e_t}(\tau,n,p_{e_t}^0(\tau,n))\le
A_n^0(\tau)\le A_n^0(\tau,p).
\end{equation}
  If $\tau=t$, then \eqref{eq::l143} implies that
$A_n^0(t,p) = A_n^0(t)$ (keeping in mind the definition of $p_{e_t}$).
\medskip

\noindent {\sc Subcase 1.1: $\tau\not= t$.} This shows that
\[H_{e_t}(\tau,n,p_{e_t})\le A_n^0(\tau) \quad \text{and}\quad 
H_{e_t}(t,n,p_{e_t})\ge A_n^0(t).\]
We now choose some $\bar \tau$ in between $t$ and $\tau$ such that
\[H_{e_t}(\bar \tau,n,p_{e_t})= A_n^0(\bar \tau)\]
and estimate, using \eqref{eq::l145} and \eqref{eq::l141} and (H5)-(H6),
\[\begin{array}{ll}
0\le I(t,s) & \le \left\{A_n^0(t,p)-H_{e_t}(\bar
\tau,n,p_{e_t})\right\} + \left\{A_n^0(\bar \tau)-A_n^0(s)\right\}\\
& \le \left\{H_{e_t}(t,n,p_{e_t})-H_{e_t}(\bar \tau,n,p_{e_t})\right\} 
+ \left\{A_n^0(\bar \tau)-A_n^0(s)\right\}\\
& \le  \bar \omega_T(|t-\bar \tau|(1+\max(A_n^0(\bar \tau),0)))
+ \bar \omega_T(|\bar \tau - s|)\\
& \le  \bar \omega_T(|t-s|(1+C_0))+ \bar \omega_T(|t- s|).
\end{array}\]
\medskip

\noindent {\sc Subcase 1.2: $\tau= t$.}
Then $A_n^0(t,p)=A_n^0(t)$. Using \eqref{eq::l145}, this gives directly
\[0\le I(t,s) \le A_n^0(t)-A_n^0(s)\le \bar \omega_T(|t-s|).\]

\paragraph{Case 2: $H_{e_t}^-(s,n,p_{e_t})=H_{e_t}(s,n,p_{e_t})$ 
and $H_{e_t}^-(t,n,p_{e_t})=H_{e_t}(t,n,p_{e_t})$.}

We have
\[\begin{array}{ll}
0\le I(t,s) & = H_{e_t}^-(t,n,p_{e_t}) - A_n^0(s,p)\\
& \le H_{e_t}^-(t,n,p_{e_t})-H_{e_t}^-(s,n,p_{e_t})\\
& =H_{e_t}(t,n,p_{e_t})-H_{e_t}(s,n,p_{e_t})\\
&\le \bar \omega_T(|t-s|(1+ \max(H_{e_t}(s,n,p_{e_t}),0)))\\
&\le \bar \omega_T(|t-s|(1+ \max(H_{e_t}^-(s,n,p_{e_t}),0)))\\
&\le \bar \omega_T(|t-s|(1+ \max(A_0^n(s,p),0))).
\end{array}\]

\paragraph{Case 3: $H_{e_t}^-(s,n,p_{e_t})=H_{e_t}(s,n,p_{e_t})$ 
and $H_{e_t}^-(t,n,p_{e_t})<H_{e_t}(t,n,p_{e_t})$.}

Then
\[p_{e_t}^0(t,n)< p_{e_t}\le p_{e_t}^0(s,n).\]
Because of \eqref{eq::l141} and the uniform bound on the Hamiltonians
for bounded gradients, (H2), we deduce that
\[H_{e_t}(s,n,p_{e_t})\le C_1\]
for some constant $C_1>0$ only depending on our
assumptions. Therefore, we have
\[\begin{array}{ll}
0\le I(t,s) & = H_{e_t}^-(t,n,p_{e_t}) - A^0_n(s,p)\\
& \le H_{e_t}^-(t,n,p_{e_t})-H_{e_t}^-(s,n,p_{e_t})\\
& < H_{e_t}(t,n,p_{e_t})-H_{e_t}(s,n,p_{e_t})\\
&\le \bar \omega_T(|t-s|(1+ C_1)).
\end{array}\]
The proof is now complete.
\end{proof}

\subsubsection*{Semi-concavity of the distance}

In order to prove Lemmas~\ref{lem::l3} and \ref{lem::l110}, we need the
following one. 
\begin{lem}[Semi-concavity of $\varphi$ and $d^2$]\label{lem::l20}
Let ${\mathcal N}$ be a network defined in \eqref{eq::N} with edges
$\mathcal E$ and vertices $\mathcal V$.  Let
\[\varphi(x,y)=\sqrt{1+ d^2(x,y)}\]
where $d$ is the distance function on the network ${\mathcal N}$.
Then $\varphi(x,\cdot)$ and $\varphi(\cdot,y)$ are $1$-Lipschitz for
all $x,y\in {\mathcal N}$.  Moreover $\varphi$ and $d^2$ are
semi-concave on $e_a\times e_b$ for all $e_a,e_b\in {\mathcal E}$.
\end{lem}
\begin{proof}
  The Lipschitz properties of $\varphi$ are trivial. Since $r \mapsto
  r^2$ and $r \mapsto \sqrt{1+r^2}$ are smooth increasing functions in
  $\R^+$, the result follows from the fact that the distance function
  $d$ itself is semi-concave; it is even the minimum of a finite
  number of smooth functions.

If $e_a=e_b$, then $d^2(x,y)=(x-y)^2$ which implies that $\varphi\in
C^1(e_a\times e_a)$.  Then we only consider the cases where $e_a\not=
e_b$.

\paragraph{Case 1: $e_a$ and $e_b$ isometric to $[0,+\infty)$.}
Then for $(x,y)\in e_a\times e_b$, we have
\[d(x,y)=x+ y+ d(e_a^0,e_b^0)\]
which implies that $\varphi\in C^1(e_a\times e_b)$.

\paragraph{Case 2: $e_a$ isometric to $[0,+\infty)$ and $e_b$
    isometric to $[0,l_b]$.}
Reversing the orientation of $e_b$ if necessary, we can assume that
\[d_0:=d(e_a^0,e_b^0)\le d(e_a^0,e_b^1)=:d_1\]
and then for $(x,y)\in e_a\times e_b$, we have
\[d(x,y)=x + \min(d_0 + y,d_1 + (l_b-y)) = \min (d_0 + x+y, d_1 +x+ (l_b-y)) .\]
Then $\varphi$ is the minimum of two $C^1$ functions, it is semi-concave.

\paragraph{Case 3: $e_a$ and $e_b$ isometric to $[0,l_a]$ and $[0,l_b]$.}
Changing the orientations of both $e_a$ and $e_b$ if necessary, we can
assume that
\[d(e_a^0,e_b^0)=\min_{i,j=0,1}d_{ij}\quad \text{with}\quad d_{ij}=d(e_a^i,e_b^j).\]
Therefore
\[d(x,y)=\min (d_{00} +x+y, d_{01} + x + (l_b-y), 
d_{10} + (l_a-x) + y, d_{11} + (l_a-x)+(l_b-y))\] and again $\varphi$ is the
minimum of four $C^1$ functions, it is therefore semi-concave.
\end{proof}


\subsubsection*{Proof of Lemma ~\ref{lem::l3}}

\begin{proof}[Proof of Lemma ~\ref{lem::l3}]
We first prove \eqref{eq::l29} for $t=s$ by adapting in a
straightforward way the proof of Lemma~\ref{lem::4}. The only difference
is that for any $e_a,e_b\in {\mathcal E}$, the function
\[\varphi(x,y)=\sqrt{1+d^2(x,y)}\]
may not be $C^1(e_a\times e_b)$. But Lemma~\ref{lem::l20} and
Remark~\ref{rem:sc-net} ensure that this is harmless.  The remaining
of the proof of Lemma~\ref{lem::4} is unchanged.  In particular the
uniform bound on the Hamiltonians for bounded gradients is used, see
(H2).

Now \eqref{eq::l29} is obtained for $t \neq s$ by following the proof
of Lemma~\ref{lem::3} and using the barriers given in the proof of
Theorem~\ref{th::l3}.
\end{proof}

\subsubsection*{Proof of Lemma~\ref{lem::l110}}

\begin{proof}[Proof of Lemma ~\ref{lem::l110}]
We do the proof for sub-solutions (the proof for super-solutions being
similar).  We consider the following barrier (similar to the ones in the proof
 of Theorem~\ref{th::l3})
\[u_\eps^+(t,x)=  u_0^\eps(x) + K_\eps t + \eps\]
with
\[|u_0^\eps - u_0|\le \eps \quad \text{and}\quad |(u_0^\eps)_x|\le L_\eps\]
and $K_\eps \ge C_\eps$ with $C_\eps$ given in \eqref{eq::l114}. 
It is enough to prove that for all $(t,x) \in [0,T) \times \mathcal{N}$,
\[  u (t,x ) \le u_\eps^+ (t,x) \]
for a suitable choice of $K_\eps \ge C_\eps$ in order to conclude. 
Indeed, this implies 
\[ u(t,x) \le u_0(x)+ f(t) \]
with
\[ f(t) =\min_{\eps >0} (K_\eps t + \eps) \]
which is non-negative, non-decreasing, concave and $f(0)=0$. 

 We consider for $0<\tau\le T$,
\[M=\sup_{(t,x)\in [0,\tau)\times {\mathcal N}} (u-u_\eps^+)(t,x)\]
and assume by contradiction that $M>0$. We know by Lemma~\ref{lem::l3}
that $M$ is finite.  Then for any $\alpha,\eta>0$ small enough, we
have $M_\alpha\ge M/2>0$ with
\[M_\alpha = \sup_{(t,x)\in [0,\tau)\times {\mathcal N}} \left\{u(t,x)
-u_\eps^+(t,x)- \frac{\eta}{\tau-t}-\alpha \psi(x)\right\}
.\]
(we recall that $\psi=d^2(x_0,\cdot)/2$). 
By the sublinearity of $u$ and $u^+_\eps$, we know that this
supremum is reached at some point $(t,x)$.  Moreover $t>0$ since
$u(0,x) \le u_0 (x) \le u_\eps^+ (0,x)$.  

This implies in particular that
\begin{align*}
0<M/2\le M_\alpha & = \displaystyle u(t,x)-u_\eps^+(t,x)-
\frac{\eta}{\tau-t}-\alpha \frac{d^2(x_0,x)}{2}\\ \\ & \le
C_T (1+ d(x_0,x))-u_0^\eps (x_0) + L_\eps d(x,x_0)-\alpha
\frac{d^2(x_0,x)}{2} \\
& \le C_T (1+ d(x_0,x)) + |u_0(x_0)| + \eps + L_\eps d(x,x_0) - \alpha 
\frac{d^2(x_0,x)}{2} \\
& \le R_\eps (1+ d(x_0,x)) - \alpha 
\frac{d^2(x_0,x)}{2} 
\end{align*}
with
\[R_\eps = C_T+ \max (L_\eps, |u_0(x_0)| + \eps).\]
Then $z=\alpha d(x_0,x)$ satisfies
\[\frac{z^2}{2}  \le R_{\eps} \alpha+  R_{\eps} z \le R_{\eps} \alpha+
R_{\eps}^2 + \frac{z^2}4\]
which implies that for $\alpha \le 1$,
\begin{equation}\label{eq::l113}
\alpha d(x_0,x)\le 2 \sqrt{R_\eps  + R_\eps^2}.
\end{equation}
Writing the sub-solution viscosity inequality, we get
\[K_\eps + H_{\mathcal N}(t,x,(u_0^\eps)_x(x) + \alpha \psi_x(x))\le 0\]
We get a contradiction for the choice
\begin{multline*}
K_\eps =1+ \\
\max \left(\sup_{t\in [0,T]}\sup_{n\in {\mathcal
    V}}| \max(A_n(t),A_n^0(t))|, \sup_{t\in[0,T]} \sup_{e\in {\mathcal
    E}} \; \sup_{x\in e} \; \sup_{|p_e|\le L_\eps + 2 \sqrt{R_\eps + R_\eps^2}}
|H_e(t,x,p_e)|\right).
\end{multline*}
\end{proof}

\section{Appendix: stationary results for networks} \label{s.stat}

This short section is devoted to the statement of an existence and
uniqueness result for the following stationary HJ equation posed on a
network $\mathcal{N}$ satisfying \eqref{eq:le-inf},
\begin{equation}\label{eq::l16-stat}
u + H_{\mathcal N}(x, u_x) =0 \quad \text{for all}\quad x \in {\mathcal N}.
\end{equation}

For each $e\in {\mathcal E}$, we consider a Hamiltonian $H_e: 
e\times \R \to \R$ satisfying
\begin{itemize}
\item \textbf{(H0-s)} (Continuity) $H_e \in C(e\times  \R)$.
\item \textbf{(H1-s)} (Uniform coercivity) 
\[\liminf_{|q|\to +\infty} H_e(x,q)=+\infty\]
uniformly with respect to  $x \in e$, $e \in \mathcal{E}$.
\item \textbf{(H2-s)} (Uniform bound on the Hamiltonians for bounded gradients) 
For all $L>0$, there exists $C_{L}>0$ such that
\[ \sup_{p\in [-L,L], x \in \mathcal{N} \setminus
  \mathcal{V}}  |H_{\mathcal{N}}(x,p)| \le C_{L}.\]
\item \textbf{(H3-s)} (Uniform modulus of continuity for bounded gradients)
For all $L>0$, there exists a modulus of continuity $\omega_{L}$
such that for all $|p|,|q|\le L$ and $x\in e\in
{\mathcal E}$, 
\[|H_e(x,p)-H_e(x,q)|\le \omega_{L}(|p-q|).\]
\item \textbf{(H4-s)} (Quasi-convexity) For all $n \in \mathcal{V}$,
  there exists a $p^0_e (n)$ such that
\[\begin{cases}
H_e(n,\cdot) \quad \text{is nonincreasing on}\quad (-\infty,p^0_e(n)],\\
H_e(n,\cdot) \quad \text{is nondecreasing on}\quad [p^0_e(n),+\infty).
\end{cases}\]
\end{itemize}
As far as flux limiters are concerned, the following assumptions will
be used.
\begin{itemize}
\item \textbf{(A1-s)} (Uniform bound on $A$)
There exists a constant $C>0$ such that for all $n\in {\mathcal V}$,
\[|A_n|\le C.\]
\end{itemize}

The following result is a straightforward adaptation of
Corollary~\ref{cor::l4}. Proofs are even simpler since the time dependance was
an issue when proving the comparison principle in the general case. 
\begin{theo}[Existence and uniqueness -- stationary case]\label{th::l3-stat}
 Assume \emph{(H0-s)-(H4-s)} and \emph{(A1-s)}. Then there exists a
 unique sublinear viscosity solution $u$ of \eqref{eq::l16-stat} in
 ${\mathcal N}$.
\end{theo}

\paragraph{Acknowledgements. }
This work was partially supported by the ANR-12-BS01-0008-01 HJnet
project. The authors thank Y. Achdou, B. Andreianov, L. Mazet,
N. Seguin and N. Tchou for enlighting discussions. They thank
J.~Guerand for pointing out a difficulty related to critical slopes in
the reduction of the set of test functions.  They also thank
G. Costeseque for his stimulating numerical simulations of
flux-limited solutions. In particular, they guided the authors towards
the explicit formula of the homogenized Hamiltonian from
Section~\ref{s.h}. The authors warmly thank the two referees for 
their attentive reading of this long paper; their valuable comments
allowed us to improve the previous version of this paper.


\bibliographystyle{plain}
\bibliography{jonction}

\end{document}